\documentclass[11pt,leqno]{amsart}

\usepackage{graphicx}
\graphicspath{ {./images/} }

\usepackage[dvipsnames]{xcolor}
\definecolor{b}{HTML}{4472c4}
\definecolor{o}{HTML}{ED7D31}
\definecolor{g}{HTML}{70ad47}

\usepackage{fullpage}
\usepackage{amsmath,amsthm,amssymb,latexsym,color}
\usepackage[hidelinks,unicode,linktocpage=true]{hyperref}
\usepackage{lipsum}
\usepackage{enumitem,bbm}
\usepackage{hyperref}
\hypersetup{
    colorlinks=true,
    linkcolor=blue,
    filecolor=blue,
    citecolor=blue,
    urlcolor=cyan}

\usepackage{tikz}
\usepackage{subfigure}

\usepackage{multirow}

\newtheorem{theorem}{Theorem}[section]

\newtheorem{lemma}[theorem]{Lemma}
\newtheorem{corollary}[theorem]{Corollary}
\newtheorem{definition}[theorem]{Definition}
\newtheorem{proposition}[theorem]{Proposition}

\newtheorem{claim}[theorem]{Claim}

\theoremstyle{remark}
\newtheorem{remark}[theorem]{Remark}

\theoremstyle{plain}

\newcommand{\R}{\mathbb{R}}

\newcommand{\cM}{\mathcal{M}}

\def\dist{{\rm dist}}
\def\diam{{\rm diam}}

\def\dist{{\rm dist}}

\def\ave{{\rm ave}}
\def\med{{\rm med}}

\newcommand{\eps}{\epsilon}

\newcommand{\seminorm}[1]{{\left\vert\kern-0.25ex\left\vert\kern-0.25ex\left\vert #1
    \right\vert\kern-0.25ex\right\vert\kern-0.25ex\right\vert}}

\makeatletter
\renewcommand{\part}{\@startsection{part}{0}%
  \z@{5ex plus 1ex minus 1ex}{3ex plus 1ex}%
  {\centering\normalfont\Large\bfseries}}
\makeatother

\begin{document}

\title[Metric Poincar\'e inequalities]{Metric Poincar\'e inequalities for graphs}

\author{Dylan J. Altschuler, Pandelis Dodos, Konstantin Tikhomirov and Konstantinos Tyros}

\address{Department of Mathematical Sciences, Carnegie Mellon University}
\email{daltschu@andrew.cmu.edu}

\address{Department of Mathematics, University of Athens, Panepistimiopolis 157 84, Athens, Greece}
\email{pdodos@math.uoa.gr}

\address{Department of Mathematical Sciences, Carnegie Mellon University}
\email{ktikhomi@andrew.cmu.edu}

\address{Department of Mathematics, University of Athens, Panepistimiopolis 157 84, Athens, Greece}
\email{ktyros@math.uoa.gr}

\thanks{2020 \textit{Mathematics Subject Classification}: 05C12, 05C48, 05C50, 05C80, 46B85.}
\thanks{\textit{Key words}: expander graphs, random regular graphs, nonlinear Poincar\'{e} inequalities, metric embeddings.}


\begin{abstract}

\smallskip
This article obtains purely metric counterparts of cornerstone results in the theory of embedding graphs into normed spaces. Our first main result is a metric analogue of Matou\v{s}ek's extrapolation relating the Poincar\'e constants $\gamma(G,\varrho^p)$ and $\gamma(G,\varrho^q)$ for any exponents $0 < p,q < \infty$, any bounded-degree expander graph $G$, and any target metric space $\mathcal{M}=(M,\varrho)$. Our second main result provides a sharp estimate of the Poincar\'e constant $\gamma(G,\varrho)$ in terms of the cardinalities of the vertex set of $G$ and the metric space $\cM=(M,\varrho)$, in the setting of {\it random} graphs. This yields optimal estimates on the minimum cardinality of (bi-Lipschitz) universal metric spaces for graphs, finally establishing a nonlinear analogue of Matou\v{s}ek's celebrated ``incompressibility" theorem (1996). Further, we obtain estimates on the nonlinear spectral gap of metric snowflakes and sharp lower bounds on the distortion of random regular graphs into arbitrary metric spaces. Our proofs develop new nonlinear techniques, including random compression methods and a novel structural dichotomy for metric embeddings.
\end{abstract}

\maketitle

\tableofcontents

\numberwithin{equation}{section}

\section{Introduction} \label{sec1}

Let $d\geqslant 3$ be an integer, and let $G=(V_G,E_G)$ be a $d$-regular graph. We set $n:=|V_G|$, and write the eigenvalues of the adjacency matrix $A_G$ of $G$ in increasing order as
\begin{equation} \label{eigen-e1}
\lambda_n(G)\leqslant \cdots \leqslant \lambda_2(G)\leqslant \lambda_1(G)=d.
\end{equation}
The \emph{spectral gap} of $G$, given by $d-\lambda_2(G)$, measures the expansion of a graph \cite{HLW06}. This quantity is of fundamental importance across multiple areas of pure and applied mathematics. It is a standard fact that the ratio $\frac{d}{d-\lambda_2(G)}$ (essentially, the reciprocal of the spectral gap) is the smallest constant $\gamma\in (0,\infty]$ such that for any function $f\colon V_G\to \mathbb{R}^k$,
\begin{equation} \label{eq-poincare-e1}
\frac{1}{|V_G|^2} \sum_{v,u\in V_G} \big\|f(v)-f(u)\big\|_2^2 \leqslant
\frac{\gamma}{|E_G|} \sum_{\{v,u\}\in E_G} \big\|f(v)-f(u)\big\|_2^2.
\end{equation}
The significance of this \emph{discrete  Poincar\'e inequality} is evidenced by decades of widespread use in graph theory, algorithm design, metric embeddings, and related areas.

A fruitful research direction---with origins in metric geometry, functional analysis, and geometric group theory---is to replace the Euclidean norm in \eqref{eq-poincare-e1} with other norms and, even more generally, arbitrary metrics. The result is a so-called {\it nonlinear} Poincar\'e inequality which serves as the main tool for establishing nonembedding results in metric geometry and geometric group theory. Early appearances of such inequalities can be found in the works of Enflo~\cite{E76}, Bourgain--Milman--Wolfson \cite{BMW86}, Gromov \cite{Gr83,Gr03}, Pisier~\cite{Pi86}, Linial--London--Rabinovich \cite{LLR95}, and Matou\v{s}ek \cite{Ma97}. We refer the reader to the introduction of Mendel--Naor \cite{MN14} and the Bourbaki lectures of Eskenazis~\cite{Es22} for a detailed overview of this rich area and its numerous remarkable applications in metric geometry and algorithm design.

The following definition---which first appeared in \cite{MN14}---systematizes and unifies the aforementioned nonlinear Poincar\'{e} inequalities.

\begin{definition}[Nonlinear spectral gap]\label{def:main}
Let $d\geqslant 3$ be an integer, and let $G=(V_G,E_G)$ be a $d\text{-regular}$ graph. Also let $\mathcal{M}=(M,\varrho)$ be a metric space, and let $p> 0$ be an exponent. Denote by $\gamma(G,\varrho^p)$ the smallest constant $\gamma\in [0,\infty]$ such that for any $f\colon V_G \to M$,
\begin{equation} \label{eq-poincare}
\frac{1}{|V_G|^2} \sum_{v,u\in V_G} \varrho\big(f(v),f(u)\big)^p \leqslant
\frac{\gamma}{|E_G|} \sum_{\{v,u\}\in E_G} \varrho\big(f(v),f(u)\big)^p.
\end{equation}
\end{definition}

A graph $G$ is said to be an \textit{expander} with respect to $\cM$ if $\gamma(G,\varrho) = O_{\mathcal{M}}(1)$. In the specialized setting where the target space $\cM$ is a normed space $X$, there exists a substantial body of classical and recent results that estimate the Poincar\'e constant $\gamma(G,\|\cdot\|_X)$ under various assumptions on~$X$. We highlight the seminal works of Lafforgue \cite{La08} and Mendel--Naor \cite{MN14} which provide families of regular graphs that are expanders with respect to every Banach space of non-trivial type, and every super-reflexive space, respectively. We also refer the reader to \cite{Es22,Opp23,dLdS23,ADTT24} and the references therein for further results in the normed-space setting.

In the purely metric setting, nonlinear spectral gaps were studied by Mendel--Naor \cite{MN15} who considered certain
Hadamard spaces---metric spaces of non-positive curvature---as target spaces, as well as in the recent work of the
authors \cite{ADTT25}, where the basic problem of estimating the Poincar\'{e} constant between random regular graphs
was resolved. Despite this progress, much of the general theory for nonlinear spectral gaps, beyond when the target
space possesses a vector space structure, remains to be developed.

The purpose of this article is to obtain \emph{purely metric counterparts} to core pieces of the theory of spectral gaps for normed spaces, and to present related applications. A principal reason for studying the fully general setting comes from the {\it Ribe program} within metric geometry, which aims to develop metric analogues of normed space invariants. Other concrete and important motivations come from theoretical computer science and geometric group theory, which frequently consider graph embeddings into metric spaces, such as groups or other graphs, without normed structure. Among the applications obtained in this work are optimal estimates on the cardinality of {\it universal metric spaces} for graphs embeddings, which provide a natural counterpart to both classical and recent universality results across multiple branches of mathematics.

\subsection{Matou\v{s}ek's extrapolation for metric spaces}

The first part of this paper is devoted to the proof of a comparison result for nonlinear spectral gaps with different exponents. Such questions were first investigated by Matou\v{s}ek \cite{Ma97}, who completely resolved the case where $\cM$ is $\ell_p$ for $1\leqslant p < \infty$. Because of their utility, vector-valued versions of such comparison results---commonly referred to as \emph{Matou\v{s}ek extrapolation arguments}---have been studied by several authors, including Bartal--Linial--Mendel--Naor \cite{BLMN05} and Naor--Silberman \cite{NS11}, culminating in the work of $\text{de Laat--de la Salle}$~\cite{dLdS21} that treats all normed spaces.

Let us highlight two of the many reasons for studying extrapolation principles. First, the Poincar\'{e} constant can have a particular geometric or combinatorial interpretation for a certain exponent that enables direct and sharp computations. A canonical example is $\gamma(G,\|\cdot\|_2^2)$ which has a fundamental spectral interpretation. Second, and more importantly, for some exponents there may be additional analytic and probabilistic tools available (e.g., martingale inequalities) that may not be available for all exponents of interest.

Despite the successful history of extrapolation principles for normed spaces, extending these results for general metric spaces has remained elusive. Analogues of Matou\v{s}ek's extrapolation within the specialized setting of Hadamard spaces have been recently studied by Eskenazis--Mendel--Naor \cite{EMN25a}, where they explicitly ask \cite[Question 1.18]{EMN25a} whether a version of Matou\v{s}ek's extrapolation holds under curvature assumptions. Our first main result resolves the extrapolation problem for all bounded-degree expander graphs and \emph{all} metric spaces.

\begin{theorem}[Extrapolation for metric spaces] \label{thm-extrapolation}
Let $d\geqslant 3$ be an integer, let $G$ be a $d$-regular graph with Cheeger constant\footnote{The Cheeger constant is defined in Subsection~\ref{sec3-Cheeger}; our asymptotic notation is defined in Subsection \ref{sec3-asymptotic}.} $h(G)>0$, and let $\mathcal{M}=(M,\varrho)$ be an arbitrary metric space. Then, for any $0 < p \leqslant q<\infty$, we have
\begin{equation} \label{eq-extra-intro-new-e1}
\gamma(G,\varrho^p) \leqslant O_{d,h(G),p,q} \Big(\max\big\{1,\gamma(G,\varrho^q)\big\}\Big)
\end{equation}
and
\begin{equation} \label{eq-extra-intro-new-e2}
\gamma(G,\varrho^q) \leqslant O_{d,h(G),p,q}\Big( \gamma(G,\varrho^p)^{\frac{q}{p}}\Big),
\end{equation}
where the Poincar\'e constants $\gamma(G,\varrho^p)$, $\gamma(G,\varrho^q)$ are defined according to Definition~\ref{def:main}.
\end{theorem}

\begin{remark}
The exponents appearing in the right-hand sides of \eqref{eq-extra-intro-new-e1} and \eqref{eq-extra-intro-new-e2} coincide with the best known exponents---see, e.g., \cite[Remark 47]{Na21} or \cite[Proposition 1.4]{Es22}---for Matou\v{s}ek's extrapolation in the vector-valued setting. It is an interesting question---and also important in applications; see, e.g., \cite[Section 5.3]{Na21}---whether the dependence of the implicit constants on~$d$ and~$h(G)$, which does not appear in the vector-valued setting, is necessary when considering embeddings into arbitrary metric spaces.
\end{remark}

\begin{remark}
Explicit estimates for the implied constants in \eqref{eq-extra-intro-new-e1} and \eqref{eq-extra-intro-new-e2} are given in Theorem~\ref{thm-extrapolation-quantitative}, which provides a quantitative restatement of Theorem \ref{thm-extrapolation} in the regime $1\leqslant p \leqslant q<\infty$.
\end{remark}

\subsection{Metric Poincar\'{e} inequalities for random graphs}

The second part of this paper focuses on estimating $\gamma(G,\varrho)$ for random regular graphs $G$ and arbitrary metric spaces $\cM = (M,\varrho)$. In addition to the applications presented in this paper, we also refer to Naor's ICM lecture \cite{Na18} as well as \cite{MN15,Es22,ADTT25,EMN25a,EMN25b} for an account of applications of metric Poincar\'{e} inequalities.

For any $d$-regular graph $G$ and any metric space $\mathcal{M}=(M,\varrho)$, combining Bourgain's celebrated embedding theorem \cite{Bou85} with Matou\v{s}ek's extrapolation \cite{Ma97} yields
\begin{equation} \label{intro-e1}
\gamma(G,\varrho)\lesssim \log\big(|V_G|\big) \cdot \frac{d}{d-\lambda_2(G)}.
\end{equation}
Apart from this basic estimate, up to now, no other general estimate for nonlinear spectral gaps is known.

As with extrapolation results, substantially more is known about Poincar\'e inequalities for normed spaces than for general metric spaces. In particular, a recent result\footnote{We note that \eqref{intro-e2} appeared in \cite[Corollary 4]{Na21}, and as a special case of a more general result \cite[Theorem~25]{Na21} concerning nonlinear spectral gaps of complex interpolation spaces; we refer the reader to \cite[Section~5]{Na18} for an interpolation-free proof, as well as to \cite[Theorem 1.3]{Es22} for a streamlined exposition of this fundamental estimate.} of Naor \cite{Na21} asserts that for any $d$-regular graph $G$ and any finite-dimensional normed space $X$,
\begin{equation} \label{intro-e2}
\gamma(G,\|\cdot\|_X)\lesssim \log\big(\dim(X)+1\big) \cdot \frac{d}{d-\lambda_2(G)}.
\end{equation}
By \eqref{intro-e1} and \eqref{intro-e2}, one then obtains the optimal estimate
\begin{equation} \label{intro-e3}
\gamma(G,\|\cdot\|_X)\lesssim \min\Big\{\log\big(|V_G|\big), \log\big(\dim(X)+1\big)\Big\} \cdot \frac{d}{d-\lambda_2(G)},
\end{equation}
which neatly combines the dimensionality of both the domain and the target space.

The extent to which \eqref{intro-e3} is valid for general target metric spaces is a fundamental open problem in metric geometry, and in particular within the Ribe program \cite{Na18}. A principal difficulty is that preexisting methods operate under strong assumptions on the target space, such as curvature or linear space structure. A second serious obstruction is that a naive nonlinear analogue of \eqref{intro-e3} cannot hold because $\mathcal{M}$ can be chosen adversarially depending on $G$; e.g., by taking $\mathcal{M}$ equal to a deterministic bounded-degree connected graph\footnote{Graphs are viewed as metric spaces equipped with the shortest-path distance; see Subsection \ref{sec3-graphs} for more~details.} $G$ we obtain that $\gamma(G,\dist_G)\gtrsim \log\big(|V_G|\big)$, i.e., no improvement over \eqref{intro-e1} is available.

Our second main contribution is the following sharp asymptotic characterization of nonlinear spectral gaps for general metric spaces when $G$ is \textit{random}. In particular, the second mentioned obstacle turns out to be {\it the only obstruction} to efficiently controlling the nonlinear Poincar\'e constant of graph embeddings.  For the remainder of this paper, for any pair $n> d\geqslant 3$ of integers, $G(n,d)$ denotes the set of all $d$-regular graphs on $[n]:=\{1,\dots,n\}$, and $\mathbb{P}_{G(n,d)}$ denotes the uniform probability measure on~$G(n,d)$. We will always assume $dn$ is even.

\begin{theorem}[Poincar\'{e} inequality for metric spaces] \label{random-poincare}
For every integer $d\geqslant 3$, there exist constants $C_d\geqslant 1$ and $\tau>0$, depending only on the degree $d$, such that for any finite metric space $\mathcal{M}=(M,\varrho)$ with at least three points, setting $N:=|M|$, we have, for every integer $n> d$ with $dn$ even,
\begin{equation} \label{random-poincare-intro-e1}
\mathbb{P}_{G(n,d)}\Big[G\colon \gamma(G,\varrho)\leqslant C_d \min\big\{\log n, \, \log\log N\big\} \Big]
\geqslant 1- O_d\Big(\frac{1}{n^{\tau}}\Big),
\end{equation}
where $\gamma(G,\varrho)$ is given by Definition~\ref{def:main}.
\end{theorem}

The proof of this theorem requires developing new methods suitable for the setting of general metric spaces; an extensive discussion is provided in Section \ref{sec2-overview}.

\begin{remark}[Optimality]
The estimate on the Poincar\'e constant in \eqref{random-poincare-intro-e1} is optimal for all values of the parameters $n,N$. This is discussed in Section~\ref{sec-optimality}. We also note that the probabilistic tail estimate in \eqref{random-poincare-intro-e1} is independent of $N$.
\end{remark}

\begin{remark}[Normed space setting, revisited]
Theorem~\ref{random-poincare} recovers the optimal estimate \eqref{intro-e3} for a normed space $X$ with $\mathrm{dim}(X)\gtrsim \log n$ when restricted to the setting of \textit{random} bounded-degree graphs, despite being a fully metric result. Indeed, for a sufficiently large constant $C>0$, let $\mathcal{M}=(M,\|\cdot\|_X)$ be an $n^{-C}$-net in the unit ball of $X$ with cardinality $|M|\leqslant (5\,n^C)^{\mathrm{dim}(X)}$. By \eqref{random-poincare-intro-e1}, with high probability, every map $f\colon [n] \to M$ satisfies estimate \eqref{eq-poincare} with a constant $\gamma=O\big(\min\big\{\log n, \log\log\big(|M|\big)\big\}\big) =O\big(\min\big\{\log n, \log\big(\mathrm{dim}(X)\big)\big\}\big)$. It is then easy to verify, by a standard approximation argument, that \eqref{eq-poincare} also holds for {\it all} maps $f\colon [n]\to X$, implying \eqref{intro-e3}.
\end{remark}

\subsection{Applications to metric embeddings: bi-Lipschitz distortion}

Our first main application is an estimate of the bi-Lipschitz distortion of embedding random graphs into arbitrary target metric spaces, yielding an optimal characterization for a wide range of parameters. Estimates on the spectral gap of metric snowflakes are also given.

\subsubsection{Embeddings of random regular graphs into metric spaces}

Recall that for two metric spaces $\mathcal{M}=(M,\varrho_{\mathcal{M}})$ and $\mathcal{N}=(N,\varrho_{\mathcal{N}})$, the \emph{$($bi-Lipschitz$)$ distortion of $\mathcal{M}$ into $\mathcal{N}$}, denoted by $c_{\mathcal{N}}(\mathcal{M})$, is defined as the smallest constant $D>0$ for which there~exists a map $f\colon M\to N$ and a scaling factor $s>0$ such that $s \varrho_{\mathcal{M}}(i,j)\leqslant \varrho_{\mathcal{N}}\big(f(i),f(j)\big)\leqslant s D \varrho_{\mathcal{M}}(i,j)$ for all $i,j\in M$. Theorem \ref{random-poincare} yields the following estimate of the bi-Lipschitz distortion of random regular graphs into any given finite metric space.

\begin{corollary} \label{intro-cor1}
Let $d\geqslant 3$ be an integer, and let $\mathcal{M}= (M,\varrho)$ be an arbitrary finite metric space with $N:=|M|\geqslant 3$ points. Then, for every integer $n>d$ with $dn$ even,
\begin{equation} \label{intro-e4}
\mathbb{P}_{G(n,d)}\bigg[ G\colon c_{\mathcal{M}}(G) \gtrsim_d  \frac{\log n}{\min\{\log n, \log \log N\}}\bigg]
\geqslant 1- O_d\Big(\frac{1}{n^{\tau}}\Big),
\end{equation}
where $c_{\mathcal{M}}(G)$ denotes the bi-Lipschitz distortion of\, $(G,\dist_G)$ into $\mathcal{M}$, and $\tau=\tau(d)>0$ is as in Theorem \ref{random-poincare}.
\end{corollary}

The proof of this corollary is a standard argument detailed in Section \ref{sec12}.

\begin{remark}
The lower bound on $c_\cM(G)$ given in \eqref{intro-e4} is optimal for a wide range of the parameters $n,N$, in~particular, for $N\geqslant n^{\log^c n}$ for any constant $c>2$; see Proposition \ref{prop-optimality} for details. Characterizing the optimal distortion for the remaining range of parameters, including $N=n^{O(1)}$, is a compelling open question.
\end{remark}

\begin{remark}[Average distortion]\label{rem:average distortion}
The estimate \eqref{intro-e4} actually holds in the stronger context of lower bounding the optimal average\footnote{The \emph{average distortion} of a function $f\colon V_G \to M$ is given by the ratio between $\sum_{v,u \in V_G} \varrho\big(f(v),f(u)\big)$ and $\sum_{v,u \in V_G} \dist_G(v,u)$; see Subsection \ref{sec3-graphs} for the definition of the graph distance $\dist_G$.} distortion of $G$ into $\cM$. This is a universal feature of distortion lower bounds that are proven via nonlinear spectral gap estimates.
\end{remark}

\subsubsection{Metric snowflakes}

Next, we consider an application of our extrapolation result to the study of metric snowflakes. The $\eps$-snowflake of a metric $\cM = (M,\varrho)$ is given by $\cM_\eps=(M,\varrho^{1-\eps})$, where $\eps \in (0,1)$. The snowflake of a space often has better embeddability properties than the original space. The canonical example of this is the seminal embedding theorem of Assouad \cite{As83}, which roughly asserts that snowflakes of doubling metric spaces admit low-distortion embeddings into finite dimensional Euclidean space. Following Assouad's theorem, snowflakes have become an important construction in metric geometry, acting as a metric notion of regularization; we refer to the survey \cite{Es22} for a comprehensive discussion of applications.

As snowflakes and nonlinear spectral gaps are both core tools in the theory of metric embeddings, it is natural to study their interplay.
\medskip

\begin{center}
\textit{Given a graph $G$, an arbitrary metric space $\cM = (M,\varrho)$, and $\eps \in (0,1)$,\\ how is the nonlinear spectral gap $\gamma(G,\varrho)$ related to $\gamma(G,\varrho^{1-\eps})$?}
\end{center}
\medskip

As an immediate consequence of Theorem \ref{thm-extrapolation}, we obtain the following answer.

\begin{corollary}[Extrapolation for snowflakes]
Let $\cM = (M,\varrho)$ be a metric space, let $\eps \in (0,1)$, let $d\geqslant 3$ be an integer, and let $G$ be a $d$-regular expander graph. There exist positive constants $c$ and $C$, that depend only on the degree $d$, the Cheeger constant $h(G)$ of\, $G$ and $\eps$, such that
\begin{equation}
c\, \gamma(G,\varrho^{1-\eps}) \leqslant \max\{1,\gamma(G,\varrho)\} \ \ \  \text{ and } \ \ \  \gamma(G,\varrho) \leqslant C\,\gamma(G,\varrho^{1-\eps})^{\frac{1}{1-\eps}}.
\end{equation}
In particular, for any fixed $\eps \in (0,1)$, a bounded-degree expander graph $G$ satisfies a dimension-free metric Poincar\'e inequality with respect to $\cM$ if and only if $G$ satisfies a dimension-free inequality for the $\eps$-snowflake of $\cM$.
\end{corollary}

\subsection{Applications to metric embeddings: universal spaces}

Our second main application is a sharp estimate on the cardinality of (bi-Lipschitz) universal metric spaces, establishing a purely metric counterpart to Matou\v{s}ek's celebrated ``incompressibility" theorem \cite{Ma96}.

The search for universal structures (spaces containing all substructures from a given class) is a fundamental problem within many branches of mathematics. Classical examples include the Whitney and Nash embedding theorems in differential geometry and the Menger--Nöbeling theorem in topology. In combinatorics, recent breakthroughs of Alon--Nenadov \cite{AN19} and Alon \cite{Al17} identify the minimum cardinality universal graphs which contain all subgraphs of a given size.

In the setting of metric embeddings, the  classical Fr\'{e}chet theorem asserts that every $n$-point metric space embeds isometrically into $\ell_\infty^k$ for $k = n-1$. A fundamental problem originating in the seminal work of Johnson--Lindenstrauss \cite{JL84} asks about the minimum dimension universal space with respect to embeddings with bi-Lipschitz distortion at most $D$. For brevity, such embeddings are said to be \emph{$D$-embeddings}.

Upper bounds were obtained by Johnson--Lindenstrauss--Schechtman \cite{JLS87} and, in a stronger form, by Matou\v{s}ek~\cite{Ma92} who proved that, as long as $1\leqslant D\leqslant \frac{\log n}{\log\log n}$, any $n$-point metric space $D$-embeds into $\ell^k_\infty$ with $k\lesssim n^{C/D}$, where $C>0$ is a universal constant. Matou\v{s}ek's celebrated ``incompressibility" theorem \cite{Ma96,Ma02} yields a matching lower bound, obtained via the following rigidity result for embeddings. For every $n$, there exists an $n$-point metric space~$\mathcal{M}$, constructed as a regular graph of large girth and small diameter, such that if $\mathcal{M}$ admits a $D$-embedding into a normed space $X$, then necessarily $\dim(X)\gtrsim n^{c/D}$, where $c>0$ is a universal constant. An alternative and more robust proof was given by Naor~\cite{Na21}, who used \eqref{intro-e2} to show that spectral expanders have the aforementioned rigidity. We refer the reader to \cite{Na18} for a detailed overview, and to \cite{AT25} for very recent progress in the regime $\Omega(\log n)\leqslant D\leqslant n$.

Our second main result, Theorem \ref{random-poincare}, enables us to obtain an answer to the following metric analogue of the classical line of questioning described above.
\medskip

\begin{center}
\textit{Given an integer $n\geqslant 2$ and a distortion parameter $D\geqslant 1$, what is the least \\ cardinality of a metric space  $\mathcal{M}$ such that every connected graph \\ on $n$ vertices embeds into $\mathcal{M}$ with distortion at most $D$?}
\end{center}
\medskip

We focus on embeddings of graphs partly because of the non-scale invariance of general metric spaces, but mainly because graphs are central in the study of bi-Lipschitz universality and its numerous applications (see, e.g., \cite{Es22}).

An upper bound of the cardinality in question can be obtained by Matou\v{s}ek's argument in \cite{Ma92} that we recall in Appendix \ref{appendix-A} for the reader's convenience: for any $n\geqslant 2$ and any $1\leqslant D\leqslant \frac{\log n}{\log\log n}$, there exists a metric space $\mathcal{M}=(M,\varrho)$ with
\begin{equation} \label{intro-e5}
|M| \lesssim \exp\big( n^{C/D} \big),
\end{equation}
where $C>0$ is a universal constant, so that every connected graph on $n$ vertices admits a $D\text{-embedding}$ into $\mathcal{M}$.

Theorem \ref{random-poincare} implies that in the above range for distortion $D$, the upper bound in \eqref{intro-e5} is optimal, up to universal constants in the power of $n$. Specifically, we obtain the following lower bound on the size of (bi-Lipschitz) universal metric spaces.

\begin{corollary} \label{intro-cor2}
There exists a universal constant $c>0$ with the following property. Let $n\geqslant 2$ be an integer, and let $D\geqslant 1$ be a distortion parameter. If $\mathcal{M}= (M,\varrho)$ is a finite metric space such that every (connected) graph on $n$ vertices embeds into $\mathcal{M}$ with distortion at most $D$, then
\begin{equation} \label{intro-e6}
\exp\big( n^{c/D} \big) \lesssim |M|.
\end{equation}
\end{corollary}

\begin{remark}[Random regular graphs are the hardest to embed]
From the perspective of trying to find low-distortion embeddings of graphs into \textit{small} metric spaces, \eqref{intro-e5} and (the proof of) Corollary \ref{intro-cor2} together assert that random regular graphs are among the hardest to embed out of all graphs. Namely, for a fixed parameter $1\leqslant D\leqslant \frac{\log n}{\log\log n}$, if a metric space $\cM'=(M',\varrho')$ is such that a random $d$-regular graph on $n$ vertices $D$-embeds into $\cM'$ with high probability, then $\log\log\big(|M'|\big)$ is at least of the same order as $\log\log\big(|M|\big)$, where $\cM = (M,\varrho)$ is the universal metric space given by Matou\v{s}ek's argument.
\end{remark}


\section{Outline of the proofs: key ideas and technical overview} \label{sec2-overview}

In this section, we present the technical challenges and the main new ideas of this article. For brevity,~\eqref{eq-poincare} will be referred to as a ``$p$-Poincar\'e inequality for $f$'' and the right-hand side of \eqref{eq-poincare} will be referred to as the ``$p$-Dirichlet form"\!. We also set
\[  \gamma(G,\varrho^q;f) := \frac{\ave_{\mathcal{M}}(f,q)}{\frac{1}{|E_{G}|} \sum_{\{v,u\} \in E_{G}} \varrho(f(v),f(u))^q }, \]
where
\[ \ave_{\cM}(f,q) := \frac{1}{|V_{G}|^2} \sum_{v,u \in V_{G}} \varrho\big(f(v),f(u)\big)^q. \]
Given $f$, the ``$\varrho$-length'' of an edge $\{v,u\}$ refers to the distance $\varrho\big(f(v),f(u)\big)$. We will always consider a fixed metric space $\cM = (M,\varrho)$ with cardinality $N := |M|$. The cardinality of the vertex set of $G$ will be denoted by $n$. Our notation in the following sketches largely mirrors that of the actual proofs, but minor modifications are made to some definitions for simplicity.

\subsection{A structural dichotomy} \label{subsec2.1}

A core observation enabling both our extrapolation and metric Poincar\'e results is the following structural dichotomy for functions $f\colon [n] \to M$. Such functions will be classified as either \textit{concentrated} or \textit{non-concentrated}. Non-concentrated functions turn out to always satisfy $\gamma(G,\varrho^q;f) = O(1)$, for any $q\geqslant 1$ and any bounded-degree expander graph $G$. On the other hand, concentrated functions exhibit structure that is amenable to analysis.

A function $f\colon [n] \to M$ is said to be $q\text{-\textit{concentrated}}$ if the mean is not much bigger than the median for the random variable
\[ [n]^2 \ni (v,u) \mapsto \varrho\big(f(v),f(u)\big)^q, \]
where the probabilistic quantities are computed with respect to the uniform  distribution of pairs $(v,u)\in [n]^2$. The precise definition is given in Definition \ref{def3.2} in the main text.

Roughly, the image of a $q$-concentrated function resembles a uniform metric space (where all pairwise distances are equal), which lends itself to analysis based on a careful path counting and multiscale partitioning. On the other hand, for functions which are \emph{not} $q$-concentrated, one can extract special pairs of far-apart clusters within the target metric space. Applying the standard Euclidean Poincar\'e inequality to these clusters yields a \emph{dimension-free} metric Poincar\'e inequality.

\subsection{Extrapolation} \label{subsec2.2}

We now sketch the proof of our extrapolation result, Theorem \ref{thm-extrapolation} (the full proof is given in Sections~\ref{sec-empirical}--\ref{sec-non-concentrated}). Recall that our goal is to compare $\gamma(G,\varrho^p)$ and $\gamma(G,\varrho^q)$, where $G$ is a bounded-degree expander graph and $0<p\leqslant q<\infty$. A simple argument enables us to restrict to $p\geqslant 1$.

An upper bound on $\gamma(G,\varrho^q;f)$ in terms of $\gamma(G,\varrho^p;f)$, when $f$ is $q\text{-concentrated}$, is obtained by a relatively elementary combination of H\"{o}lder's and Markov's inequalities with the definition of $\gamma(G,\varrho^p)$. In this setting, we obtain a stronger \textit{function-wise} comparison.

The remaining cases of interest---extrapolation of power $q$ to power $p$ for concentrated $f$, and two-sided extrapolation for non-concentrated $f$---are considered below.

\subsubsection{Extrapolation of power $q$ to power $p$ for concentrated $f$} \label{subsubsec2.2.1}

Let $f$ be $p$-concentrated. We seek to upper bound $\gamma(G,\varrho^p;f)$ in terms of $\gamma(G,\varrho^q)$, where $1\leqslant p \leqslant q$. For shorthand, set $\Gamma := C \,\gamma(G,\varrho^q)$, where $C$ is a sufficiently large positive constant.

Consider first the case when, for some small constant $c_0 > 0$,
\[ \Gamma \cdot \frac{1}{|E_G|} \sum_{\substack{\{v,u\}\in E_G \\
\varrho(f(v),f(u))\leqslant c_{0}\,\med(f)}}
\varrho\big(f(v),f(u)\big)^q \geqslant \med(f)^q. \]
That is, the contribution of edges with small $\varrho$-length in the $q$-Dirichlet form for $f$ is significant. In this case, by the assumed $p$-concentration of $f$,
\begin{align*}
\frac{\Gamma}{|E_G|} \sum_{\{v,u\}\in E_G }
\frac{\varrho\big(f(v),f(u)\big)^p}{\ave_{\cM}(f,p)}
& \gtrsim_{c_0,p} \frac{\Gamma}{|E_G|}
\sum_{\substack{\{v,u\}\in E_G \\ \varrho(f(v),f(u))\leqslant c_{0}\, \med(f)}}
\bigg(\frac{\varrho\big(f(v),f(u)\big)}{c_0\,\med(f)}\bigg)^p \\
& \gtrsim_{c_0,p,q} \frac{\Gamma}{|E_G|} \sum_{\substack{\{v,u\}\in E_G \\
\varrho(f(v),f(u))\leqslant c_{0}\, \med(f)}}
\bigg(\frac{\varrho\big(f(v),f(u)\big)}{c_0\,\med(f)}\bigg)^q \gtrsim_{c_0, q} 1,
\end{align*}
where the last inequality utilizes the definition of $\gamma(G,\varrho^q)$. The desired $p$-Poincar\'e inequality for $f$ follows.

Next, consider the case that the dominant contribution to the $p$-Dirichlet form for $f$ comes from edges of ``exceptional'' $\varrho$-length. For every $\xi>0$, define $\mathcal{E}_\xi \subseteq E_G$ by
\begin{equation*}
\mathcal{E}_\xi := \big\{ \{v,u\}\in E_G \colon \varrho\big(f(v),f(u)\big) \geqslant  \xi \cdot \ave_{\mathcal{M}}(f,p)^{1/p} \big\}.
\end{equation*}
If $|\mathcal{E}_\xi| \geqslant \Gamma^{-1}\,\xi^{-p}\,|E_G|$ for some $\xi > 0$, then the desired $p$-Poincar\'e inequality for $f$ also follows.

The main technical claim---which is enough to complete the proof of extrapolation from power $q$ to $p$---is that the two extreme cases considered above are exhaustive: there is no intermediate regime. Arguing by contradiction, the idea is to construct an auxiliary function $g$ that agrees with $f$ on a large set and violates the assumed $q$-Poincar\'e inequality\footnote{Note that since the assumed $q$-Poincar\'e inequality is applied to an auxiliary function $g$---different, possibly, than $f$---we do not obtain a function-wise comparison principle like in the case of extrapolation from power $p$ to $q$ mentioned earlier.}. The details of the construction of $g$ are given in Lemma \ref{lemma-new} in the main text.

\subsubsection{Non-concentrated $f$} \label{subsubsec2.2.2}

We now sketch the estimate $\gamma(G,\varrho^q;f) = O(1)$ for any $q\geqslant 1$, any non-concentrated $f$, and any bounded-degree expander graph $G$. A quantitative formulation is given in Proposition \ref{prop:non-concentrated}. We emphasize again that this is not a comparison inequality; it is a stronger, unconditional result.

The argument begins by using the definition of $\med(f)$ and the pigeonhole principle, to select some vertex $\tilde v$ so that $|B| \geqslant n/2$, where
\[ B := \big\{u \in [n] \colon \varrho\big(f(u),f(\tilde v)\big) \leqslant 2\,\med(f) \big\}. \]
Set $A:=B^\complement$. Since $f$ is not $q$-concentrated, we see that $\frac{1}{n^2}\sum_{v,u\in B} \varrho\big(f(v),f(u)\big)^q$ is much smaller than $\ave_{\cM}(f,q)$. Moreover, for any $u,v \in [n]$, we have $\varrho\big(f(v),f(u)\big) \leqslant \varrho\big(f(v),f(\tilde v)\big) + \varrho\big(f(\tilde{v}),f(u)\big)$. Thus, in total,
\[ \frac{1}{n}\sum_{v\in A} \varrho\big(f(v),f(\tilde v)\big)^q =\Theta\big(\ave_{\cM}(f,q)\big). \]
It remains to compare the $q$-Dirichlet form for $f$ with the left-hand side of the above equation. Towards this goal, define the dyadic decomposition of $A$ into increasingly exceptional sets of vertices,
\[  A_k:= \Big\{v\in [n] \colon  2^{k-1}\,\med(f) < \varrho\big(f(v),f(\tilde v)\big) \leqslant 2^k \,\med(f)\Big\}, \]
and notice that
\[ \ave_{\cM}(f,q) \lesssim \frac{1}{n} \sum_{v \in A} \varrho\big(f(v),f(\tilde v)\big)^q \leqslant
\frac{1}{n} \sum_{k = 1}^\infty 2^{kq}\, \med(f)^q \,|A_k|. \]
Thus, it suffices to show that for each $k \geqslant 1$, there is some $m_k \geqslant 1$ with
\begin{equation}\label{eq:overview nonconc}
\Big|\Big\{ \{v,u\} \in E_G\colon  \varrho\big(f(v),f(u)\big) \geqslant \frac{2^k\,\med(f)}{2^{m_k}}\Big\}\Big| ~\gtrsim ~3^{m_k\,q}\,|A_k|;
\end{equation}
that is, there are edges whose contribution to the $q$-Dirichlet form for $f$ ``compensates'' for $A_k$'s contribution to $\ave_{\cM}(f,q)$. There can either be a few long compensating edges or many short ones;, and the parameter $m_k$ quantifies this trade-off. The existence of such an $m_k$---which ultimately relies on standard scalar expansion via Cheeger's inequality---is established in Lemma \ref{l:2.4}.

\subsection{Metric Poincar\'e inequalities} \label{subsec2.3}

In this subsection, we give an overview of the proof of our second main result, Theorem \ref{random-poincare}. Let $\boldsymbol{G}$ be a uniformly random $d$-regular graph on $n$ vertices, let $\cM = (M,\varrho)$ be a metric space with cardinality $N := |M|\geqslant 3$, and recall that we need to show that, with high-probability, $\gamma(\boldsymbol{G},\varrho) \leqslant C\,\min\{\log\log N, \log n\}$, where the constant $C$ depends only on the degree $d$. The secondary $\log n$ bound follows from Bourgain's embedding theorem \cite{Bou85}. Thus, we can assume, without loss of generality, that $\log\log N = o(\log n)$, i.e., $\log(N) \leqslant n^{o(1)}$. In this \textit{moderate cardinality} regime, we must prove that, with high-probability, $\gamma(\boldsymbol{G},\varrho) \leqslant C\log\log N$. Note that, in view of the outlined comparison results---more precisely, Proposition \ref{prop:non-concentrated}---it suffices to show that, with high-probability, $\gamma(\boldsymbol{G},\varrho;f) \leqslant C\log\log N$ for functions $f\colon [n] \to M$ that are $1$-concentrated.

First, let us discuss the difficulties that one encounters when one attempts to apply a trivial union bound argument over all potential concentrated maps. Fix a $1$-concentrated map $f$, and consider randomly generating the edges of $\boldsymbol{G}$ one at a time. Since $f$ is $1$-concentrated, each edge $\{v,u\}$ of $\boldsymbol{G}$ will have a constant probability of satisfying $\varrho\big(f(v),f(u)\big) \gtrsim \ave_\cM(f,1)$. By standard concentration estimates,
\[  \mathbb{P}\big[ \gamma(\boldsymbol{G},\varrho;f) = O(1)\big] \geqslant 1 - \exp\big(-\Theta(n)\big). \]
The central obstruction is that the number of $1$-concentrated functions is, for some choices of $\cM$, of order $\Theta(N)^n$, preventing a union bound over all $f$. This is a fundamental information-theoretic barrier: there are only order $\exp(c_d\,n\log n)$ $d$-regular graphs on $n$ vertices, which is much smaller than the total number of mappings when $N=n^{\omega(1)}$.

We introduce a {\it compression method} aimed at  sidestepping this obstacle. The main idea is to restrict ourselves to accessing only a small amount of information about $f$, leading to a much more efficient union bound. The commensurate price is that, for most pairs of vertices $u$ and $v$, we are only able to estimate $\varrho\big(f(v),f(u)\big)$ up to some substantial approximation error, leading to the $\log\log N$ growth of the nonlinear spectral gap. As already noted in the introduction---and further detailed in Section \ref{sec-optimality}---this growth is necessary.

\subsubsection{Information about $f$} \label{subsubsec2.3.1}

Let us describe the information about $f$ that we will allow ourselves to ``query"\!. Our goal is to record data about $f$ with low complexity---i.e., the number of total possible realizations of the data is much smaller than $\Theta(N)^n$---thereby allowing us to apply certain union bound arguments.
\medskip

\noindent (1) \textit{Average}: we will query $\ave_{\cM}(f,1)$. The number of possible values of this average is bounded above by $(n+N)^N$, which is acceptable if $N \ll \sqrt{n}$. For the regime of $N\gg \sqrt{n}$, substantial new ideas will be needed; they will be discussed later on in this overview.
\medskip

\noindent (2) \textit{Values of $f$ on the ``seeds"}: instead of recording values of $f(v)$ for all $v \in [n]$, we will only query values on a particular $k$-element subset of vertices, that we call \textit{seeds}. Note that this has complexity $N^k$; for example, if $N$ is polynomial in $n$, the latter has order $e^{\Theta(k \log n)}$.

\subsubsection{Approximation scheme} \label{subsubssec3.3.2}

Let $m\geqslant 1$ be a parameter, and consider a $1$-concentrated function $f$ that satisfies
\begin{equation}\label{eq:intro contra}
\gamma(\boldsymbol{G},\varrho;f) = \omega(m).
\end{equation}
Let us investigate the effect of \eqref{eq:intro contra} on typical paths of length $m$. By $d$-regularity, each edge $e \in E_{\boldsymbol{G}}$ is contained in at most order $m(d-1)^{m-1}$ paths of length $m$. On the other hand, letting $\mathcal{P}_m$ denote the set of paths of length $m$, by the locally tree-like geometry of $\boldsymbol{G}$, we have $|\mathcal{P}_m|\gtrsim |E_{\boldsymbol{G}}|\, (d-1)^{m-1}$. An elementary double counting argument yields that
\[  \frac{m}{|E_{\boldsymbol{G}}|} \sum_{\{v,u\} \in E_{\boldsymbol{G}}} \varrho\big(f(v),f(u)\big) \gtrsim \frac{1}{|\mathcal{P}_m|} \sum_{P \in \mathcal{P}_m} \sum_{\{v,u\} \in P} \varrho\big(f(v),f(u)\big). \]
Thus, by Markov's inequality, \eqref{eq:intro contra} implies that typical paths $P \in \mathcal{P}_m$ have $\varrho$-length $o\big(\ave_{\cM}(f,1)\big)$.

We shall develop an approximation scheme that proceeds by attempting to associate to each vertex $v$ another vertex $h(v)$ that belongs to a $k$-element set of vertices, the aforementioned ``seeds". This assignment is done so that $\dist_{\boldsymbol{G}}\big(v,h(v)\big) = m$, where the radius $m$ satisfies $m \asymp \log \log N$. We aim to approximate distances between pairs of vertices with the following properties.
\begin{itemize}
\item[(P1)] The seed $h(v)$ exists.
\item[(P2)] There is a unique path $P_{v,h(v)}$ in $\boldsymbol{G}$ of length $m$ from $v$ to $h(v)$.
\item[(P3)] The path $P_{v,h(v)}$ has typical length, i.e.,
\[ \sum_{\{v',u'\} \in P_{v,h(v)}} \varrho\big(f(v'),f(u')\big) = o\big(\ave_{\cM}(f,1)\big). \]
\item[(P4)] The seed $h(v)$ is not overused: we have $\big|h^{-1}\big(h(v)\big)\big| \leqslant \frac{(d-1)^m}{m}$.
\end{itemize}
Let us assume\footnote{This assumption is not far from what we can actually prove.}, for the sake of exposition, that, with probability $1 - e^{-\omega(k\log n)}$, a constant fraction of vertices satisfy (P1)--(P4). We also observe that, since $f$ is $1$-concentrated, with high probability, for most pairs $\{s,s'\}$ of seeds,
\[ \varrho\big(f(s),f(s')\big) = \Omega\big(\ave_{\cM}(f,1)\big). \]

Our subsequent goal is to show that since most vertices satisfy (P1)--(P4), a constant fraction of the edges in $E_{\boldsymbol{G}}$ should have vertices as endpoints that satisfy (P1)--(P4). However, \textit{the core technical issue} in our argument is that the seeds, the edge-set of $\boldsymbol{G}$, and the location of these good vertices are coupled. However, imagine, again for the sake of exposition, that we could argue that most edges have as endpoints vertices that satisfy (P1)--(P4). Then, for most $\{v,u\} \in E_{\boldsymbol{G}}$, we have
\[ \varrho\big(f(v),f(u)\big) \geqslant \varrho\big(f(h(v)),f(h(u))\big) - \varrho\big(f(v),f(h(v))\big) - \varrho\big(f(u),f(h(u))\big) = \Omega\big(\ave_{\cM}(f,1)\big). \]
This implies $\gamma(\boldsymbol{G},\varrho;f) = O(1)$, contradicting \eqref{eq:intro contra}. In summary, if all of the above assertions held with probability at least $1 - e^{-\omega(k \log n)}$, then the probability of \eqref{eq:intro contra} holding is at most $e^{-\omega(k\log n)}$, which is sufficient to beat the union bound and establish the desired Poincar\'e inequality.

\subsubsection{Multistage construction of the random graph} \label{subsubsec2.3.3}

The above sketch contains a number of serious issues, mostly centered around the fact that the seeds,  the edge-set of $\boldsymbol{G}$ and the good vertices are coupled. This motivates a multistage construction of $\boldsymbol{G}$, in which a significant portion of randomness is reserved for further technical analysis. The construction of $\boldsymbol{G}$ proceeds as follows. (See Section \ref{sec-models} for the precise definitions.)
\begin{enumerate}
\item[$\bullet$] Fix, once and for all, a representative of each isomorphism class of $d$-regular graphs on~$[n]$. Denote by $\boldsymbol{U}$ the random variable that assigns to $\boldsymbol{G}$ its representative.
\item[$\bullet$] Let $\boldsymbol{G}_{-\ell}$ denote the random graph obtained by deleting exactly $\ell$ many edges of $\boldsymbol{G}$ uniformly at random. Denote by $\boldsymbol{U}_{-\ell}$ the random graph obtained by applying the same procedure~to~$\boldsymbol{U}$.
\item[$\bullet$] Apply a uniform random permutation $\boldsymbol{\pi}$, independent of all other random variables, to label the vertices of $\boldsymbol{U}$. This process results in a uniformly random $d$-regular graph, which we denote by $\boldsymbol{H}$. We also denote by $\boldsymbol{H}_{-\ell}$ the random variable obtained by applying $\boldsymbol{\pi}$ to $\boldsymbol{U}_{-\ell}$.
\item[$\bullet$] The resulting pair  $(\boldsymbol{H},\boldsymbol{H}_{-\ell})$ is equal in distribution to the pair $(\boldsymbol{G},\boldsymbol{G}_{-\ell})$; see Lemma \ref{rm-l1}.
\end{enumerate}

\subsubsection{Main steps of the proof in the regime $N \ll \sqrt{n}$} \label{subsubsec2.3.4}

With this construction in hand, let us return to the analysis of $\gamma(\boldsymbol{G},\varrho;f)$, where $f$ is a fixed $1\text{-concentrated}$ function. Roughly, we will apply the described approximation scheme to $\boldsymbol{U}_{-\ell}$, and then judiciously use the randomness of $\boldsymbol{\pi}$ and $E_{\boldsymbol{U}}\setminus E_{\boldsymbol{U}_{-\ell}}$ to find many pairs of vertices $\{v,u\}$ that are ``far apart" in a sense that will be described below. When selecting the parameter $\ell$, there is a tradeoff: the larger $\ell$ is, the more randomness we have in reserve to power our concentration estimates. However, if $\ell$ is too large, then $\boldsymbol{U}_{-\ell}$ behaves quite differently from a $d$-regular graph, invalidating our analysis in the previous sketch.

We briefly summarize some of the key steps. Say that an event holds with \textit{overwhelming probability} if it holds with conditional probability $1-o(1)$, conditioned on the isomorphism class of \emph{any} ``locally tree-like'' realization of the isomorphism class $\boldsymbol{U}$. (The precise meaning of ``locally tree-like" is given in Proposition~\ref{prop-tree}.)
\medskip

\noindent \emph{Step 1: approximation}. The seeds will be the vertices $[k]$. The assigned seed $g_{-\ell}(v) \in [k]$ of a vertex $v$ is a $\boldsymbol{H}_{-\ell}\text{-measurable}$ random variable, obtained by taking the smallest (in the natural order) seed at distance exactly $m$ from $v$ in the graph $\boldsymbol{H}_{-\ell}$. The assigned seed $g(v)$ of a vertex $v$ in $\boldsymbol{H}$ is defined analogously.
\medskip

\noindent \emph{Step 2: typical vertices}. Assume that
\eqref{eq:intro contra} holds true. With overwhelming probability, there exists a random set $\mathcal{O}$ of vertices---that we shall refer to as \emph{typical}---of size proportional to $\ell$, with the following properties.
\begin{enumerate}
\item[(i)] The degree of a vertex $v\in\mathcal{O}$ in $\boldsymbol{H}_{-\ell}$ is exactly $d-1$.
\item[(ii)] For all vertices $v\in\mathcal{O}$, the seed exists and $g_{-\ell}(v) = g(v)$.
\item[(iii)] The set of deleted edges $E_{\boldsymbol{H}}\setminus E_{\boldsymbol{H}_{-\ell}}$ contains a perfect matching of $\mathcal{O}$.
\item[(iv)] For most pairs $\{v,u\}$ of vertices of $\mathcal{O}$, we have $\varrho\big(f(g(v)),f(g(u))\big) = \Theta\big(\ave_{\cM}(f,1)\big)$. On the other hand,
we have $\varrho\big(f(g(v)),f(g(u))\big) = o\big(\ave_{\cM}(f,1)\big)$ for most of the deleted edges $\{v,u\}\in E_{\boldsymbol{H}}\setminus E_{\boldsymbol{H}_{-\ell}}$.
\end{enumerate}
The existence of typical vertices is (essentially) achieved in Proposition \ref{prop-seeds-and-conc-final}; the proof is technically demanding and occupies the entire Section \ref{sec-prop-5.2-new}.
\medskip

\noindent \emph{Step 3: derandomization}. There is a subtle challenge when dealing with the measurability of the set $\mathcal{O}$ of typical vertices: it is \emph{not} $\boldsymbol{H}_{-\ell}$-measurable. This problem is fixed by derandomizing some aspects of $\mathcal{O}$ and $f$. As $\mathcal{O}$ is a subset of the $\boldsymbol{H}_{-\ell}\text{-measurable}$ set (of size at most $\ell$) of vertices satisfying (i) above, there are at most $2^\ell$ possibilities for $\mathcal{O}$. Thus, we can afford to condition on the relative position of $\mathcal{O}$ within the $\boldsymbol{H}_{-\ell}$-measurable set of vertices satisfying (i). We can also afford to condition on the values of $f$ on the seeds.
\medskip

\noindent \textit{Step 4: concluding.} The derandomization argument eventually allows us to condition on just $\boldsymbol{H}_{-\ell}$ while still having access to $\mathcal{O}$, with $|\mathcal{O}|$ proportional to~$\ell$. Moreover, we also obtain a large $\boldsymbol{H}_{-\ell}\text{-measurable}$ set of pairs $\mathcal{Q} \subseteq \mathcal{O} \times \mathcal{O}$ such that the set of deleted edges $E_{\boldsymbol{H}}\setminus E_{\boldsymbol{H}_{-\ell}}$ contains a perfect matching of $\mathcal{O}$, and yet most deleted edges do not belong to $\mathcal{Q}$. This (and recalling that $(\boldsymbol{G},\boldsymbol{G}_{-\ell})$ and $(\boldsymbol{H},\boldsymbol{H}_{-\ell})$ are equal in distribution) yields strong upper bounds on the probability that \eqref{eq:intro contra} holds for a $1$-concentrated $f$.
\medskip

In total, the following choices of parameters suffice:
\[ m  \asymp \log\log N, \ \ \ \ell  \asymp \frac{n}{m}, \ \ \ \text{  and } \ \ \  k\asymp \frac{n}{(d-1)^m} . \]
The multistage generation of $\boldsymbol{G}$ allows us to rigorously implement the previously sketched compression argument, thus completing Theorem \ref{random-poincare} within the regime $N = o(\sqrt{n})$.

\begin{remark}[Local properties]
The most technical aspect of our proof is the collection of overwhelming probability bounds on typical vertices. In this direction, the main observation is that most of these properties are \textit{local} and, consequently, if $v$ and $u$ are far apart, then the events $\mathbbm{1}_{[v \text{ is typical}]}$ and $\mathbbm{1}_{[u \text{ is typical}]}$ are independent. This fact, together with the locally tree-like behavior of random regular graphs, allows us to apply standard Chernoff bounds and prove the existence of typical vertices with the desired probability.
\end{remark}

\subsubsection{Well-conditioned and arbitrary metric spaces} \label{subsubsec2.3.5}

Towards dropping the assumption $N = o(\sqrt{n})$ on the cardinality of $\cM$, the critical part of the above argument that does not immediately transfer to the much broader regime of $n \geqslant \log(N)^{\omega(1)}$ turns out to be the union bound over $\ave_{\cM}(f,1)$. Towards replacing this union bound, we say a metric space is \textit{well-conditioned} if it has a sub-exponential aspect ratio, that is,
\[  \frac{\diam(\cM)}{\min\{ \varrho(x,y)\colon x,y \in M \text{ and } x\neq y\}} \leqslant e^{N}. \]
Then, the idea is to union bound over $r$, where $\ave_\cM(f,1) \in [2^r,\, 2^{r+1})\, \mathrm{diam}(\cM)$. As there are only $O(N+\log n)$ scales, a union bound now succeeds. Finally, the case that $N$ is extremely large, i.e., $N \geqslant e^{n^{\Omega(1)}}$ is handled easily using existing techniques.

So, it remains only to lift the restriction of being well-conditioned when $N = e^{n^{o(1)}}$. We will show that for an arbitrary metric space $\cM = (M,\varrho)$, one can always construct an auxiliary \text{well-conditioned} metric space $\cM' = (M',\varrho')$, with roughly the same cardinality, such that $\gamma(G,\varrho') \geqslant \frac{1}{2}\,\gamma(G,\varrho)$ for any connected graph $G$ on $n$ vertices. The idea is to create $N^2$ disjoint ``truncated" copies of $\cM$, indexed by the pairwise distances $\{\tau_1,\tau_2,\dots,\tau_{N^2}\}$ of elements in $\cM$, with each copy corresponding to a different ``scale'' $\tau$. Then, given a non-constant map $f\colon [n] \to M$, we define the relevant scale as
\[  \tau=\tau(f) := \max\Big\{ \varrho\big(f(v),f(u)\big)\colon v,u \in [n] \Big\}>0.  \]
Let $G$ be any connected graph on $n$ vertices; there is a path of length at most $n$ connecting the pair of vertices which saturate $\tau$, and thus $G$ contains some edge with $\varrho$-length of at least~$\tau/n$. So, the Dirichlet form is not affected to leading order if we were to truncate all distances less than $o(\tau/n^2)$, as there are order $n$ edges in the average defining the Dirichlet form. Similarly, the statistic $\ave_{\cM}(f,1)$ is not affected to leading order if we truncate distances less than $o(\tau/n^2)$, as there is at least one pair of vertices at distance $\tau$ and order $n^2$ total pairs. These observations are enough to show that $\gamma(G,\varrho') \geqslant \frac{1}{2}\,\gamma(G,\varrho)$. The details are given in Proposition \ref{prop-aspect-ratio} in the main text.

\section{Preliminaries} \label{sec3}

\subsection{Graphs} \label{sec3-graphs}

All graphs in this paper are finite and simple. For any graph $G$, by $V_G$ we denote its vertex set, and by $E_G$ we denote the set of its edges; moreover, for any subset $S$ of $V_G$, by $G[S]$ we denote the induced on $S$ subgraph of $G$.

\subsubsection{Graph distance}

If $G$ is a graph, then by $\dist_G(\cdot,\cdot)$ we denote the shortest-path distance on~$G$ between vertices; by convention, we set $\dist_G(v,u):=|V_G|$ if $v,u\in V_G$ are contained in different connected components of $G$.

If $v\in V_G$ and $S\subseteq V_G$ is nonempty, then define the shorthand
\[ \dist_G(v,S):= \min\big\{ \dist_G(v,u)\colon u\in S\big\}. \]
Moreover, for every integer $\ell\geqslant 0$, define
\begin{enumerate}
\item[$\bullet$] $B_G(S,\ell):=\big\{v\in V_G\colon \dist_G(v,S)\leqslant \ell\big\}$, and
\item[$\bullet$] $\partial B_G(S,\ell):=\big\{v\in V_G\colon \dist_G(v,S)=\ell\big\}$.
\end{enumerate}
(In particular, we have $B_G(S,0)=\partial B_G(S,0)=S$.) Finally, define the \emph{diameter} $\diam(G)$ of $G$ to be the quantity $\max\{\dist_G(v,u):v,u\in V_G\}$; namely, $\diam(G)$ is the diameter of the metric space~$(G,\dist_G)$.

\subsubsection{Regular graphs}

For any pair $n> d\geqslant 3$ of positive integers, $G(n,d)$ denotes the set of all $d$-regular graphs on $[n]:=\{1,\dots,n\}$, and $\mathbb{P}_{G(n,d)}$ denotes the uniform probability measure on $G(n,d)$. \emph{We will always assume that $dn$ is even.}

\subsection{Metric spaces}

If $\mathcal{M}=(M,\varrho)$ is a metric space, then for any nonempty subset $A$ of $M$,~set
\[ \mathrm{diam}_{\varrho}(A):= \max\big\{\varrho(x,y)\colon x,y\in A\big\};\]
notice that if $A=M$, then $\mathrm{diam}_{\varrho}(M)$ is just the diameter of the metric space $\mathcal{M}$, and we shall denote it simply by $\mathrm{\diam}(\mathcal{M})$.

\subsection{Embedding finite metric spaces into $\ell_1$}

Let $\mathcal{M}=(M,\varrho)$ be a metric space and let $c_{1}(\mathcal{M})$ denote the (bi-Lipschitz) distortion of $\mathcal{M}$ into $\ell_1$. Since $\ell_1$ is a vector space, the quantity $c_{1}(\mathcal{M})$ can also be defined as the smallest constant $D>0$ for which there exists a map $e\colon M\to \ell_1$ such that $\varrho(i,j)\leqslant \|e(i)-e(j)\|_{\ell_1}\leqslant D \varrho(i,j)$ for all $i,j\in M$. We will need the following classical embedding theorem due to Bourgain \cite{Bou85}.

\begin{theorem} \label{thm-bourgain}
There exists a universal constant $c_{\rm B}>0$ such that for any $n$-point metric space~$\mathcal{M}$ with $n\geqslant2$,
\begin{equation} \label{eq-bourgain}
c_1(\mathcal{M}) \leqslant c_{\rm B} \, \log n.
\end{equation}
\end{theorem}

\subsection{The Cheeger constant and spectral properties of regular graphs} \label{sec3-Cheeger}

Let $d\geqslant 3$ be an integer, and let $G$ be a $d$-regular graph. The \emph{Cheeger constant} of $G$ is defined by
\begin{equation} \label{eq-cheeger}
h(G):= \min_{\substack{\emptyset \neq S\subseteq V_G\\ |S|\leqslant \frac{|V_G|}{2}}}
\frac{\big|\big\{ \{v,u\}\in E_G:\; v\in S \text{ and } u\in S^{\complement}\big\}\big|}{|S|}.
\end{equation}
The Cheeger constant of $G$ has the following well--known relation to the spectral properties of the adjacency matrix of $G$. Denoting by $\lambda_n(G)\leqslant \cdots \leqslant \lambda_2(G) \leqslant \lambda_1(G)$ the eigenvalues of the adjacency matrix $A_G$ of $G$, it holds (see, e.g., \cite[Theorem 4.1]{HLW06}),
\begin{equation} \label{e-cheeger-e1}
\frac{d-\lambda_2(G)}{2} \leqslant h(G) \leqslant \sqrt{2d\big(d-\lambda_2(G)\big)}.
\end{equation}
The following consequence of Friedman's second eigenvalue theorem \cite{Fr08} will also be used.
\begin{corollary} \label{Friedman}
Let $n> d\geqslant 3$ be integers with $dn$ even. There exists $\tau_1=\tau_1(d)>0$ such that
\begin{equation} \label{friedman-e1}
\mathbb{P}_{G(n,d)}\Big[ \lambda_2(G) \leqslant 2.1 \sqrt{d-1} \Big] \geqslant 1- O_d\Big( \frac{1}{n^{\tau_1}}\Big);
\end{equation}
consequently, by \eqref{e-cheeger-e1},
\begin{equation} \label{friedman-e2}
\mathbb{P}_{G(n,d)}\big[ h(G) \geqslant 0.005d \big] \geqslant 1- O_d\Big( \frac{1}{n^{\tau_1}}\Big).
\end{equation}
\end{corollary}

\subsection{Locally tree-like behavior of random regular graphs}

We utilize a quantitative form of the well-known fact that random regular graphs are locally tree-like. For every integer $d\geqslant 3$, every $G\in G(n,d)$, and every nonnegative integer $m\leqslant \frac{1}{25} \log_{d-1} n$, set
\begin{equation} \label{eq-tree}
\mathcal{T}(G,m):=\big\{ v\in [n]\colon \text{the induced subgraph of $G$ on $B_G(v,3m)$ is a tree}\big\}.
\end{equation}
The following proposition follows from the results in \cite{MWW04} (see, e.g., \cite[Appendix~A.1]{Hu19} for details).

\begin{proposition} \label{prop-tree}
For every integers $n> d\geqslant 3$ with $dn$ even and every $m\in \big\{ 0,\dots,\frac{1}{25} \log_{d-1} n\big\}$,
\begin{equation} \label{eq2.3}
\mathbb{P}_{G(n,d)} \Big[ G\colon |\mathcal{T}(G,m)|\geqslant n-\sqrt{n}\Big]\geqslant 1-O_d\Big(\frac{1}{\sqrt{n}}\Big).
\end{equation}
\end{proposition}

\subsection{Chernoff bounds}

We recall the multiplicative formulation of Chernoff's classical deviation inequality for sums of independent Bernoulli random variables (see, e.g., \cite{AS16}).

\begin{lemma} \label{chernoff}
Let $X_1,\dots,X_n$ be independent Bernoulli random variables, and set $X:=X_1+\dots+X_n$ and $\mu:=\mathbb{E}[X]$. Then,
\begin{align}
\label{chernoff-e1} \mathbb{P}\big[X \geqslant (1+\delta)\mu\big] \leqslant \exp\Big(-\frac{\delta^2\mu}{2+\delta}\Big),
\ \ \ \ \ & \text{for all } \delta>0, \\
\label{chernoff-e2} \mathbb{P}\big[X \leqslant (1-\delta)\mu\big] \leqslant \exp\Big(-\frac{\delta^2\mu}{2}\Big),
\ \ \ \ \ \ \ & \text{for all } 0<\delta\leqslant 1.
\end{align}
\end{lemma}

\subsection{Asymptotic notation} \label{sec3-asymptotic}

If $a_1,\dots,a_k$ are parameters, then we write $O_{a_1,\dots,a_k}(X)$ to denote a quantity that is bounded in magnitude by $X C_{a_1,\dots,a_k}$, where $C_{a_1,\dots,a_k}$ is a positive constant that depends on the parameters $a_1,\dots,a_k$; we also write $Y\lesssim_{a_1,\dots,a_k}\!X$ or $X\gtrsim_{a_1,\dots,a_k}\!Y$ for the estimate $|Y|=O_{a_1,\dots,a_k}(X)$. Finally, we write $Y=\Theta_{a_1,\dots,a_k}(X)$ if $Y=O_{a_1,\dots,a_k}(X)$ and $X=O_{a_1,\dots,a_k}(Y)$.


\part{Matou\v{s}ek's extrapolation for metric spaces} \label{part1}

\section{Empirical statistics and concentrated functions} \label{sec-empirical}

We start by defining the statistics of an embedding.

\begin{definition}[Empirical statistics, and empirical averages] \label{def3.1}
Let $\mathcal{M}=(M,\varrho)$ be a metric space, let $S$ be a nonempty finite set, and let $f\colon S\to M$. Also let $0<\tau<1$. By $Q_\tau(f):= Q_\tau(\mathcal{M},f)$ we denote the \emph{$\tau$-th empirical quantile} of $\varrho\big(f(v),f(u)\big)$, where $v,u$ are independent uniformly random elements of~$S$; that~is,
\begin{equation} \label{extra-e1}
Q_\tau(f):= \inf\Big\{ t>0\colon
\big|\big\{(v,u)\in S\times S \colon \varrho\big(f(v), f(u)\big) \leqslant t \big\}\big| \geqslant \tau |S|^2 \Big\}.
\end{equation}
Moreover, for every $q\geqslant 1$, by $\ave_{\mathcal{M}}(f,q)$ we denote the \emph{empirical average} of $\varrho\big(f(v),f(u)\big)^q$ defined by setting
\begin{equation} \label{extra-e2}
\ave_{\mathcal{M}}(f,q) := \frac{1}{|S|^2} \sum_{v,u \in S} \varrho\big(f(v),f(u)\big)^q.
\end{equation}
\end{definition}

The following definition will be used in the central dichotomy that enables our extrapolation results, as well as the proof of Theorem \ref{random-poincare}; see Section \ref{sec2-overview} for a high-level overview.

\begin{definition}[Concentrated functions] \label{def3.2}
Let $\mathcal{M}=(M,\varrho)$ be a metric space, let $S$ be a nonempty finite set, and let $f\colon S\to M$. Also let $K>0$, $q\geqslant 1$ and $0<\tau<1$. We say that the function $f$ is \emph{$(K,q,\tau)$-concentrated} if
\begin{equation} \label{extre-e3}
\ave_{\mathcal{M}}(f,q) \leqslant K\cdot Q_\tau(f)^q,
\end{equation}
where $\ave_{\mathcal{M}}(f,q)$ is as in \eqref{extra-e2}, and $Q_\tau(f)$ is as in \eqref{extra-e1}.
\end{definition}

Part of the interest in this notion stems from the fact that non-concentrated functions
satisfy a nonlinear Poincar\'{e} inequality with a constant \emph{independent} of the target metric space $\mathcal{M}$. This is the content of the following proposition.

\begin{proposition}[Poincar\'{e} inequality for non-concentrated functions] \label{prop:non-concentrated}
Let $n> d\geqslant 3$ be integers, and let $G\in G(n,d)$ with Cheeger constant $h(G)>0$. Let $\mathcal{M}=(M,\varrho)$ be a metric space, and let $f\colon [n]\to M$. Also let $q\geqslant1$, $C_R\geqslant 5^q$, and $\tau\in(1/n,1)$, and assume that the function $f$ is not $(C_R,q,\tau)$-concentrated. Set
\begin{equation}\label{eq:2.01}
\ell := \bigg\lceil\frac{\max\big\{\log_2\frac{1}{2\tau },0\big\}}{\log_2\big(1+\frac{h(G)}{d}\big)} \bigg\rceil +
\bigg\lceil\frac{1}{\log_2\big(1+\frac{h(G)}{2^{2q+4}d}\big)} \bigg\rceil.
\end{equation}
Then,
\begin{equation} \label{eq:2.02}
\frac{1}{n^2}\sum_{v,u\in [n]} \varrho\big( f(v), f(u) \big)^q \leqslant
\big(30 \cdot 16^q d^{\ell+1} \ell^{q+1}\big) \cdot \frac{1}{|E_G|}\sum_{\{v,u\}\in E_G} \varrho\big( f(v), f(u) \big)^q.
\end{equation}
\end{proposition}

Proposition \ref{prop:non-concentrated}, whose proof is deferred to Section \ref{sec-non-concentrated}, is the bulk of the proof of Theorem \ref{thm-extrapolation}. Before we proceed to the proof of Theorem \ref{thm-extrapolation}, let us present the following one-sided functional-wise version whose proof is also based on Proposition \ref{prop:non-concentrated}.

\begin{corollary}[One-sided functional-wise extrapolation] \label{cor-extrapolation}
Let $d\geqslant 3$ be an integer, let $G$ be a $d\text{-regular}$ graph with Cheeger constant $h(G)>0$, and let $\mathcal{M}=(M,\varrho)$ be an arbitrary metric space. Also let $1\leqslant p \leqslant q <\infty$, let $C>0$, and set
\begin{equation} \label{extra-intro-e1}
\Gamma=\Gamma\big(d,h(G),p,q,C\big):=\max\bigg\{ \exp\Big( 64\cdot 4^q  \frac{d}{h(G)}  \log d\Big), \, 5^q \, 2^{\frac{q}{p}} \, C^{\frac{q}{p}}\bigg\}.
\end{equation}
If $f\colon V_G\to M$ satisfies
\begin{equation} \label{extra-intro-e2}
\frac{1}{|V_G|^2} \sum_{v,u\in V_G} \varrho\big(f(v),f(u)\big)^p \leqslant C \cdot \frac{1}{|E_G|}
\sum_{\{v,u\}\in E_G} \varrho\big(f(v),f(u)\big)^p,
\end{equation}
then
\begin{equation} \label{extra-intro-e3}
\frac{1}{|V_G|^2} \sum_{v,u\in V_G} \varrho\big(f(v),f(u)\big)^q \leqslant \Gamma \cdot \frac{1}{|E_G|}
\sum_{\{v,u\}\in E_G} \varrho\big(f(v),f(u)\big)^q.
\end{equation}
\end{corollary}

\begin{proof}
If $f$ is not $(5^q,q,1/2)$-concentrated, then, by Proposition \ref{prop:non-concentrated} applied for $\tau=\frac12$, we have
\begin{align*}
& \frac{1}{|V_G|^2} \sum_{v,u\in V_G}  \varrho\big( f(v), f(u) \big)^q  \leqslant
\big(30 \cdot 16^q d^{\ell+1} \ell^{q+1}\big) \cdot \frac{1}{|E_G|}\sum_{\{v,u\}\in E_G} \varrho\big( f(v), f(u) \big)^q \\
& \ \ \ \ \ \leqslant \exp\Big( 64\cdot 4^q  \frac{d}{h(G)}  \log d\Big)
\cdot \frac{1}{|E_G|}\sum_{\{v,u\}\in E_G} \varrho\big( f(v), f(u) \big)^q
\stackrel{\eqref{extra-intro-e1}}{\leqslant} \Gamma \cdot \frac{1}{|E_G|}\sum_{\{v,u\}\in E_G} \varrho\big( f(v), f(u) \big)^q,
\end{align*}
where in the second inequality we have used the fact that $2^{2q+3} \frac{d}{h(G)} \leqslant \ell \leqslant \ell + 1 \leqslant 2^{2q+5} \frac{d}{h(G)}$ and $\ell$~is as in \eqref{eq:2.01}.

So, assume that $f$ is $(5^q,q,1/2)$-concentrated, that is, $\ave_{\mathcal{M}}(f,q) \leqslant 5^q\, Q_\frac{1}{2}(f)^q$. Then, by the definition of $Q_\frac{1}{2}(f)$, we see that
\begin{equation} \label{eq-aux-e1}
\ave_{\mathcal{M}}(f,p)\geqslant \frac{1}{2}\, Q_\frac{1}{2}(f)^p \geqslant \frac{1}{2\cdot 5^p}\, \ave_{\mathcal{M}}(f,q)^{\frac{p}{q}}.
\end{equation}
Consequently,
\begin{align*}
\ave_{\mathcal{M}}(f,q) & \stackrel{\eqref{eq-aux-e1}}{\leqslant} 2^\frac{q}{p} 5^q \, \ave_{\mathcal{M}}(f,p)^\frac{q}{p}
\stackrel{\eqref{extra-intro-e2}}{\leqslant} \big(2^\frac{q}{p} 5^q C^\frac{q}{p}\big)\cdot
\Big(\frac{1}{|E_G|}\sum_{\{v,u\}\in E_G} \varrho\big(f(v), f(u)\big)^p\Big)^\frac{q}{p} \\
& \ \leqslant \big(2^\frac{q}{p} 5^q C^\frac{q}{p}\big) \cdot
\frac{1}{|E_G|}\sum_{\{v,u\}\in E_G} \varrho\big(f(v), f(u)\big)^q
\stackrel{\eqref{extra-intro-e1}}{\leqslant} \Gamma \cdot \frac{1}{|E_G|}\sum_{\{v,u\}\in E_G} \varrho\big( f(v), f(u) \big)^q,
\end{align*}
where the penultimate inequality follows from H\"{o}lder's inequality.
\end{proof}

\subsection{Proof of Theorem \ref{thm-extrapolation}}

We first observe that we may assume that $1 \leqslant p \leqslant q < \infty$. Indeed, let $\mathcal{M}=(M,\varrho)$ be an arbitrary metric space, and let $p\leqslant q$ with $0<p<1$. Set $\eps:=1-p\in (0,1)$, and let $\mathcal{M}'=(M,\varrho^{1-\eps})$ be the $\eps$-snowflake of $\mathcal{M}$. Also set $p':=1$ and $q':= \frac{q}{1-\eps}=\frac{q}{p}$; clearly, we have $1=p'\leqslant q'<\infty$. Then notice that inequality \eqref{eq-extra-intro-new-e1} for $\mathcal{M},p,q$ follows by applying \eqref{eq-extra-intro-new-e1} for the metric space~$\mathcal{M}'$ and the exponents $p'$ and $q'$, and similarly, inequality \eqref{eq-extra-intro-new-e2} for $\mathcal{M},p,q$ follows by applying \eqref{eq-extra-intro-new-e2} for the metric space $\mathcal{M}'$ and the same exponents $p'$ and $q'$.

In the regime $1 \leqslant p \leqslant q < \infty$, we will actually prove the following quantitative version.

\begin{theorem}[Extrapolation for metric spaces---quantitative form]
\label{thm-extrapolation-quantitative}
Let $d, G$ and $\mathcal{M}$ be as in Theorem \ref{thm-extrapolation}. Then, for any $1 \leqslant p \leqslant q<\infty$, we have
\begin{equation} \label{eq-extra-intro-e1}
\gamma(G,\varrho^p) \leqslant \max\Big\{ C_1, \, C_2 \max\big\{1,\gamma(G,\varrho^q)\big\}\Big\}
\end{equation}
and
\begin{equation} \label{eq-extra-intro-e2}
\gamma(G,\varrho^q) \leqslant \max\Big\{ C_3, \, C_4\, \gamma(G,\varrho^p)^{\frac{q}{p}} \Big\},
\end{equation}
where the constants $C_1,C_2,C_3,C_4$ are given by
\begin{align}
\label{edit-e1} C_1:= \exp\Big( 64\cdot 4^p  \frac{d}{h(G)}  \log d\Big),
\ & \ \
C_2: = 24\cdot d\cdot 5^p \, \bigg(88\, p\, \Big(\frac{d}{h(G)}\Big)^2\bigg)^{2q-p},\\
\label{edit-e2} C_3:= \exp\Big( 64\cdot 4^q  \frac{d}{h(G)}  \log d\Big),
\ & \ \
C_4:= 5^q \, 2^{\frac{q}{p}}.
\end{align}
\end{theorem}

\begin{proof}
The regime covered by \eqref{eq-extra-intro-e2} follows from Corollary \ref{cor-extrapolation}. Thus, we only need to show \eqref{eq-extra-intro-e1}. To this end, in what follows we will assume that the quantity $\gamma(G,\varrho^q)$ is well defined. Set $\gamma = \max\{1, \gamma(G,\varrho^q)\}$ and notice that $\gamma\geqslant 1$.

Fix $f\colon V_G\to M$. Arguing as in the proof of Corollary \ref{cor-extrapolation}, if $f$ is not $(5^p,p,1/2)$-concentrated, then
\eqref{eq-extra-intro-e1} follows by Proposition \ref{prop:non-concentrated} applied to $\tau=\frac12$. Thus, setting
\begin{equation} \label{edit-e3}
c_0 := \frac{88\, p\, d^2}{h(G)^2} \ \ \ \text{ and } \ \ \ \Gamma_1 :=
24\, c_0^{q}\,\gamma,
\end{equation}
it is enough to show that if $f$ is $(5^p,p,1/2)$-concentrated, then
\begin{equation} \label{eq-new2}
\ave_{\mathcal{M}}(f,p) \leqslant d\,5^p\, c_0^{q-p}\, \Gamma_1 \cdot \frac{1}{|E_G|} \sum_{\{v,u\}\in E_G} \varrho\big(f(v),f(u)\big)^p.
\end{equation}
Set
\begin{equation}
\label{eq-new3} \alpha:=\frac{h(G)}{2d}, \ \ \ c_1 := \frac{2}{h(G)} \cdot
\frac{1-\alpha}{(1 - (1-\alpha)^{1/p})^p},
 \ \ \ \epsilon:=\frac{1}{\Gamma_1},  \ \ \ \ell_*:=\bigg\lceil
\frac{\log\frac{1}{\epsilon}}{\log\frac{1}{1-\alpha}}\bigg\rceil,
\end{equation}
and notice that with these choices and \eqref{edit-e3} we have
\begin{equation} \label{edit-epsilon}
1\leqslant 22(c_1\, d)^{1/p}\leqslant c_0 \ \ \ \ \ \text{ and } \ \ \ \ \
\epsilon\leqslant \frac{1}{100}.
\end{equation}

Consider, first, the somewhat degenerate case that $|V_G|\leqslant 4\epsilon^{-1}$. Then, by H\"{o}lder's inequality and the fact that $\|\cdot\|_{\ell_q} \leqslant \|\cdot\|_{\ell_p}$ (since $q\geqslant p$), we see that
\begin{align*}
\ave_{\mathcal{M}}(f,p) & \leqslant \ave_{\mathcal{M}}(f,q)^{p/q} \leqslant \gamma^{p/q} \cdot
\Big( \frac{1}{|E_G|} \sum_{\{v,u\}\in E_G} \varrho\big(f(v),f(u)\big)^q\Big)^{p/q} \\
& \ \ \ \ \ \leqslant \big(\gamma^{p/q} \cdot
|E_G|^{1-p/q}\big) \cdot \frac{1}{|E_G|} \sum_{\{v,u\}\in E_G}
\varrho\big(f(v),f(u)\big)^p \\
& \stackrel{\eqref{edit-e3},\eqref{eq-new3}}{\leqslant} 2d\,\Gamma_1 \cdot \frac{1}{|E_G|} \sum_{\{v,u\}\in E_G} \varrho\big(f(v),f(u)\big)^p.
\end{align*}
As a consequence, \eqref{eq-new2} is satisfied if $|V_G|\leqslant 4\epsilon^{-1}$. Thus, in what follows, we may assume that
\begin{equation} \label{e-asymptotic}
|V_G|\geqslant 4 \epsilon^{-1}.
\end{equation}
For each $\xi>0$, define $\mathcal{E}_\xi \subseteq E_G$ by
\begin{equation} \label{eq-new4}
\mathcal{E}_\xi := \big\{ \{v,u\}\in E_G \colon \varrho\big(f(v),f(u)\big)
\geqslant  \xi \cdot \ave_{\mathcal{M}}(f,p)^{1/p} \big\}.
\end{equation}
The following lemma is the main step of the proof.

\begin{lemma} \label{lemma-new}
Assume that $f$ is $(5^p,p,1/2)$-concentrated. Then at least one of the following holds.
\begin{enumerate}
\item[(H1)] \label{edit-H1} For some $\xi>0$, we have $|\mathcal E_\xi| \geqslant \Gamma_1^{-1} \, \xi^{-p}\cdot |E_G|$.
\item[(H2)] \label{edit-H2} We have
\begin{equation} \label{edit-eq-H2}
\Gamma_1 \cdot \frac{1}{|E_G|} \sum_{\substack{\{v,u\}\in E_G \\
\varrho(f(v),f(u))\leqslant c_{0}\,Q_{\frac{1}{2}}(f)}}
\varrho\big(f(v),f(u)\big)^q \geqslant  Q_{\frac{1}{2}}(f)^q.
\end{equation}
\end{enumerate}
\end{lemma}

\begin{proof}[Proof of Lemma \ref{lemma-new}]
Assume for contradiction that neither (\hyperref[edit-H1]{H1}) nor (\hyperref[edit-H2]{H2}) holds. By the definition of
$Q_{\frac12}(f)$ and the pigeonhole principle, there is $v_0 \in V_G$ such that, setting
\[ S_0 := \big\{ u \in V_G \colon \varrho\big(f(u),f(v_0)\big) \leqslant Q_{\frac12}(f)\big\}, \]
we have $|S_0| \geqslant \frac{|V_G|}{2}$ and, by triangle inequality, we have\footnote{Recall that for any nonempty $A\subseteq M$, we set $\mathrm{diam}_{\varrho}(A):=\max\big\{\varrho(x,y)\colon x,y\in A\big\}$.} $\diam_\varrho\big(f(S_0)\big) \leqslant 2 \, Q_{\frac12}(f)$;
refining if necessary (since $\epsilon\leqslant 1/100$ and $|V|>4/\epsilon\geqslant 400$), we may also assume that $|S_0|\leqslant (1-\epsilon)|V_G|$.

For every $\ell\in [\ell_*]$, we shall select, recursively, a set $S_\ell\subseteq B_G(S_0,\ell)$ such that
\begin{enumerate}
\item[(i)] either $|S_\ell^\complement| \leqslant \epsilon\,|V_G|$ or $|S_\ell^\complement| =|S_{\ell-1}^\complement|-\big\lceil \alpha
|S_{\ell-1}^\complement|\big\rceil$, and
\item[(ii)] $\mathrm{diam}_{\varrho}\big( f(S_\ell)\big) \leqslant \Big( 2+ 10\big(c_1 \,\frac{|E_G|}{|S_{\ell}^\complement|}
\,\Gamma_1^{-1}\big)^{1/p}\Big)\cdot Q_{\frac12}(f)$.
\end{enumerate}
Let $\ell\in [\ell_*]$, and assume that $S_{\ell-1}$ has been selected. If $|S_{\ell-1}|\geqslant (1-\epsilon)|V_G|$, then set $S_\ell:=S_{\ell-1}$. Otherwise, set
\[ \xi_\ell := \bigg( \frac{2}{h(G)}\cdot\Gamma_1^{-1}\cdot\frac{|E_G|}{|S_{\ell-1}^\complement|}\bigg)^{1/p}.\]
Since we assume the negation of (\hyperref[edit-H1]{H1}), this choice ensures that
\[ |\mathcal{E}_{\xi_\ell}| \leqslant \frac{h(G)}{2}\,|S_{\ell-1}^\complement|; \]
moreover, by \eqref{eq-new4} and the assumption that $f$ is $(5^p,p,1/2)$-concentrated, for all $\{v,u\}\in \mathcal{E}_{\xi_\ell}^\complement$,
\begin{equation}\label{eq:Ek-complement}
\varrho\big(f(v),f(u)\big) \leqslant  \xi_\ell\cdot \ave_{\mathcal{M}}(f,p)^{1/p} \leqslant  5\, \Big(\frac{2}{h(G)}\cdot
\frac{|E_G|}{|S_{\ell-1}^\complement|}\cdot\Gamma_1^{-1}\Big)^{1/p} \cdot Q_{\frac12}(f).
\end{equation}
Using the fact that $\frac{|V_G|}{2}\geqslant |S_{\ell-1}^{\complement}|\geqslant \epsilon |V_G|$, the definition of Cheeger's constant $h(G)$ and our estimate on $|\mathcal{E}_{\xi_\ell}|$, we obtain that
\[ \Big|\Big\{ \{v,u\}\in E_G\colon v\in S_{\ell-1}, u\in S_{\ell-1}^\complement \text{ and } \varrho\big(f(v),f(u)\big)
\leqslant \xi_\ell\cdot \ave_{\mathcal{M}}(f,p)^{1/p}\Big\}\Big|\geqslant
\frac{h(G)}{2}\, |S_{\ell-1}^\complement|. \]
Thus, by the $d$-regularity of $G$ and the choice of $\alpha$ in \eqref{eq-new3}, there exists a subset $\Theta_{\ell}$ of $\partial B_G(S_{\ell-1},1)$ with $|\Theta_{\ell}|= \big\lceil \alpha |S_{\ell-1}^\complement|\big\rceil$ such that for every $u\in \Theta_{\ell}$, we have $\varrho\big(f(u), f(S_{\ell-1})\big) \leqslant \xi_\ell\cdot \ave_{\mathcal{M}}(f,p)^{1/p}$. Set $S_\ell:=S_{\ell-1}\cup \Theta_\ell$. With this choice we have
\begin{equation} \label{eq:(i)}
|S_\ell^\complement| = |S_{\ell-1}^\complement|-\big\lceil \alpha |S_{\ell-1}^\complement|\big\rceil;
\end{equation}
in particular, part (i) of the selection is satisfied, To see that part (ii) is also satisfied, fix $v,u\in S_\ell$ and observe that our inductive assumptions and the triangle inequality yield
\begin{align} \label{eq-d1}
\varrho\big(f(v),&\, f(u)\big) \stackrel{\eqref{eq:Ek-complement}}{\leqslant}
\bigg( 2+ 10\Big(c_1 \,\frac{|E_G|}{|S_{\ell-1}^\complement|} \,\Gamma_1^{-1}\Big)^{1/p}  +
10\Big(\frac{2}{h(G)}\cdot \frac{|E_G|}{|S_{\ell-1}^\complement|}\cdot\Gamma_1^{-1} \Big)^{1/p}\bigg)
\cdot Q_{\frac12}(f) \\
& \stackrel{\eqref{eq:(i)}}{\leqslant} \bigg(2 +
10\Big(c_1^{1/p} + \Big(\frac{2}{h(G)}\Big)^{1/p}\Big)\cdot
(1-\alpha)^{1/p} \cdot \Big( \frac{|E_G|}{|S_{\ell}^\complement|} \,\Gamma_1^{-1} \Big)^{1/p}
\bigg) \cdot Q_{\frac12}(f) \nonumber \\
& \stackrel{\eqref{eq-new3}}{=} \bigg( 2+ 10\Big(c_1
\,\frac{|E_G|}{|S_{\ell}^\complement|} \,\Gamma_1^{-1}\Big)^{1/p} \bigg)
\cdot Q_{\frac12}(f). \nonumber
\end{align}
This completes the recursive selection of the sets $S_1,\dots, S_{\ell_*}$.

Now set $S:=S_{\ell_*}$ and notice that, by part (i) of the selection and the choice of $\ell_*$ in \eqref{eq-new3},
\begin{equation} \label{eq-d2}
|S| \geqslant \min\Big\{ (1-\epsilon)|V_G|, \big(1-(1-\alpha)^{\ell_*}\big)|V_G|\Big\}\geqslant (1-\epsilon)|V_G|;
\end{equation}
moreover, using part (i) of the selection once again, \eqref{e-asymptotic} and the choice of $\alpha$ in \eqref{eq-new3}, we have
\begin{equation} \label{edit-new}
|S^{\complement}| \geqslant \frac{\epsilon}{4}\, |V_G|.
\end{equation}
Recall $v_0\in V_G$ from the construction of $S_0$. Define the auxiliary function $g\colon V_G \to M$ by
\[ g(u) :=
\begin{cases}
f(u)   & u \in S,\\
f(v_0) & u \in S^\complement.
\end{cases} \]

\begin{claim} \label{cl-dnew}
We have
\begin{equation} \label{eq-d3}
\frac{\Gamma_1}{|E_G|} \sum_{\{v,u\}\in E_G} \varrho\big(g(v),g(u)\big)^q \leqslant (3c_0^q)\cdot Q_{\frac12}(f)^q.
\end{equation}
\end{claim}

\begin{proof}[Proof of Claim \ref{cl-dnew}]
Fix $\{v,u\} \in E_G$. If both $v,u \in S^{\complement}$, then $\varrho\big(g(v),g(u)\big) = 0$. On the other hand, if both $u,v \in S$,
then part (ii) of the recursive selection yields that
\begin{align}
\label{edit-eq1.1} \varrho\big(g(v),g(u)\big) & \ \, =
\varrho\big(f(v),f(u)\big) \leqslant \mathrm{diam}_{\varrho}\big(f(S)\big) \\
& \ \, \leqslant \bigg( 2+ 10\Big(c_1 \,\frac{|E_G|}{|S^\complement|}
\,\Gamma_1^{-1}\Big)^{1/p}\bigg)\cdot Q_{\frac12}(f) \nonumber \\
& \stackrel{\eqref{edit-new}}{\leqslant} \bigg( 2+ 10\Big(c_1
\,\frac{2d}{\epsilon} \,\Gamma_1^{-1}\Big)^{1/p}\bigg)\cdot
Q_{\frac12}(f)\leqslant 22\, \Big(\frac{c_1\, d}{\epsilon\,
\Gamma_1}\Big)^{1/p} \cdot Q_{\frac12}(f), \nonumber
\end{align}
that further implies, by \eqref{eq-new3} and \eqref{edit-epsilon}, that $\varrho\big(g(v),g(u)\big)=\varrho\big(f(v),f(u)\big)\leqslant c_0\, Q_{\frac12}(f)$; combining this estimate with the assumed negation of~(\hyperref[edit-H2]{H2}), we obtain that
\begin{equation} \label{edit-eq1.2}
\frac{\Gamma_1}{|E_G|} \sum_{\substack{\{v,u\}\in E_G \\ \{v,u\} \subseteq
S}} \varrho\big(f(v),f(u)\big)^q \leqslant Q_{\frac{1}{2}}(f)^q.
\end{equation}
Finally, if $v \in S^\complement$ and $u\in S$, then, as above,
\[ \varrho\big(g(v),g(u)\big) = \varrho\big(f(v_0),f(u)\big) \leqslant \mathrm{diam}_{\varrho}\big(f(S)\big)
\stackrel{\eqref{edit-eq1.1}}{\leqslant} 22\, \Big(\frac{c_1\, d}{\epsilon\, \Gamma_1}\Big)^{1/p} \cdot Q_{\frac12}(f); \]
note, however, that, by the $d$-regularity of $G$ and the fact that $|S^\complement|\leqslant \epsilon |V_G|$, there are at most $d\epsilon
|V_G|=2\epsilon |E_G|$ such edges. Therefore, invoking again \eqref{eq-new3} and \eqref{edit-epsilon},
\begin{equation} \label{edit-eq1.3}
\frac{\Gamma_1}{|E_G|} \sum_{\substack{\{v,u\}\in E_G \\ v\in S^\complement,\, u\in S}} \varrho\big(g(v),g(u)\big)^q  \leqslant
2\epsilon\,\Gamma_1\,  22^q\, \Big(\frac{c_1\, d}{\epsilon\, \Gamma_1}\Big)^{q/p} \cdot Q_{\frac12}(f)^q
\leqslant (2c_0^q) \cdot Q_{\frac12}(f)^q.
\end{equation}
The desired estimate \eqref{eq-d3} follows by \eqref{edit-eq1.2} and \eqref{edit-eq1.3}. The claim is proved.
\end{proof}

We are now in a position to derive the contradiction. By the definition of $Q_{\frac12}(f)$, the fact that $|S^\complement|\leqslant \epsilon |V_G|$, the estimate on $\epsilon$ in \eqref{edit-epsilon} and the definition of $g$, we have
\[ \frac{1}{|V_G|^2} \big|\big\{ (v,u)\in V_G\times V_G \colon
\varrho\big(g(v),g(u)\big)\geqslant Q_{\frac12}(f)\big\}\big|\geqslant \frac14, \]
that implies that
\begin{equation} \label{eq-new6}
\ave_{\mathcal{M}}(g,q)\geqslant \frac14 \, Q_{\frac12}(f)^q.
\end{equation}
Finally, by the definitions of $\gamma(G,\varrho^q)$ and $\gamma$,
\begin{equation} \label{eq-new7}
\ave_{\mathcal{M}}(g,q)\leqslant \frac{\gamma}{|E_G|} \sum_{\{v,u\}\in E_G} \varrho\big(g(v),g(u)\big)^q.
\end{equation}
By \eqref{eq-d3}, \eqref{eq-new6}, \eqref{eq-new7}, and the choice of $\Gamma_1$ in \eqref{edit-e3}, we arrive at the desired contradiction
\[ \frac14 \, Q_{\frac12}(f)^q \leqslant \ave_{\mathcal{M}}(g,q) \leqslant
\bigg(\frac{3\,c_0^q\, \gamma}{\Gamma_1}\bigg)\cdot
Q_{\frac{1}{2}}(f)^q \leqslant \frac{1}{8}\, Q_{\frac{1}{2}}(f)^q. \]
The proof of Lemma \ref{lemma-new} is thus completed.
\end{proof}

We are ready to complete the proof of \eqref{eq-new2}. If (\hyperref[edit-H1]{H1}) holds, then the result is immediate by definition
of $\mathcal E_\xi$. On the other hand, if (\hyperref[edit-H2]{H2}) holds, then, by the fact that $\gamma(G,\varrho^p)\geqslant 1$ and the fact that $f$ is $(5^p,p,1/2)$-concentrated and $q\geqslant p$,
\begin{align*}
5^p \, c_0^{q-p}\, \Gamma_1 \cdot \frac{1}{|E_G|} \sum_{\{v,u\}\in E_G} &
\frac{\varrho\big(f(v),f(u)\big)^p}{\ave_{\mathcal{M}}(f,p)} \geqslant 5^p\, c_0^{q-p} \cdot \frac{\Gamma_1}{|E_G|}
\sum_{\substack{\{v,u\}\in E_G \\ \varrho(f(v),f(u)) \leqslant c_{0}\,Q_{\frac{1}{2}}(f)}}
\frac{\varrho\big(f(v),f(u)\big)^p}{\ave_{\mathcal{M}}(f,p)} \\
& \geqslant c_0^q \cdot  \frac{\Gamma_1}{|E_G|} \sum_{\substack{\{v,u\}\in
E_G \\ \varrho(f(v),f(u)) \leqslant c_{0}\,Q_{\frac{1}{2}}(f)}}
\bigg(\frac{\varrho\big(f(v),f(u)\big)}{c_0\,Q_{\frac{1}{2}}(f)}\bigg)^p \\
& \geqslant c_0^q \cdot \frac{\Gamma_1}{|E_G|} \sum_{\substack{\{v,u\}\in
E_G \\ \varrho(f(v),f(u)) \leqslant c_{0}\,Q_{\frac{1}{2}}(f)}} \bigg(
\frac{\varrho\big(f(v),f(u)\big)}{c_0\,Q_{\frac{1}{2}}(f)}\bigg)^q
\stackrel{\eqref{edit-eq-H2}}{\geqslant} c_0^{q-q} = 1.
\end{align*}
This completes the proof of \eqref{eq-new2}, and so, the entire proof of Theorem \ref{thm-extrapolation} is completed.
\end{proof}

\section{Proof of Proposition \ref*{prop:non-concentrated}} \label{sec-non-concentrated}

We start by observing that we may assume that $Q_\tau(f)>0$. Indeed, suppose that $Q_\tau(f)=0$. Let $\mathcal{M}'=(M\times [0,1],\varrho_{M\times [0,1]})$ be the metric space with
\[ \varrho_{M\times[0,1]}\big( (v,x),(u,y) \big) = \varrho(v,u) + |x-y|, \]
and define $g\colon M\to M\times [0,1]$ by setting $g(v):= (v,1/2)$, for all $v\in M$. Notice that $g$ is an isometry; consequently, $f$ satisfies \eqref{eq:2.02} if and only if $g\circ f$ satisfies the analogue of~\eqref{eq:2.02}. We select a sequence $(h_k)$ of one-to-one maps from $[n]$ into $M\times[0,1]$ that converge pointwise to $g \circ f$. Since $\tau>1/n$ and $h_k$ is one-to-one for all $k$, we see that $Q_\tau(h_k)>0$ for all $k$, and $Q_\tau(h_k)\to 0$; moreover, $\ave_{\mathcal{M}'}(h_k,q) \to \ave_{\mathcal{M}'}(f,q)$. Therefore, since $f$ is not $(C_R,q,\tau)$-concentrated, there exists a positive integer $k_0$ such that for all $k\geqslant k_0$, the map $h_k\colon [n]\to M\times [0,1]$ is not $(C_R,q,\tau)$-concentrated either. After noticing that
\[ \frac{1}{|E_G|}\sum_{\{v,u\}\in E_G} \varrho_{\mathcal{M}\times [0,1]}\big( h_k(v), h_k(u) \big)^q
\to \frac{1}{|E_G|}\sum_{\{v,u\}\in E_G} \varrho\big( f(v), f(u) \big)^q, \]
we conclude that if $h_k$ satisfies the analogue of \eqref{eq:2.02} for all $k\geqslant k_0$, then $f$ must also satisfy \eqref{eq:2.02}. So, in what follows, we will assume that $Q_\tau(f)>0$.

By the definition of $Q_\tau(f)$, there exists a vertex $\tilde{v}\in [n]$ such that, setting
\begin{equation}\label{eq:2.03}
N_{\tilde{v}}:=\big\{ v\in [n]\colon \varrho\big( f(v), f(\tilde{v})\big)\leqslant Q_\tau(f) \big\},
\end{equation}
we have
\begin{equation}\label{eq:2.04}
|N_{\tilde{v}}| \geqslant \tau n.
\end{equation}
Set
\begin{equation} \label{eq-A&B}
A:=\big\{ v\in [n]\colon \varrho\big( f(v), f(\tilde{v}) \big)> 2Q_\tau(f)\big\} \ \ \ \text{ and } \ \ \
B:=\big\{ v\in [n]\colon \varrho\big( f(v), f(\tilde{v}) \big) \leqslant 2Q_\tau(f)\big\}.
\end{equation}

\begin{lemma} \label{l:2.2}
We have
\begin{equation}\label{eq:2.05}
\frac{1}{n^2}\sum_{v,u\in [n]} \varrho\big( f(v), f(u) \big)^q \leqslant
\big(5\cdot 2^{q-1}\, d\big) \cdot \frac{1}{|E_G|} \sum_{v\in A} \varrho \big(f(v), f(\tilde{v})\big) ^q.
\end{equation}
\end{lemma}

\begin{proof}
By triangle inequality, the fact that $(a+b)^q\leqslant 2^{q-1}(a^q+b^q)$ for every $a,b\geqslant 0$ and the fact that $[n]$ is the disjoint union of $A$ and $B$, we obtain that
\begin{align} \label{eq:2.06}
\ave_{\mathcal{M}}(f,q) & \leqslant \frac{1}{n^2} \sum_{v,u\in [n]} \Big( \varrho( f(v), f(\tilde{v}) \big) +
\varrho\big( f(u), f(\tilde{v}) \big) \Big)^q \\
& \leqslant  \frac{2^{q-1}}{n^2} \sum_{v,u\in [n]} \Big(\varrho\big( f(v), f(\tilde{v}) \big)^q +
\varrho\big( f(u), f(\tilde{v}) \big)^q \Big) =
\frac{2^q}{n} \sum_{v\in [n]} \varrho\big( f(v), f(\tilde{v}) \big)^q \nonumber \\
& = \frac{2^q}{n} \sum_{v\in A}  \varrho\big( f(v), f(\tilde{v}) \big)^q +
\frac{2^q}{n} \sum_{v\in B} \varrho\big( f(v), f(\tilde{v}) \big)^q . \nonumber
\end{align}
Using the fact that $f$ is not $(C_R,q,\tau)$-concentrated and $C_R\geqslant 5^q$, we see that $\ave_{\mathcal{M}}(f,q)> 5^q Q_\tau(f)^q$. Thus, by the definition of $B$,
\begin{equation} \label{eq:2.07}
\frac{2^q}{n} \sum_{v\in B} \varrho\big( f(v), f(\tilde{v}) \big)^q \leqslant 4^q Q_\tau(f)^q \leqslant \Big(\frac{4}{5}\Big)^q \ave_{\mathcal{M}}(f,q).
\end{equation}
Since $q\geqslant 1$, by \eqref{eq:2.06} and \eqref{eq:2.07}, we have
\begin{equation}\label{eq:2.08}
\frac{2^q}{n} \sum_{v\in A} \varrho\big(f(v), f(\tilde{v}) \big)^q \geqslant \frac{1}{5}\,\ave_{\mathcal{M}}(f,q).
\end{equation}
Inequality \eqref{eq:2.05} follows from \eqref{eq:2.08} and the fact that $|E_G|=dn/2$.
\end{proof}

By the choice of $\ell$ in \eqref{eq:2.01} and the fact that $x\leqslant\log_2 (1+x)\leqslant x/\log 2$ for every $x\in [0,1]$,
\[ \ell \geqslant \bigg\lceil\frac{1}{\log_2\big(1+\frac{h(G)}{2^{2q+4}d}\big)} \bigg\rceil \geqslant
\bigg\lceil\frac{2^{2q+3}}{\log_2\big(1+\frac{h(G)}{d}\big)}\bigg\rceil \]
and, therefore,
\begin{equation}\label{eq:2.09}
\Big(1 + \frac{h(G)}{d} \Big)^\ell\geqslant 2^{2^{2q+3}} \geqslant 4^{q+1}.
\end{equation}

\begin{lemma} \label{l:2.3}
We have
\begin{equation}\label{eq:2.10}
\big| B_G(N_{\tilde{v}},\ell)\big| \geqslant \Big(1-\frac{1}{2^{2q+4}}\Big)n.
\end{equation}
\end{lemma}

\begin{proof}
Set
\[\ell_1:= \bigg\lceil\frac{\max\big\{\log_2\frac{1}{2\tau },0\big\}}{\log_2\big(1+\frac{h(G)}{d}\big)} \bigg\rceil
\ \ \ \text{ and } \ \ \
\ell_2 := \bigg\lceil\frac{1}{\log_2\big(1+\frac{h(G)}{2^{2q+4}d}\big)} \bigg\rceil.\]
Notice, in particular, that if $\tau\geqslant \frac12$, then $\ell_1=0$. Moreover, we have
\begin{equation} \label{eq:2.11}
\big| B_G(N_{\tilde{v}},\ell_1)\big|\geqslant \min\bigg\{\frac{n}{2}, \Big(1+\frac{h(G)}{d}\Big)^{\ell_1}|N_{\tilde{v}}|\bigg\}=\frac{n}{2}.
\end{equation}
Indeed, if $\tau\geqslant \frac12$, then \eqref{eq:2.11} follows immediately from \eqref{eq:2.04}; on the other hand, if $\tau<1/2$, then \eqref{eq:2.11} follows from the definition of $h(G)$ in \eqref{eq-cheeger}, the $d$-regularity of $G$, the choice of $\ell_1$ and \eqref{eq:2.04}. Next observe that, for every $S\subseteq [n]$ with $\frac{n}{2}\leqslant |S|\leqslant (1-2^{-(2q+4)})n$, we have $| S^{\complement}| \geqslant 2^{-(2q+4)} |S|$. Hence, invoking again the definition of $h(G)$ and the $d$-regularity of $G$, for every $S\subseteq [n]$ with $\frac{n}{2}\leqslant|S|\leqslant (1-2^{-(2q+4)})n$, we have
\begin{equation} \label{eq-new1}
\big|B_G(S,1)\big|\geqslant \min\bigg\{ \Big(1-\frac{1}{2^{2q+4}}\Big)n, \Big(1+\frac{h(G)}{2^{2q+4}d}\Big) |S| \bigg\}.
\end{equation}
The proof is completed by \eqref{eq:2.11} and \eqref{eq-new1} after observing that $\ell= \ell_1+\ell_2$.
\end{proof}

We define a dyadic partition of $A$ into $(A_k)_{k=1}^\infty$ by setting, for every integer $k\geqslant 1$,
\begin{equation} \label{eq-A_k}
A_k:= \Big\{v\in [n] \colon  2^{k-1}2Q_\tau(f) < \varrho\big(f(v),f(\tilde v)\big) \leqslant 2^k 2Q_\tau(f)]\Big\}.
\end{equation}

\begin{lemma}[Main lemma] \label{l:2.4}
For every integer $k\geqslant 1$, there exists an integer $m_k\geqslant 0$ such that
\begin{equation}\label{eq:2.12}
\Big| \Big\{ \{v,u\}\in E_G \colon \varrho\big(f(v), f(u)\big)\geqslant \frac{2^k2 Q_\tau(f)}{2^{m_k} 8\ell} \Big\} \Big|\geqslant \frac{3^{m_k q}}{2 \ell d^\ell} |A_k|.
\end{equation}
\end{lemma}

\begin{proof}
Fix $k$. If $A_k=\emptyset$, then \eqref{eq:2.12} is straightforward; so, suppose that $A_k\neq\emptyset$. Assume, towards a contradiction, that for every nonnegative integer $m$,
\begin{equation}\label{eq:2.13}
\Big| \Big\{ \{v,u\}\in E_G \colon \varrho\big(f(v), f(u)\big)\geqslant \frac{2^k2 Q_\tau(f)}{2^m 8\ell} \Big\} \Big|< \frac{3^{mq}}{2 \ell d^\ell} |A_k|.
\end{equation}
Set $A_{k,0}:=A_k$, and define recursively, for every positive integer $m$,
\[ A_{k,m}:=\bigg\{v\in[n]\colon \mathrm{dist}_G(v,A_{k,m-1})\leqslant \ell \text{ and }
\varrho\big(f(v),f(\tilde v)\big)\geqslant \Big(1-\frac{1}{4}\big(1-\frac{1}{2^m}\big)\Big) 2^{k-1}2Q_\tau(f)\bigg\}.\]
Notice that the sequence $(A_{k,m})_{m=0}^\infty$ is increasing with respect to inclusion.

\begin{claim} \label{cl:2.5}
For every integer $m\geqslant 1$, if\, $4^{q+1}|A_{k,m-1}|\leqslant \frac{n}{2}$, then $|A_{k,m}|\geqslant 3^q|A_{k,m-1}|$.
\end{claim}

\begin{proof}[Proof of Claim \ref{cl:2.5}]
Suppose not. We define $m_0$ to be the smallest positive integer such that
$4^{q+1}|A_{k,m_0-1}|\leqslant \frac{n}{2}$ and
\begin{equation} \label{eq:2.14}
|A_{k,m_0}|< 3^q |A_{k,m_0-1}|.
\end{equation}
Since the sequence $(A_{k,m})_{m=0}^\infty$ is increasing, by the choice of $m_0$, we must have
\begin{equation}\label{eq:2.15}
|A_{k,m_0-1}|\geqslant 3^{(m_0-1)q} |A_k|.
\end{equation}
Set
\[ U:= \bigg\{v\in[n]\colon \mathrm{dist}_G(v,A_{k,m_0-1})\leqslant \ell \text{ and }
\varrho\big(f(v),f(\tilde v)\big)< \Big(1-\frac{1}{4}\big(1-\frac{1}{2^{m_0}}\big)\Big) 2^{k-1}2Q_\tau(f)\bigg\}. \]
By the definition of $h(G)$, the $d$-regularity of $G$, the fact that $4^{q+1}|A_{k,m_0-1}|\leqslant \frac{n}{2}$ and \eqref{eq:2.09}, we~have
\[ \big| B_G(A_{k,m_0-1},\ell)\big|  \geqslant
\min\bigg\{\frac{n}{2}, \Big(1+\frac{h(G)}{d}\Big)^\ell |A_{k,m_0-1}| \bigg\} \geqslant 4^{q+1}  |A_{k,m_0-1}|. \]
After observing that $B_G(A_{k,m_0-1},\ell)$ is the disjoint union of $A_{k,m_0}$ and $U$, by \eqref{eq:2.14} and \eqref{eq:2.15},
\begin{equation}\label{eq:2.16}
|U|\geqslant 3^q |A_{k,m_0-1}| \geqslant 3^{m_0\, q}|A_k|.
\end{equation}
By the triangle inequality and the choice of $U$ and $A_{k,m_0-1}$, for every $u\in U$ and every $v\in A_{k,m_0-1}$,
\begin{equation}\label{eq:2.17}
\varrho\big(f(v), f(u)\big) \geqslant \frac{2^{k}2Q_\tau(f)}{2^{m_0}8}.
\end{equation}
Next observe that $U$ is a subset of $B_G(A_{k,m_0-1},\ell)$, and so, for every $u\in U$, there exists a path $P_u$ of length at most $\ell$ connecting $u$ to some vertex in $A_{k,m_0-1}$. By \eqref{eq:2.17}, for every $u\in U$, there exists an edge $\{v_u,w_u\}\in E_G$ that belongs to this path $P_u$ such that
\[ \varrho\big(f(v_u), f(w_u)\big) \geqslant \frac{2^{k}2Q_\tau(f)}{2^{m_0}8\ell}. \]
On the other hand, due to the $d$-regularity of $G$, every edge of the form $\{v_u,w_u\}$ belongs to at most $\ell d^\ell$ many paths of the form $P_{u'}$. Therefore,
\[ \Big| \Big\{ \{v,u\}\in E_G \colon \varrho\big(f(v), f(u)\big)\geqslant \frac{2^k2 Q_\tau(f)}{2^{m_0} 8\ell} \Big\} \Big|
\geqslant \frac{|U|}{\ell d^\ell} \stackrel{\eqref{eq:2.16}}{\geqslant} \frac{3^{m_0 q}}{\ell d^\ell} |A_k|, \]
which contradicts \eqref{eq:2.13}. The proof of Claim \ref{cl:2.5} is thus completed.
\end{proof}

We define $m_k$ to be the smallest nonnegative positive integer such that
\begin{equation}\label{eq:2.18}
4^{q+1}|A_{k,m_k}|>\frac{n}{2}.
\end{equation}
By repeated applications of Claim \ref{cl:2.5}, we obtain that
\begin{equation} \label{eq:2.19}
|A_{k,m_k}| \geqslant 3^{m_k q}|A_k|.
\end{equation}
Hence, by \eqref{eq:2.10}, \eqref{eq:2.18} and \eqref{eq:2.19}, we have
\begin{equation} \label{eq:2.20}
\big|B_G(N_{\tilde{v}},\ell) \cap A_{k,m_k}\big|\geqslant \frac{|A_{k,m_k}|}{2}\geqslant \frac{1}{2} \,3^{m_kq}|A_k|.
\end{equation}
By the triangle inequality and the choice of $N_{\tilde{v}}$ and $A_{k,m_k}$, for every $v\in N_{\tilde{v}}$ and every $u \in A_{k,m_k}$,
\begin{equation} \label{eq:2.21}
\varrho\big(f(v), f(u)\big) \geqslant \frac{2^{k}2Q_\tau(f)}{8}.
\end{equation}
Notice that for every $v\in B_G(N_{\tilde{v}},\ell) \cap A_{k,m_k}$, there exists a path $P_v$ of length at most $\ell$ connecting $v$ to some vertex in $N_{\tilde{v}}$. Therefore, by \eqref{eq:2.21}, for every $v\in B_G(N_{\tilde{v}},\ell) \cap A_{k,m_k}$, there exists an edge $\{u_v,w_v\}\in E_G$ that belongs to this path $P_v$ such that
\[ \varrho\big(f(u_v), f(w_v)\big) \geqslant \frac{2^{k}2Q_\tau(f)}{8\ell}. \]
As in the proof of Claim \ref{cl:2.5}, we note that, due to the $d$-regularity of $G$, every edge of the form $\{u_v,w_v\}$ belongs to at most $\ell d^\ell$ many paths of the form $P_{v'}$. Thus, we conclude that
\[ \Big| \Big\{ \{v,u\}\in E_G \colon \varrho\big(f(v), f(u)\big)\geqslant \frac{2^k2 Q_\tau(f)}{ 8\ell} \Big\} \Big|
\geqslant \frac{|B_G(N_{\tilde{v}},\ell) \cap A_{k,m_k}|}{\ell d^\ell}
\stackrel{\eqref{eq:2.20}}{\geqslant} \frac{3^{m_k q}}{2\ell d^\ell} |A_k|, \]
which contradicts, again, \eqref{eq:2.13}. The proof of Lemma \ref{l:2.4} is completed.
\end{proof}

We are finally in a position to complete the proof of Proposition \ref{prop:non-concentrated}. By the choice of $A$ in \eqref{eq-A&B}, the choice of the sequence $(A_k)_{k=1}^\infty$ in \eqref{eq-A_k}, and \eqref{eq:2.12}, we have
\begin{align}\label{eq:2.22}
& \sum_{v\in A} \varrho\big(  f(v), f(\tilde{v})\big) ^q  \leqslant \sum_{k=1}^{\infty}2^{kq}\big(2Q_\tau(f)\big)^q |A_k|\\
&  \ \leqslant  (2\,\ell^{q+1} d^\ell 8^q)\cdot \sum_{k=1}^{\infty} \Big(\frac{2}{3} \Big)^{m_k}
 \Big(\frac{2^k2 Q_\tau(f)}{2^{m_k} 8\ell}\Big)^q \,
\Big| \Big\{ \{v,u\}\in E_G \colon \varrho\big(f(v), f(u)\big)\geqslant \frac{2^k2 Q_\tau(f)}{2^{m_k} 8\ell} \Big\} \Big| \nonumber\\
& \ \leqslant  (2\,\ell^{q+1} d^\ell 8^q)\cdot \sum_{m=0}^{\infty} \Big(\frac{2}{3} \Big)^m \sum_{k=1}^{\infty}
 \Big(\frac{2^k2 Q_\tau(f)}{2^m 8\ell}\Big)^q \,
\Big| \Big\{ \{v,u\}\in E_G \colon \varrho\big(f(v), f(u)\big)\geqslant \frac{2^k2 Q_\tau(f)}{2^m 8\ell} \Big\} \Big| \nonumber\\
& \ \leqslant (4\, \ell^{q+1} d^\ell 8^q) \cdot \sum_{m=0}^{\infty} \Big(\frac{2}{3} \Big)^m  \sum_{\{v,u\}\in E_G} \varrho\big(f(v), f(u)\big)^q
\leqslant (12\, \ell^{q+1} d^\ell 8^q) \sum_{\{v,u\}\in E_G} \varrho\big(f(v), f(u)\big)^q. \nonumber
\end{align}
The desired estimate \eqref{eq:2.02} follows by combining \eqref{eq:2.05} and \eqref{eq:2.22}. The proof of Proposition \ref{prop:non-concentrated} is completed.


\part{Metric Poincar\'{e} inequalities for random regular graphs} \label{part2}

\section{Models for random regular graphs, and related random variables} \label{sec-models}

In what follows, let $n> d\geqslant 3$ be integers such that $dn$ is even. Recall that $G(n,d)$ denotes the set of all $d$-regular graphs on $[n]$. By $\mathbb{S}_n$ we denote the symmetric group over $[n]$. The purpose of the following sequence of definitions is to split the generation of a random regular graph into several stages: a partial edge set is selected, followed by a completion of the graph, followed by a random labelling of the vertices.

\subsection{The random variable $\boldsymbol{G}$} \label{subsec-G}

We fix a random variable $\boldsymbol{G}$ uniformly distributed on $G(n,d)$.

\subsection{The random variable $\boldsymbol{U}$} \label{subsec-U}

Define the ``isomorphism'' equivalence relation $\sim$ on $G(n,d)$ by
\begin{align} \label{rm-e1}
G_1 \sim G_2 \Leftrightarrow & \text{ there exists } \sigma\in \mathbb{S}_n  \text{ such that for all } e\in \binom{[n]}{2}, \\
& \ e\in E_{G_1} \text{ if and only if } \sigma(e)\in E_{G_2}. \nonumber
\end{align}
Namely, $G_1\sim G_2$ if and only if $G_1$ and $G_2$ are isomorphic as graphs. We fix, once and for all, a ``representation" map $r\colon G(n,d)\to G(n,d)$ such that for all $G_1,G_2\in G(n,d)$ with $G_1\sim G_2$, we have $G_1\sim r(G_1)=r(G_2)$, and we define
\begin{equation} \label{rm-e2}
\boldsymbol{U}:=r(\boldsymbol{G}).
\end{equation}

\subsection{The random variables $\boldsymbol{G}_\ell$ and $\boldsymbol{U}_\ell$} \label{subsec-Gl}

Let $\ell \leqslant \frac{dn}{2}$ be a positive integer parameter. Let $\boldsymbol{L}$ be a random variable such that, conditioned on any realization $G$ of $\boldsymbol{G}$, $\boldsymbol{L}$ is uniformly distributed on $\binom{E_G}{\ell}$, and set
\begin{equation} \label{rm-e3}
\boldsymbol{G}_{\ell} := \boldsymbol{L} \ \ \ \text{ and } \ \ \ \boldsymbol{G}_{-\ell} := E_{\boldsymbol{G}}\setminus \boldsymbol{L}.
\end{equation}
Namely, $\boldsymbol{G}_{\ell}$ is obtained by randomly selecting $\ell$ many edges of $\boldsymbol{G}$, while $\boldsymbol{G}_{-\ell}$ is obtained by randomly deleting $\ell$ many edges of $\boldsymbol{G}$. Notice that $\boldsymbol{G}_{\frac{dn}{2}-\ell}$ is equal in distribution to $\boldsymbol{G}_{-\ell}$.

Respectively, let $\boldsymbol{N}$ be a random variable such that, conditioned on any realization $U$ of $\boldsymbol{U}$, $\boldsymbol{N}$ is uniformly distributed on~$\binom{E_U}{\ell}$, and set
\begin{equation} \label{rm-e4}
\boldsymbol{U}_{\ell} := \boldsymbol{N} \ \ \ \text{ and } \ \ \ \boldsymbol{U}_{-\ell} := E_{\boldsymbol{U}}\setminus \boldsymbol{N}.
\end{equation}
Again, observe that  $\boldsymbol{U}_{\frac{dn}{2}-\ell}$ is equal in distribution to $\boldsymbol{U}_{-\ell}$.

\subsection{The random variables $\boldsymbol{H}$ and $\boldsymbol{H}_\ell$} \label{subsec-H}

As before, let $\ell \in \big[\frac{dn}{2}\big]$. We fix a random variable $\boldsymbol{\pi}$ that is uniformly distributed on $\mathbb{S}_n$ and independent of $\boldsymbol{U}$, $\boldsymbol{U}_\ell$ and $\boldsymbol{U}_{-\ell}$, and we set
\begin{equation} \label{rm-e5}
\boldsymbol{H}:= \boldsymbol{\pi}(\boldsymbol{U}), \ \ \
\boldsymbol{H}_\ell:= \boldsymbol{\pi}(\boldsymbol{U}_\ell) \ \ \ \text{ and } \ \ \
\boldsymbol{H}_{-\ell}:= \boldsymbol{\pi}(\boldsymbol{U}_{-\ell}).
\end{equation}
We will need the following lemma.

\begin{lemma}\label{rm-l1}
For every $\ell \in \big[\frac{dn}{2}\big]$, the pair $(\boldsymbol{G},\boldsymbol{G}_\ell)$ is equal in distribution to the pair~$(\boldsymbol{H},\boldsymbol{H}_\ell)$. Consequently, the pair $(\boldsymbol{G},\boldsymbol{G}_{-\ell})$ is equal in distribution to the pair $(\boldsymbol{H},\boldsymbol{H}_{-\ell})$.
\end{lemma}

\begin{proof}
Fix $G_0\in G(n,d)$, set
\[ \mathrm{Hom}(G_0):=\big\{G\in G(n,d)\colon G\sim G_0\big\} \ \ \ \text{ and } \ \ \
\mathrm{Stab}(G_0):=\big\{ \sigma\in \mathbb{S}_n\colon \sigma(G_0)= G_0\big\},\]
and observe that $|\mathrm{Hom}(G_0)| = |\mathbb{S}_n|/|\mathrm{Stab}(G_0)|$. Hence,
\[ \mathbb{P}\big[\boldsymbol{H}=G_0\big] = \frac{|\mathrm{Hom}(G_0)|}{|G(n,d)|} \cdot \frac{|\mathrm{Stab}(G_0)|}{|\mathbb{S}_n|}=
\frac{1}{|G(n,d)|} = \mathbb{P}\big[\boldsymbol{G}=G_0\big]; \]
that is, $\boldsymbol{G}$ is equal in distribution to $\boldsymbol{H}$.

Next let $W_0$ be a subgraph of $G_0$ with exactly $\ell$ edges. Notice that
\begin{equation} \label{rm-e6}
\mathbb{P}\big[(\boldsymbol{G},\boldsymbol{G}_\ell)=(G_0,W_0)\big] = \frac{1}{|G(n,d)|} \cdot
\frac{1}{\binom{dn/2}{\ell}}.
\end{equation}
On the other hand, since $\boldsymbol{G}$ is equal in distribution to $\boldsymbol{H}$, we have
\begin{equation} \label{rm-e7}
\mathbb{P}\big[(\boldsymbol{H},\boldsymbol{H}_\ell)=(G_0,W_0)\big] =
\mathbb{P}\big[\boldsymbol{H}_\ell= W_0 \, \big| \, \boldsymbol{H}=G_0\big] \cdot \frac{1}{|G(n,d)|}.
\end{equation}
Finally, observe that
\begin{align} \label{rm-e8}
\mathbb{P}\big[\boldsymbol{H}_\ell & = W_0 \, \big| \, \boldsymbol{H}=G_0\big] =
\sum_{\substack{(U,\pi)\\\pi(U)=G_0}}
\mathbb{P}\big[\boldsymbol{H}_\ell= W_0 \, \big| \, (\boldsymbol{U},\boldsymbol{\pi})=(U,\pi)\big] \cdot
\frac{\mathbb{P}\big[(\boldsymbol{U},\boldsymbol{\pi})=(U,\pi)\big]}{\mathbb{P}\big[\boldsymbol{H}=G_0\big]} \\
& \stackrel{\eqref{rm-e5}}{=} \sum_{\substack{(U,\pi)\\\pi(U)=G_0}}
\mathbb{P}\big[ \boldsymbol{U}_\ell= \pi^{-1}(W_0) \, \big| \, (\boldsymbol{U},\boldsymbol{\pi})=(U,\pi)\big] \cdot
\frac{\mathbb{P}\big[(\boldsymbol{U},\boldsymbol{\pi})=(U,\pi)\big]}{\mathbb{P}\big[\boldsymbol{H}=G_0\big]} = \frac{1}{\binom{dn/2}{\ell}}. \nonumber
\end{align}
By \eqref{rm-e6}--\eqref{rm-e8}, the result follows.
\end{proof}

\section{Main technical results} \label{sec-technical}

The technical basis of Theorem \ref{random-poincare} is a compression procedure for maps $f\colon V_{\boldsymbol{G}} \to M$. The procedure relies on tracking only the images of a small number of vertices called \textit{seeds}. This section presents the main technical results and definitions needed to analyze this compression procedure.

\subsection{Seeds, and functions taking values on seeds} \label{subsec-seeds}

In what follows, we allow some functions to assign the value ``$\square$'', an auxiliary element that is our notation for ``undefined''. This allows for the formal construction of functions that may fail to assign a value to some inputs.

Let $G$ be a graph on $[n]$ (not necessarily regular), and let $m,k$ be positive integers with $m,k\leqslant n$. We view the vertices in  $[k]$ as ``seeds". We define a function $g_{G,m,k}\colon [n]\to [k]\cup\{\square\}$ as follows. Let $v\in [n]$ be arbitrary.
\begin{enumerate}
\item If $\partial B_G(v,m)\cap [k]\neq \emptyset$ (that is, there is a vertex in $[k]$ which is at distance from $v$ exactly~$m$), then let $g_{G,m,k}(v)$ be the smallest, with respect to the natural order of $[n]$, element $s\in [k]$ that satisfies $\dist_G(s,v)=m$.
\item Otherwise, if there is no vertex $s\in [k]$ with $\dist_G(s,v)=m$, then set $g_{G,m,k}(v):=\square$.
\end{enumerate}

We also introduce the following definition.

\begin{definition}[Invariance] \label{defn-invariant}
Assume the setting in Section \ref{sec-models}. Let $\mathcal{L}$ be a random $(\boldsymbol{U},\boldsymbol{\pi})$-measurable set of unordered pairs of elements of\, $[n]$. We say that $\mathcal{L}$ is \emph{invariant} if, almost everywhere,
\begin{equation} \label{e1-invariance}
\mathcal{L}(\boldsymbol{U},\boldsymbol{\pi})= \boldsymbol{\pi}\big(\mathcal{L}(\boldsymbol{U},\mathrm{Id})\big),
\end{equation}
where $\mathrm{Id}$ denotes the identity permutation on $[n]$.
\end{definition}

The following proposition is the most technically demanding result of this article. Its proof is given in Section \ref{sec-prop-5.2-new}. Recall that a \emph{perfect matching} of a set $X$ is a collection of pairwise vertex-disjoint edges of the complete graph on $X$, such that every element of $X$ belongs to some edge.

\begin{proposition} \label{prop-seeds-and-conc-final}
For every positive integer $d\geqslant 3$ there exists $K_1=K_1(d)\geqslant 1$ with the following property. Let $K\geqslant K_1$ and let $m,n$ be integers with
\begin{equation} \label{eq:068a}
K \leqslant m\leqslant  K ^{-1}\log_{d-1} n.
\end{equation}
Set $\ell_0:=\big\lfloor\frac{d n}{K m}\big\rfloor$ and\, $k_0:=\big\lfloor\frac{K n}{(d-1)^m}\big\rfloor$. Assume the setting in Section \ref{sec-models}, and let $\mathcal{L}$ be a random $(\boldsymbol{U},\boldsymbol{\pi})\text{-measurable}$ invariant (in the sense of Definition \ref{defn-invariant}) set of unordered pairs of vertices such that almost everywhere,
\begin{equation} \label{eq:069a}
|\mathcal{L}|\leqslant  K ^{-1} (d-1)^m\,n.
\end{equation}
Finally, define the event $\mathcal{X} := \big[|\mathcal{T}(\boldsymbol{U},m)|\geqslant n - \sqrt{n}\big]$, where $\mathcal{T}(\boldsymbol{U},m)$ is as in \eqref{eq-tree}. Then, there exists a $(\boldsymbol{U},\boldsymbol{U}_{-\ell_0}, \boldsymbol{\pi})$-measurable set\, $\mathcal{O}$ of vertices such that, for every realization $U$ of\, $\boldsymbol{U}$ on $\mathcal{X}$, conditioned on the event $[\boldsymbol{U}=U]$, with probability at least $1 - \frac{3}{\sqrt[4]{K}}$, the following hold.
\begin{enumerate}
\item [(i)] \label{prop-i} For every $v\in\mathcal{O}$,
\begin{enumerate}
\item[(i.1)] \label{prop-i.1} $\mathrm{deg}_{\boldsymbol{H}_{-\ell_0}}(v)=d-1$,
\item[(i.2)] \label{prop-i.2} $g_{\boldsymbol{H},m,k_0}(v)=g_{\boldsymbol{H}_{-\ell_0},m,k_0}(v)\neq\square$,
\item[(i.3)] \label{prop-i.3} $\big\{v,g_{\boldsymbol{H},m,k_0}(v)\big\}\not\in \mathcal{L}$, and
\item[(i.4)] \label{prop-i.4} $\big|\big\{w\in\mathcal{O} \colon g_{\boldsymbol{H}_{-\ell_0},m,k_0}(w)=g_{\boldsymbol{H}_{-\ell_0},m,k_0}(v)\big\}\big| \leqslant \frac{(d-1)^m}{m}$.
\end{enumerate}
\item[(ii)] \label{prop-ii} The set of edges $E_{\boldsymbol{H}}\setminus E_{\boldsymbol{H}_{-\ell_0}}$ contains a perfect matching of the set\, $\mathcal{O}$; moreover, $\mathcal{O}$ contains no edges of\, $\boldsymbol{H}_{-\ell_0}$, that is, $\binom{\mathcal{O}}{2}\cap E_{\boldsymbol{H}_{-\ell_0}}=\emptyset$.
\item[(iii)] \label{prop-iii} We have $|\mathcal{O}|\geqslant 0.6 \ell_0$.
\end{enumerate}
\end{proposition}

\subsection{Seeds, and concentrated functions} \label{subsec-concentrated}

We will need the following result that shows that concentrated functions, in the sense of Definition \ref{def3.2}, have the same statistics when restricted on typical, relatively large, sets.

\begin{lemma} \label{lemma-5.2}
Let $0<\varepsilon\leqslant \frac{1}{31}$, and let $n\geqslant k\geqslant \frac{2}{\varepsilon}$ be integers. Let $\mathcal{M}=(M,\varrho)$ be a metric space, let $f\colon [n]\to M$, and assume that $f$ is $(5,1,\varepsilon)$-concentrated according to Definition~\ref{def3.2}. Let $\boldsymbol{\pi}$ be a random permutation uniformly distributed on\, $\mathbb{S}_n$, and define the random function $\boldsymbol{f}\colon [n] \to M$ by setting $\boldsymbol{f}:= f\circ \boldsymbol{\pi}^{-1}$. (Note that, by construction, $\ave_{\mathcal{M}}(\boldsymbol{f},1)=\ave_{\mathcal{M}}(f,1)$ almost everywhere.) Then we have
\begin{equation} \label{e1-lemma-5.2}
\mathbb{P}\bigg[ \Big|\Big\{ \{v,u\}\in \binom{[k]}{2}\colon \varrho\big(\boldsymbol{f}(v),\boldsymbol{f}(u)\big) \geqslant
\frac{\ave_{\mathcal{M}}(f,1)}{5}\Big\}\Big| \geqslant (1-2\varepsilon)\binom{k}{2} \bigg] \geqslant 1-\frac{15}{\varepsilon^2\, k}.
\end{equation}
\end{lemma}

\begin{proof}
Define the set of ``good'' pairs of vertices by
\begin{equation} \label{e2-lemma-5.2}
\mathcal{G}:= \bigg\{ \{v,u\}\in \binom{[n]}{2}\colon \varrho\big(f(v),f(u)\big) \geqslant \frac{\ave_{\mathcal{M}}(f,1)}{5}\bigg\}.
\end{equation}
Since $f$ is $(5,1,\varepsilon)$-concentrated, by the definition of $Q_{\varepsilon}(f)$, we see that $|\mathcal{G}|\geqslant (1-\varepsilon) \binom{n}{2}$.

Let $\boldsymbol{R}$ be a random variable uniformly distributed on $\binom{[n]}{k}$; it is enough to show that
\begin{equation} \label{e3-lemma-5.2}
\mathbb{P}\bigg[ \Big|\mathcal{G} \cap \binom{[\boldsymbol{R}]}{2}\Big| \geqslant (1-2\varepsilon)\binom{k}{2} \bigg] \geqslant 1-\frac{15}{\varepsilon^2\, k}.
\end{equation}
Let $X:=\sum_{\{v,u\}\in\mathcal{G}} \mathbbm{1}_{\boldsymbol{R}}(v)\mathbbm{1}_{\boldsymbol{R}}(u)$ denote the random variable that counts the number of good pairs of vertices that are contained in $\boldsymbol{R}$. Then, a straightforward computation shows that
\begin{equation} \label{e4-lemma-5.2}
\binom{k}{2} \geqslant \mathbb{E}[X]\geqslant (1-\varepsilon) \binom{k}{2} \ \ \ \text{ and } \ \ \
\mathbb{E}[X^2]\leqslant \mathbb{E}[X]^2 \, \Big(1+\frac{15}{k}\Big).
\end{equation}
Therefore, by Chebyshev's inequality,
\begin{equation} \label{e5-lemma-5.2}
\mathbb{P}\bigg[ \Big|\mathcal{G} \cap \binom{[\boldsymbol{R}]}{2}\Big| < (1-2\varepsilon)\binom{k}{2} \bigg] \leqslant
\mathbb{P}\bigg[ \big|X-\mathbb{E}[X]\big|\geqslant \varepsilon \binom{k}{2}\bigg] \leqslant \frac{15}{\varepsilon^2\, k}. \qedhere
\end{equation}
\end{proof}

\subsection{Matchings} \label{subsec-matchings}

For every nonempty finite set $X$ with even cardinality, let $\mathrm{M}(X)$ denote the set of all perfect matchings of $X$, and let $\mathbb{P}_{\mathrm{M}(X)}$ denote the uniform probability measure on $\mathrm{M}(X)$.

We will need the following result that expresses the fact that a uniformly random perfect matching of a set $X$ is unlikely to avoid a given large set of pairs of $X$.

\begin{lemma} \label{lemma-matchings}
Let $\ell=2r\geqslant 4$ be an even positive integer, and let\, $0<c\leqslant \varepsilon \leqslant \frac18$. Also let\, $Y\subseteq \binom{[\ell]}{2}$ with\, $|Y|\geqslant (1-\varepsilon)\, \binom{\ell}{2}$. Then,
\begin{equation} \label{match-e1}
\mathbb{P}_{\mathrm{M}([\ell])}\big[ \mu\colon |\mu\cap Y|\leqslant cr\big] \leqslant \exp\bigg(-\Big(\frac{1-2c}{4}\Big) \cdot \log\Big(\frac{1-2c}{16e}\cdot \frac{1}{\varepsilon}\Big) \cdot \ell\bigg).
\end{equation}
\end{lemma}

\begin{proof}
Let $(\boldsymbol{\xi}_1,\dots,\boldsymbol{\xi}_r) \in \binom{[\ell]}{2}^r$ be a random vector that is uniformly distributed on the set of perfect matchings of $[\ell]$. Observe that, setting $Y^\complement := \binom{[\ell]}{2}\setminus Y$,
\begin{align} \label{match-e2}
\mathbb{P}_{\mathrm{M}([\ell])}\big[ \mu\colon & |\mu\cap Y|\leqslant cr\big] =
\mathbb{P}\Big[ \sum_{i=1}^r \mathbbm{1}_Y(\boldsymbol{\xi}_i) \leqslant cr\Big]\leqslant
\mathbb{P}\bigg[ \sum_{i=1}^{\lceil r/2\rceil} \mathbbm{1}_Y(\boldsymbol{\xi}_i) \leqslant cr\bigg] \\
& = \mathbb{P}\bigg[ \sum_{i=1}^{\lceil r/2\rceil} \mathbbm{1}_{Y^\complement}(\boldsymbol{\xi}_i) > \lceil r/2\rceil - cr\bigg]
\leqslant  \mathbb{P}\bigg[ \sum_{i=1}^{\lceil r/2\rceil} \mathbbm{1}_{Y^\complement}(\boldsymbol{\xi}_i) \geqslant \lceil r/2\rceil \cdot (1-2c)\bigg]. \nonumber
\end{align}
Moreover, since $\ell \geqslant 4$, for every $i\in \big[\lceil r/2\rceil\big]$,
\begin{equation} \label{match-e3}
\mathbb{P}\big[ \boldsymbol{\xi}_i\in Y^\complement\big] \leqslant \frac{\varepsilon \cdot \binom{\ell}{2}}{\binom{\ell-2(i-1)}{2}}
\leqslant \frac{\varepsilon \cdot \binom{\ell}{2}}{\binom{\ell-r}{2}}=
\frac{\varepsilon \cdot \binom{\ell}{2}}{\binom{\ell/2}{2}}= 4\varepsilon \cdot \frac{\ell}{\ell-2}\leqslant 8\varepsilon.
\end{equation}
Finally, note that the random variable $\sum_{i=1}^{\lceil r/2\rceil} \mathbbm{1}_{Y^\complement}(\boldsymbol{\xi}_i)$ is stochastically dominated\footnote{Recall that a (real-valued) random variable $X$ is \emph{stochastically dominated} by a (real-valued) random variable $Y$ if $\mathbb{P}[X\geqslant t]\leqslant \mathbb{P}[Y\geqslant t]$, for all $t\in\mathbb{R}$.} by the binomial $\mathrm{Bi}\big(\lceil r/2\rceil, 8\varepsilon\big)$. Thus, setting $k:=\big\lceil (1-2c)\lceil r/2\rceil\big\rceil$, we have
\begin{align} \label{match-e4}
\mathbb{P}\bigg[ \sum_{i=1}^{\lceil r/2\rceil} & \mathbbm{1}_{Y^\complement}(\boldsymbol{\xi}_i) \geqslant \lceil r/2\rceil \cdot (1-2c)\bigg] \leqslant \mathbb{P}\Big[ \mathrm{Bi}\big(\lceil r/2\rceil, 8\varepsilon\big)\geqslant \lceil r/2\rceil \cdot (1-2c)\Big] \\
& \leqslant \binom{\lceil r/2\rceil}{k} \cdot (8\varepsilon)^k \leqslant \Big( \frac{\lceil r/2\rceil \cdot e}{k} \cdot 8\varepsilon\Big)^k
\leqslant \exp\bigg(-\Big(\frac{1-2c}{4}\Big) \cdot \log\Big(\frac{1-2c}{16e}\cdot \frac{1}{\varepsilon}\Big) \cdot \ell\bigg). \nonumber
\end{align}
The proof is completed by combining \eqref{match-e2}--\eqref{match-e4}.
\end{proof}

\section{Small metric spaces and well-conditioned metric spaces} \label{sec-dichotomy}

The proof of Theorem \ref{random-poincare} is sensitive to the structure of a given metric space $\mathcal{M}$. In this section, we prove Theorem \ref{random-poincare} in two special settings. First, we give a proof in the case where $\mathcal{M}$ has small cardinality. Second, we consider the setting where $\mathcal{M}$ has moderate cardinality, and at most exponential {\it aspect ratio}.

\begin{definition}[Well-conditioned metric spaces]  \label{def-regular-metric}
Let $\mathcal{M}=(M,\varrho)$ be a finite metric space, let
\begin{equation} \label{e-diam-m}
\mathrm{diam}(\mathcal{M}):=\max\big\{\varrho(x,y)\colon x,y\in M\big\} \ \text{ and } \
m(\mathcal{M}):=\min\big\{\varrho(x,y)\colon x,y\in M \text{ and } x\neq y\big\}
\end{equation}
denote the diameter and the minimal distance between distinct points of $\mathcal{M}$, respectively, and define the \emph{aspect ratio} of $\mathcal{M}$ by setting
\begin{equation} \label{eq-aspect-ratio}
a(\mathcal{M}):= \frac{\mathrm{diam}(\mathcal{M})}{m(\mathcal{M})}.
\end{equation}
We say that the metric space $\mathcal{M}$ is \emph{well-conditioned} if\, $a(\mathcal{M})\leqslant \exp\big(|M|\big)$.
\end{definition}

\begin{remark}
Notice that every connected graph equipped with the shortest-path distance is a $\text{well-conditioned}$ metric space in the sense of Definition \ref{def-regular-metric}. Also observe that if $\mathrm{diam}(\mathcal{M})\leqslant 1$, then $\mathcal{M}=(M,\varrho)$ is well-conditioned if $\varrho(x,y)\geqslant e^{-|M|}$, for all distinct $x,y\in M$. Thus, in a $\text{well-conditioned}$ metric space with bounded diameter, there is no pair of points having exponentially small distance.
\end{remark}

As already noted, our first theorem establishes Theorem \ref{random-poincare} under a restriction on the size of~$\cM$.

\begin{theorem}[Small metric spaces] \label{random-poincare-2}
For every integer $d\geqslant 3$, there exist constants $C_d\geqslant 1$ and $\tau>0$, that depend only on $d$, such that for any metric space $\mathcal{M}=(M,\varrho)$ with $N:=|M|\geqslant d$ and any integer $n>d$ with $dn$ even,
\begin{equation} \label{rp-e1}
\mathbb{P}_{G(n,d)}\Big[G\colon \gamma(G,\varrho)\leqslant C_d\, \log_{d-1}\log_{d-1} N\Big] \geqslant
1- O_d\Big(\frac{1}{n^{\tau}}\Big) - O_d\bigg( \exp\Big( - \frac{n}{N^2}\Big)\bigg).
\end{equation}
\end{theorem}

\begin{remark} \label{rem8.4}
Note that the estimate in \eqref{rp-e1} is non-trivial only if $n\geqslant N^2$. Also observe that \eqref{rp-e1} yields the claimed bound in Theorem \ref{random-poincare} as long as $n\geqslant N^{2+\Theta(1)}$, that is, when the size of $\mathcal{M}$ is relatively small compared with the size of the vertex set of the random graph.
\end{remark}

The next result complements Theorem \ref{random-poincare-2} by expanding the range of $n,N$ considered, under the additional assumption that $\cM$ is well-conditioned. The assumption on the aspect ratio is removed in Section \ref{reduction}.

\begin{theorem}[Well-conditioned metric spaces] \label{regular-random-poincare}
For every integer $d\geqslant 3$, there exist constants $C_d, A\geqslant 1$ and $\tau>0$, that depend only on $d$, such that for any well-conditioned metric space $\mathcal{M}=(M,\varrho)$ with $N:=|M|\geqslant d$ and any integer $n>d$ with $dn$ even,
\begin{equation} \label{regular-e1}
\mathbb{P}_{G(n,d)}\Big[G\colon \gamma(G,\varrho)\leqslant C_d\, \log_{d-1}\log_{d-1} N\Big] \geqslant
1- O_d\Big(\frac{1}{n^{\tau}}\Big) - O_d\bigg( \exp\Big( - \frac{n}{(\log_{d-1}N)^A}\Big)\bigg).
\end{equation}
\end{theorem}

\begin{remark} \label{rem8.6}
Again observe that \eqref{regular-e1} is non-trivial only if $n\geqslant (\log_{d-1} N)^A$; moreover, in the regime $n\geqslant (\log_{d-1} N)^{A+\Theta(1)}$, the bound in \eqref{regular-e1} implies the claimed bound in Theorem \ref{random-poincare}.
\end{remark}

\begin{remark} \label{remark-new}
The requirement in Theorems \ref{random-poincare-2} and \ref{regular-random-poincare} that $N>d$ is a technical convenience. Indeed, if $3\leqslant N\leqslant d-1$, then, by Theorem \ref{thm-bourgain} and Corollary \ref{Friedman}, it is easy to see that
\[ \mathbb{P}_{G(n,d)}\big[G\colon \gamma(G,\varrho) \leqslant c\, \log d \big]\geqslant 1-O_d\Big(\frac{1}{n^{\tau_1}}\Big), \]
where $\tau_1=\tau_1(d)>0$ and $c>0$ is a universal constant.
\end{remark}

\subsection{Organization}

Theorem \ref{random-poincare-2} is proven imminently in Subsection \ref{proof-random-poincare-2}; the proof of Theorem~\ref{regular-random-poincare} is given in Subsection \ref{proof-regular-random-poincare}. Although Theorem \ref{regular-random-poincare} is quantitatively stronger than Theorem \ref{random-poincare-2}, its proof follows the same steps as the proof of Theorem \ref{random-poincare-2}. Thus, a full proof of Theorem \ref{random-poincare-2} is given, and then the minor adjustments needed to prove Theorem \ref{regular-random-poincare} are detailed in Subsection \ref{proof-regular-random-poincare}.

As already mentioned, in the subsequent Section \ref{reduction}, the condition on the aspect ratio in Theorem~\ref{regular-random-poincare} will be removed via a reduction argument, establishing Theorem \ref{random-poincare} for metric spaces of small and moderate cardinality. The remaining cases of metric spaces of huge cardinality follow easily from Theorem~\ref{random-poincare-2}; the argument is given in Section \ref{sec-proof-random-poincare}. This will complete the proof of Theorem \ref{random-poincare}.

\subsection{Proof of Theorem \ref{random-poincare-2}} \label{proof-random-poincare-2}

In what follows, besides $C_d$ and $\tau$, we also use three auxiliary parameters $C\geqslant K \geqslant 1$ and $\varepsilon>0$ whose exact value will be determined in the course of the proof and will eventually depend \emph{only} on the degree $d$. The parameter $K$ will be large enough in terms of $d$, the parameter $C$ will be large enough in terms of $d,K$ and the parameter $\varepsilon^{-1}$ will be large enough in terms of $d, K, C$; finally, the constant $C_d$ will be selected large enough in terms of $d, K, C$ and~$\varepsilon^{-1}$. An observation that will be frequently used is that while our choices of $d$, $K$ and $C$ are coupled, the tuple $(d,K,C)$ as a whole depends only on $d$.

We proceed to the details. As we have already pointed out in Remark \ref{rem8.4}, the desired estimate~\eqref{rp-e1} is non-trivial only if $n\geqslant N^2$. Thus, in what follows, we may assume that $n$ is sufficiently large in terms of $d$ and satisfies $n\geqslant N^2$. We start by setting
\begin{equation} \label{rp-e2}
m := \big\lfloor C \, \log_{d-1}\log_{d-1} N \big\rfloor.
\end{equation}
Notice that if $n$ is sufficiently large in terms of $d$ and satisfies $n\geqslant N^2$, then
\begin{equation} \label{rp-enew}
m\leqslant \frac{1}{25} \log_{d-1}n.
\end{equation}
(Let us record, for later usage in the proof of Theorem \ref{regular-random-poincare}, that $n \geqslant (\log_{d-1}N)^{25\,C}$ also suffices in place of $n \geqslant N^2$). Define the $\boldsymbol{U}$-measurable events
\begin{equation} \label{rp-e3}
\mathcal{X}:= \big[|\mathcal{T}(\boldsymbol{U},m)|\geqslant n - \sqrt{n}\big]
\ \ \ \text{ and } \ \ \
\mathcal{Y} := \big[ h(\boldsymbol{U}) \geqslant 0.005d \big],
\end{equation}
where $\mathcal{T}(\boldsymbol{U},m)$ is as in \eqref{eq-tree}. By \eqref{rp-enew}, Proposition \ref{prop-tree} and Corollary \ref{Friedman}, there is $\tau_1=\tau_1(d)>0$ so that, if $n$ is sufficiently large in terms of $d$ and satisfies $n\geqslant N^2$,
\begin{equation} \label{rp-e4}
\mathbb{P}\big[\mathcal{X}\cap \mathcal{Y}\big] \geqslant 1-O_d\Big(\frac{1}{n^{\tau_1}}\Big).
\end{equation}
Define the $\boldsymbol{U}$-measurable event
\begin{equation}
\label{rp-e5} \mathcal{E}_1 := \big[\gamma(\boldsymbol{U},\varrho)> C_d\, \log_{d-1}\log_{d-1} N\big]
\cap (\mathcal{X} \cap \mathcal{Y}).
\end{equation}
By \eqref{rp-e4}, it is enough to show that if $n$ is sufficiently large in terms of $d$ and satisfies $n\geqslant N^2$, then
\begin{equation} \label{eq-goal}
\mathbb{P}\big[\mathcal{E}_1\big]\leqslant \exp\big( - \frac{n}{N^2}\big).
\end{equation}

\subsection*{Step 1: obtaining potential counterexamples} \label{rp-sec1}

We define a random  $\boldsymbol{U}$-measurable function $f_{\boldsymbol{U}}\colon [n] \to M$ as follows. For every realization $U$ of\, $\boldsymbol{U}$ on the event $\mathcal{E}_1$, let $f_U\colon [n]\to M$ be a function such that
\begin{equation}\label{eq:005b}
\frac{\ave_{\mathcal{M}}(f_U,1)}{C_d\, \log_{d-1}\log_{d-1} N}> \frac{2}{dn}\sum_{\{v,u\}\in E_U} \varrho\big(f_U(v),f_U(u)\big);
\end{equation}
such a function exists since $\gamma(U,\varrho)> C_d\, \log_{d-1}\log_{d-1} N$. If $U$ is a realization of $\boldsymbol{U}$ on $\mathcal{E}_1^\complement$, then let $f_U\colon [n]\to M$ be any constant function.

Note that if $U$ is any realization of $\boldsymbol{U}$ on the event $\mathcal{E}_1$, then $h(U)\geqslant 0.005d$. Hence, if $n$ is sufficiently large in terms of $d$ and the function $f_U$ is not $(5,1,\varepsilon)$-concentrated in the sense of Definition \ref{def3.2}, then, by Proposition \ref{prop:non-concentrated}, $f_U$ must satisfy Poincar\'{e} inequality with constant $O_{d,\varepsilon}(1)$. Thus, selecting
\begin{enumerate}
\item[(C1)] \label{rp-C1} $C_d$ large enough in terms of $d$ and $\varepsilon^{-1}$,
\end{enumerate}
we have that
\begin{equation} \label{rp-e7}
f_{\boldsymbol{U}} \text{ is $(5,1,\varepsilon)$-concentrated almost everywhere on the event } \mathcal{E}_1.
\end{equation}

We also define the random $(\boldsymbol{U},\boldsymbol{\pi})$-measurable function $\boldsymbol{f}\colon [n]\to M$ by setting
\begin{equation} \label{eq:006b}
\boldsymbol{f} := f_{\boldsymbol{U}}\circ \boldsymbol{\pi}^{-1}.
\end{equation}
Observe that, by \eqref{eq:005b}, almost everywhere on the event $\mathcal{E}_1$,
\begin{align} \label{eq:007b}
\frac{\ave_{\mathcal{M}}(\boldsymbol{f},1)}{C_d\, \log_{d-1}\log_{d-1} N}
& = \frac{\ave_{\mathcal{M}}(f_{\boldsymbol{U}},1)}{C_d\, \log_{d-1}\log_{d-1} N}
> \frac{2}{dn}\sum_{\{v,u\}\in E_{\boldsymbol{U}}} \varrho\big(f_{\boldsymbol{U}}(v),f_{\boldsymbol{U}}(u)\big)\\
& = \frac{2}{dn}\sum_{\{v,u\}\in E_{\boldsymbol{U}}}
\varrho\Big(\boldsymbol{f}\big(\boldsymbol{\pi}(v)\big),\boldsymbol{f}\big(\boldsymbol{\pi}(u)\big)\Big)
=\frac{2}{dn}\sum_{\{v,u\}\in E_{\boldsymbol{H}}} \varrho\big(\boldsymbol{f}(v),\boldsymbol{f}(u)\big). \nonumber
\end{align}

\subsection*{Step 2: application of Proposition \ref{prop-seeds-and-conc-final}} \label{rp-sec2}

Let $K_1=K_1(d)\geqslant 1$ be as in Proposition \ref{prop-seeds-and-conc-final}, and select the parameter $K$ so that
\begin{enumerate}
\item[(C2)] \label{rp-C2} $K_1 \leqslant K \leqslant C$.
\end{enumerate}
Observe that, if $n$ is sufficiently large in terms of $d$ with $n\geqslant N^2$, then
\begin{equation} \label{rp-e8}
K \leqslant m\leqslant K^{-1} \log_{d-1}n,
\end{equation}
where $m$ is as in \eqref{rp-e2}. (We again record, for future usage in the proof of Theorem \ref{regular-random-poincare}, that \eqref{rp-e8} is also satisfied if $n\geqslant (\log_{d-1}N)^A$ with $A\geqslant CK$, rather than $n \geqslant N^2$.) Also define
\begin{gather}
\label{eq:014b} \ell_0:=\Big\lfloor\frac{d n}{ K  m}\Big\rfloor \ \ \ \text{ and } \ \ \
k_0:=\Big\lfloor\frac{ K  n}{(d-1)^m}\Big\rfloor, \\
\label{eq:009b} \mathcal{L} := \bigg\{\{v,u\}\in\binom{[n]}{2}\colon \mathrm{dist}_{\boldsymbol{H}}(v,u)=m \text{ and }
\varrho\big(\boldsymbol{f}(v),\boldsymbol{f}(u)\big)> \frac{\ave_{\mathcal{M}}(\boldsymbol{f},1)}{\sqrt{C_d}}\bigg\},
\end{gather}
where $\boldsymbol{f}$ is as in \eqref{eq:006b}. Note that $\mathcal{L}$ is $(\boldsymbol{U},\boldsymbol{\pi})$-measurable and invariant in the sense of Definition~\ref{defn-invariant}.

\begin{claim} \label{cl:7.2}
If $n$ is large enough with respect to $d$, then, almost everywhere,
\begin{equation}\label{eq:010b}
|\mathcal{L}| \leqslant n\frac{(d-1)^m}{K}.
\end{equation}
\end{claim}

\begin{proof}[Proof of Claim \ref{cl:7.2}]
Since $|\mathcal{L}|=0$ almost everywhere on the event $\mathcal{E}_1^\complement$, it is enough to show \eqref{eq:010b} almost everywhere on the event $\mathcal{E}_1$. Let $\mathcal{P}_m$ denote the set of simple paths (no vertex is revisited) of length $m$ in $\boldsymbol{H}$. By the $d$-regularity of $\boldsymbol{H}$, for every edge $e\in E_{\boldsymbol{H}}$,
\begin{equation}\label{eq:011b}
\big|\big\{ P\in\mathcal{P}_m \colon e\in P\big\}\big| \leqslant m(d-1)^{m-1}.
\end{equation}
Moreover, almost everywhere of the event $\mathcal{E}_1$, for every $v\in \mathcal{T}(\boldsymbol{H},m)$, the number of paths in $\mathcal{P}_m$ having $v$ as one end-poind is equal to $d(d-1)^{m-1}$. Hence, almost everywhere on the event $\mathcal{E}_1$, if $n$ is large enough in terms of $d$,
\begin{equation}\label{eq:012b}
|\mathcal{P}_m|\geqslant \frac{1}{2}(n-\sqrt{n})d(d-1)^{m-1}\geqslant\frac{n(d-1)^m}{4}.
\end{equation}
By \eqref{eq:007b}, \eqref{eq:011b} and \eqref{eq:012b}, almost everywhere on $\mathcal{E}_1$,
\begin{align*}
& \frac{\ave_{\mathcal{M}}(\boldsymbol{f},1)}{C_d\, \log_{d-1}\log_{d-1} N} >
\frac{2}{dn}\sum_{\{v,u\}\in E_{\boldsymbol{H}}} \varrho\big(\boldsymbol{f}(v),\boldsymbol{f}(u)\big) \\
& \ \ \ \ \ \geqslant \frac{2}{dn m(d-1)^{m-1}}
\sum_{P\in\mathcal{P}_m}\sum_{\{v,u\}\in P} \varrho\big(\boldsymbol{f}(v),\boldsymbol{f}(u)\big)
\geqslant \frac{d-1}{2dm}\cdot\frac{1}{|\mathcal{P}_m|}
\sum_{P\in\mathcal{P}_m}\sum_{\{v,u\}\in P} \varrho\big(\boldsymbol{f}(v),\boldsymbol{f}(u)\big),
\end{align*}
and so, by the choice of $m$ in \eqref{rp-e2}, almost everywhere on the event $\mathcal{E}_1$,
\[ \frac{1}{|\mathcal{P}_m|} \sum_{P\in\mathcal{P}_m}\sum_{\{v,u\}\in P} \varrho\big(\boldsymbol{f}(v),\boldsymbol{f}(u)\big)
< \frac{3C\, \ave_{\mathcal{M}}(\boldsymbol{f},1)}{C_d}. \]
On the other hand, almost everywhere we have $|\mathcal{P}_m|\leqslant nd(d-1)^{m-1}/2\leqslant n(d-1)^m$, due to the $d$-regularity of $\boldsymbol{H}$. By the triangle inequality, Markov's inequality and the choice of $\mathcal{L}$ in \eqref{eq:009b}, we have, almost everywhere on the event $\mathcal{E}_1$,
\begin{equation} \label{eq:016b}
|\mathcal{L}|\leqslant \Big|\Big\{P\in\mathcal{P}_m \colon \sum_{\{v,u\}\in P} \!\!
\varrho\big(\boldsymbol{f}(v),\boldsymbol{f}(u)\big)> \frac{\ave_{\mathcal{M}}(\boldsymbol{f},1)}{\sqrt{C_d}}\Big\}\Big|
\leqslant \frac{3C}{\sqrt{C_d}}\, |\mathcal{P}_m|  \leqslant n\, \frac{(d-1)^m}{\sqrt{C_d}/(3C)}.
\end{equation}
Selecting $C_d$ so that
\begin{enumerate}
\item[(C3)] \label{rp-C3} $(3CK)^2\leqslant C_d$,
\end{enumerate}
we see that \eqref{eq:010b} follows from \eqref{eq:016b}.
\end{proof}

By \eqref{rp-e8} and Claim \ref{cl:7.2}, we are in a position to apply Proposition \ref{prop-seeds-and-conc-final}. More precisely, selecting $K$ so that
\begin{enumerate}
\item[(C4)] \label{rp-C4} $K\geqslant 12^{5}$,
\end{enumerate}
we obtain an event $\mathcal{E}_{2,1}\subseteq\mathcal{E}_1$ with
\begin{equation}\label{eq:015b}
\mathbb{P}\big[\mathcal{E}_{2,1}\big]\geqslant \mathbb{P}\big[\mathcal{E}_1\big] \,
\Big(1 - \frac{3}{\sqrt[4]{K}}\Big)
\geqslant \frac{3}{4}\, \mathbb{P}\big[ \mathcal{E}_1\big],
\end{equation}
and a random set $\mathcal{O}$ of vertices such that, almost everywhere on $\mathcal{E}_{2,1}$, it satisfies properties (\hyperref[prop-i]{i}), (\hyperref[prop-ii]{ii}) and (\hyperref[prop-iii]{iii}) described in Proposition \ref{prop-seeds-and-conc-final}.

\subsection*{Step 3: application of Lemma \ref{lemma-5.2}} \label{rp-sec3}
Select
\begin{enumerate}
\item[(C5)] \label{rp-C5} $0<\varepsilon \leqslant \frac{1}{31}$.
\end{enumerate}
By \eqref{rp-e8} and \eqref{eq:014b}, if $n$ is sufficiently large in terms of $d$, then $k_0\geqslant \frac{60}{\varepsilon^2}\geqslant \frac{2}{\varepsilon}$. By \eqref{rp-e7}, and applying Lemma \ref{lemma-5.2}, we obtain an event $\mathcal{E}_{2,2}\subseteq \mathcal{E}_1$ with
\begin{equation}\label{eq:020b}
\mathbb{P}\big[\mathcal{E}_{2,2}\big] \geqslant
\mathbb{P}\big[ \mathcal{E}_1\big] \, \Big(1-\frac{15}{\varepsilon^2\, k_0}\Big)
\geqslant \frac{3}{4}\, \mathbb{P}\big[\mathcal{E}_1\big]
\end{equation}
such that, almost everywhere on $\mathcal{E}_{2,2}$,
\begin{equation}\label{eq:021b}
\Big|\Big\{ \{v,u\}\in \binom{[k_0]}{2}\colon \varrho\big(\boldsymbol{f}(v),\boldsymbol{f}(u)\big) \geqslant
\frac{\ave_{\mathcal{M}}(\boldsymbol{f},1)}{5}\Big\}\Big| \geqslant (1-2\varepsilon) \binom{k_0}{2},
\end{equation}
where the random function $\boldsymbol{f}$ is as in \eqref{eq:006b}. Set
\begin{equation} \label{eq:022b}
\mathcal{E}_2:= \mathcal{E}_{2,1}\cap\mathcal{E}_{2,2}
\end{equation}
and notice that, by \eqref{eq:015b} and \eqref{eq:020b}, if $n$ is sufficiently large in terms of $d$ with $n\geqslant N^2$, then
\begin{equation}\label{eq:023b}
\mathbb{P}\big[\mathcal{E}_2\big]\geqslant \frac{1}{2}\, \mathbb{P}\big[\mathcal{E}_1\big] .
\end{equation}

\subsection*{Step 4: conditioning on the value of $\ave_{\mathcal{M}}(\boldsymbol{f},1)$} \label{rp-sec4}

In order to obtain the desired bound on~$\mathbb{P}\big[\mathcal{E}_1\big]$, we need to derandomize some characteristics of the random objects constructed so far. This will be achieved in Steps \hyperref[rp-sec4]{4}--\hyperref[rp-sec6]{6}. In this step, we will condition on the value of $\ave_{\mathcal{M}}(\boldsymbol{f},1)$.

More precisely, we first observe that the number of all possible values of $\ave_{\mathcal{M}}(\boldsymbol{f},1)$ is upper bounded by $n^{2N}$. Indeed, notice that the average $\ave_\mathcal{M}(f,1)$ of an arbitrary map $f\colon [n]\to M$ is completely determined by the tuple $(|f^{-1}(x)|)_{x\in M}$, and this is an $N$-tuple of nonnegative integers that sum up to $n$. If $n$ is sufficiently large in terms of $d$ with $n\geqslant N^2$, then the number of $N$-tuples of nonnegative integers that sum up to $n$ is at most
\[ \binom{n+N-1}{N-1} \leqslant \binom{n+N}{N} \leqslant (n+N)^{N} \leqslant n^{2N}, \]
as claimed. On the other hand, by the choices of $m$ and $\ell_0$ in \eqref{rp-e2} and \eqref{eq:014b} respectively, we have for $n$ sufficiently large in terms of $d$ with $n\geqslant N^2$,
\[ n^{2N} = \exp\big(2N\log n\big) \leqslant \exp\Big( \frac{dn}{2KC \log_{d-1}\log_{d-1}N}\Big)
\leqslant \exp\Big( \frac{dn}{2Km}\Big) \leqslant \exp(\ell_0). \]
Thus, there exists $r_0>0$ such that, setting
\begin{equation} \label{eq:024b}
\mathcal{E}_3 := \mathcal{E}_2\cap \big[\ave_\mathcal{M}(\boldsymbol{f},1)=r_0\big],
\end{equation}
we have, if $n$ is sufficiently large in terms of $d$ with $n\geqslant N^2$,
\begin{equation} \label{eq:025b}
\mathbb{P}\big[\mathcal{E}_3\big] \geqslant \mathbb{P}\big[\mathcal{E}_2\big] \, n^{-2N} \geqslant
\mathbb{P}\big[\mathcal{E}_2\big] \, \exp(-\ell_0) \stackrel{\eqref{eq:023b}}{\geqslant}
\frac{1}{2} \, \exp(-\ell_0) \, \mathbb{P}\big[\mathcal{E}_1\big].
\end{equation}

\subsection*{Step 5: conditioning on the values of $\boldsymbol{f}$ on $[k_0]$} \label{rp-sec5}

If $n$ is sufficiently large in terms of $d$ with $n\geqslant N^2$, then, selecting
\begin{enumerate}
\item[(C6)] \label{rp-C6} $C$ sufficiently large in terms of $K$ and $d$,
\end{enumerate}
we have, by the choices of $m, \ell_0$ and $k_0$ in \eqref{rp-e2} and \eqref{eq:014b} respectively,
\[ N^{k_0} = \exp\big(k_0 \log N \big) \leqslant
\exp\bigg(\frac{dK\log(d-1)\, n}{(\log_{d-1} N)^{C-1}} \bigg)
\leqslant \exp\Big( \frac{dn}{2KC \log_{d-1}\log_{d-1}N}\Big) \leqslant \exp(\ell_0), \]
where we have used the fact that $N>d$. Since the number of all maps from $[k_0]$ into $M$ is equal to $N^{k_0}$, there exists a map $\alpha\colon [k_0]\to M$ such that, setting
\begin{equation} \label{eq:026b}
\mathcal{E}_4 := \mathcal{E}_3\cap \big[ \boldsymbol{f}|_{[k_0]} = \alpha\big],
\end{equation}
we have, if $n$ is sufficiently large in terms of $d$ with $n\geqslant N^2$,
\begin{equation} \label{eq:027b}
\mathbb{P}\big[\mathcal{E}_4\big] \geqslant \mathbb{P}\big[\mathcal{E}_3\big] \, N^{-k_0} \geqslant
\mathbb{P}\big[\mathcal{E}_3\big] \, \exp(-\ell_0) \stackrel{\eqref{eq:025b}}{\geqslant}
\frac{1}{2} \, \exp(-2\ell_0) \, \mathbb{P}\big[\mathcal{E}_1\big].
\end{equation}

\subsection*{Step 6: conditioning on the relative position of the set $\mathcal{O}$} \label{rp-sec6}

We need to introduce some pieces of notation. Fix a nonempty subset $B$ of $[2\ell_0]$; for every subset $A$ of $[n]$ with $2\ell_0 \geqslant |A|\geqslant |B|$, let $\{a_1<\dots <a_{|A|}\}$ denote its increasing enumeration, and set
\begin{equation} \label{eq:028b}
r_B(A) := \Big\{ a_i \colon i\in B\cap \big[|A|\big]\Big\}.
\end{equation}
Thus, $r_B(A)$ is a subset of $A$ with cardinality $|B|$ whose relative position (in the usual order) inside $A$ is completely determined by $B$; more precisely, note that for any nonempty subset $\Gamma$~of~$A$, there exists a unique set $B\subseteq [2\ell_0]$ such that $\Gamma=r_B(A)$.

Define a random $\boldsymbol{H}_{-\ell_0}$-measurable set of vertices by
\begin{equation} \label{eq:029b}
\mathcal{D}_{d-1}:=\big\{ v\in [n] \colon \mathrm{deg}_{\boldsymbol{H}_{-\ell_0}}(v)=d-1\big\}.
\end{equation}
Notice that, almost everywhere on the event $\mathcal{E}_4 \subseteq \mathcal{E}_{2,1}$, the set $\mathcal{O}$ is a subset of $\mathcal{D}_{d-1}$; on the other hand, almost everywhere, $|\mathcal{D}_{d-1}|\leqslant 2\ell_0$. Therefore, there exists a subset $B$ of $[2\ell_0]$ with $|B|\geqslant 0.6\, \ell_0$ such that, setting
\begin{equation} \label{eq:030b}
\mathcal{E}_5 := \mathcal{E}_4 \cap \big[ r_B(\mathcal{D}_{d-1}) = \mathcal{O}\big],
\end{equation}
we have, if $n$ is sufficiently large in terms of $d$ with $n\geqslant N^2$,
\begin{equation} \label{eq:031b}
\mathbb{P}\big[\mathcal{E}_5\big] \geqslant \mathbb{P}\big[\mathcal{E}_4\big] \, 2^{-2\ell_0} \geqslant
\mathbb{P}\big[\mathcal{E}_4\big] \, 2\exp(-2\ell_0) \stackrel{\eqref{eq:027b}}{\geqslant}
\exp(-4\ell_0) \, \mathbb{P}\big[\mathcal{E}_1\big].
\end{equation}

\subsection*{Step 7: an $\boldsymbol{H}_{-\ell_0}$-measurable set of pairs of vertices} \label{rp-sec7}

Define a random $\boldsymbol{H}_{-\ell_0}$-measurable set of pairs of vertices by
\begin{equation} \label{def-Q}
\mathcal{Q}:= \bigg\{ \{v,u\}\in\binom{r_B(\mathcal{D}_{d-1})}{2}\colon
\varrho\Big(\alpha\big( g_{\boldsymbol{H}_{-\ell_0},m,k_0}(v)\big),
\alpha\big( g_{\boldsymbol{H}_{-\ell_0},m,k_0}(u)\big) \Big) \geqslant \frac{r_0}{10}\bigg\}.
\end{equation}
We gather, below, several properties of $r_B(\mathcal{D}_{d-1})$ and $\mathcal{Q}$  that are satisfied almost everywhere on the event $\mathcal{E}_5$.

\begin{claim} \label{cl:7.3}
If $n$ is sufficiently large in terms of $d$ with $n\geqslant N^2$, then, almost everywhere on the event $\mathcal{E}_5$, the following hold.
\begin{enumerate}
\item[(1)] \label{pr1} We have $|r_B(\mathcal{D}_{d-1})|\geqslant 0.6\ell_0$.
\item[(2)] \label{pr2} The set of edges $E_{\boldsymbol{H}}\setminus E_{\boldsymbol{H}_{-\ell_0}}$ contains a perfect matching of the set $r_B(\mathcal{D}_{d-1})$; moreover, the set $r_B(\mathcal{D}_{d-1})$ contains no edges of $\boldsymbol{H}_{-\ell_0}$, that is, $\binom{r_B(\mathcal{D}_{d-1})}{2}\cap E_{\boldsymbol{H}_{-\ell_0}}=\emptyset$.
\item[(3)] \label{pr3} We have $|\mathcal{Q}| \geqslant (1-5\varepsilon K^4)\binom{|r_B(\mathcal{D}_{d-1})|}{2}$.
\item[(4)] \label{pr4} We have $\big|\big\{\{v,u\}\in E_{\boldsymbol{H}}\setminus E_{\boldsymbol{H}_{-\ell_0}} \colon \{v,u\}\in \mathcal{Q}\big\}\big| \leqslant \frac{1}{\sqrt[4]{C_d}}\, |r_B(\mathcal{D}_{d-1})|$.
\end{enumerate}
\end{claim}

\begin{proof}[Proof of Claim \ref{cl:7.3}]
Since $\mathcal{E}_5\subseteq\mathcal{E}_{2,1}$, by the selection of $\mathcal{E}_{2,1}$ in Step \hyperref[rp-sec2]{2}, parts (1) and (2) of the claim follow by properties (\hyperref[prop-iii]{iii}) and (\hyperref[prop-ii]{ii}) of Proposition \ref{prop-seeds-and-conc-final}, respectively.

Next observe that since $\mathcal{E}_5\subseteq \mathcal{E}_3\cap \mathcal{E}_4\cap \mathcal{E}_{2,2}$, by \eqref{eq:021b}, part (\hyperref[prop-i.4]{i.4}) of Proposition \ref{prop-seeds-and-conc-final} and the triangle inequality, almost everywhere on the event $\mathcal{E}_5$,
\begin{equation} \label{eq:033bnew}
\!\! \bigg|\bigg\{ \{v,u\}\in\binom{r_B(\mathcal{D}_{d-1})}{2}:
\varrho\Big(\alpha\big( g_{\boldsymbol{H}_{-\ell_0},m,k_0}(v)\big),
\alpha\big( g_{\boldsymbol{H}_{-\ell_0},m,k_0}(u)\big) \Big) \leqslant \frac{r_0}{10}\bigg\}\bigg|
\leqslant \varepsilon k_0^2\,  \frac{(d-1)^{2m}}{m^2}.
\end{equation}
Moreover, if $n$ is sufficiently large in terms of $d$ with $n\geqslant N^2$, by the choices of $\ell_0$ and $k_0$ in \eqref{eq:014b} and part (1) of the claim,
\begin{equation} \label{rp-enewnew}
\varepsilon k_0^2\,  \frac{(d-1)^{2m}}{m^2} \leqslant \varepsilon \frac{K^2 n^2}{m^2} \leqslant
\frac{4\varepsilon K^4}{d^2} \, \ell_0^2 \leqslant \frac{16\varepsilon K^4}{0.6^2\, d^2} \, \binom{r_B(\mathcal{D}_{d-1})}{2}
\leqslant 5\varepsilon K^4 \, \binom{r_B(\mathcal{D}_{d-1})}{2}.
\end{equation}
Thus, part (3) of the claim follows by \eqref{eq:033bnew}, \eqref{rp-enewnew} and the choice of $\mathcal{Q}$ in \eqref{def-Q}.

We now proceed to the proof of part (4). Observe that, since $\mathcal{E}_5\subseteq\mathcal{E}_{2,1}$ and $\mathcal{E}_5\subseteq\mathcal{E}_3$, by property~(\hyperref[prop-i]{i}) of Proposition \ref{prop-seeds-and-conc-final}, the choice of $\mathcal{L}$ in \eqref{eq:009b}, and the construction of $g_{\cdot,m,\cdot}(v)$ as a vertex at distance exactly $m$ from $v$, it holds almost everywhere on the event $\mathcal{E}_5$ that for every $v\in r_B(\mathcal{D}_{d-1})$,
\begin{equation} \label{eq:032b}
\varrho \Big( \boldsymbol{f}(v), \alpha\big(g_{\boldsymbol{H}_{-\ell_0},m,k_0}(v)\big) \Big) =
\varrho \Big( \boldsymbol{f}(v), \boldsymbol{f}\big(g_{\boldsymbol{H},m,k_0}(v)\big) \Big)
\leqslant \frac{r_0}{\sqrt{C_d}} \leqslant \frac{r_0}{50},
\end{equation}
where the last inequality is satisfied by selecting
\begin{enumerate}
\item[(C7)] \label{rp-C7} $C_d\geqslant 50^4$.
\end{enumerate}
Let $\boldsymbol{E}:=\big\{ e\in E_{\boldsymbol{H}}\setminus E_{\boldsymbol{H}_{-\ell_0}} \colon e \in \binom{r_B(\mathcal{D}_{d-1})}{2}\big\}$ denote the set of edges in $E_{\boldsymbol{H}}\setminus E_{\boldsymbol{H}_{-\ell_0}}$ that are contained in $r_B(\mathcal{D}_{d-1})$. Notice that, by parts (1) and (2) of the claim, almost everywhere on~$\mathcal{E}_5$,
\[ |\boldsymbol{E}| = \frac{\big| r_B(\mathcal{D}_{d-1})\big|}{2} \geqslant \frac{0.6\ell_0}{2} \]
and so, by \eqref{eq:007b} and the choices of $m$ and $\ell_0$ in \eqref{rp-e2} and \eqref{eq:014b} respectively,
\begin{align*}
\frac{r_0}{C_d\, \log_{d-1}\log_{d-1} N} & >
\frac{0.6\ell_0}{dn} \cdot \frac{1}{|\boldsymbol{E}|} \sum_{\{v,u\}\in \boldsymbol{E}}
\varrho\big(\boldsymbol{f}(v),\boldsymbol{f}(u)\big) \\
& \geqslant \frac{1}{2K C \log_{d-1}\log_{d-1} N} \cdot \frac{1}{|\boldsymbol{E}|}
\sum_{\{v,u\}\in \boldsymbol{E}}  \varrho\big(\boldsymbol{f}(v),\boldsymbol{f}(u)\big),
\end{align*}
where we have used the fact that $n$ can be taken large in terms of $d$. Thus, by condition (\hyperref[rp-C3]{C3}),
\begin{equation} \label{rp-malakia}
\frac{1}{|\boldsymbol{E}|} \sum_{\{v,u\}\in \boldsymbol{E}}
\varrho\big(\boldsymbol{f}(v),\boldsymbol{f}(u)\big) \leqslant \frac{r_0}{\sqrt{C_d}},
\end{equation}
which further implies, by Markov's inequality, condition (\hyperref[rp-C7]{C7}) and part (2) of the claim,
\begin{equation} \label{rp-malakia^2}
\Big|\Big\{ \{v,u\}\in \boldsymbol{E} \colon \varrho\big(\boldsymbol{f}(v),\boldsymbol{f}(u)\big)
\geqslant \frac{r_0}{50}\Big\}\Big|\leqslant \frac{|\boldsymbol{E}|}{\sqrt[4]{C_d}}
\leqslant \frac{|r_B(\mathcal{D}_{d-1})|}{\sqrt[4]{C_d}}.
\end{equation}
Observe that, by \eqref{eq:032b} and the triangle inequality, if $\{v,u\}\in \boldsymbol{E}$ with $\{v,u\}\in\mathcal{Q}$, then necessarily we have $\varrho\big(\boldsymbol{f}(v),\boldsymbol{f}(u)\big)\geqslant \frac{r_0}{30}$. The proof of part (4) of the claim is thus completed by \eqref{rp-malakia^2}.
\end{proof}

\subsection*{Step 8: completion of the proof} \label{rp-sec8}

At this point, the previous steps have established that a large portion of the deleted edges $E_{\boldsymbol{H}}\setminus E_{\boldsymbol{H}_{-\ell_0}}$ forms a perfect matching of a $\boldsymbol{H}_{-\ell_0}$-measurable set of vertices and, at the same time, this matching essentially avoids a large and (crucially) also $\boldsymbol{H}_{-\ell_0}$-measurable set of pairs of vertices of the aforementioned set. The goal of this step is to show, using Lemma \ref{lemma-matchings}, that this event is unlikely.

We proceed to the details. Define the following random $\boldsymbol{G}_{-\ell_0}$-measurable sets of vertices and pair of vertices, respectively,
\begin{gather}
\label{eq:038b} \mathcal{D}'_{d-1}:=\big\{ v\in [n] \colon \mathrm{deg}_{\boldsymbol{G}_{-\ell_0}}(v)=d-1\big\}, \\
\label{eq:039b} \mathcal{Q}':= \bigg\{ \{v,u\}\in\binom{r_B(\mathcal{D}'_{d-1})}{2}\colon
\varrho\Big(\alpha\big( g_{\boldsymbol{G}_{-\ell_0},m,k_0}(v)\big),
\alpha\big( g_{\boldsymbol{G}_{-\ell_0},m,k_0}(u)\big) \Big) \geqslant \frac{r_0}{10}\bigg\}.
\end{gather}
Also define the events
\begin{align}
\label{eq:040b} \mathcal{E}_6 & :=\bigg[ |r_B(\mathcal{D}_{d-1})|\geqslant 0.6\ell_0,
E_{\boldsymbol{H}}\setminus E_{\boldsymbol{H}_{-\ell_0}} \text{ contains a perfect matching of } r_B(\mathcal{D}_{d-1}), \\
& \ \ \ \ \ \ \ \ \ \ \ \ \ \ \ \ \ \ \ \ \ \ \ \
\binom{r_B(\mathcal{D}_{d-1})}{2}\cap E_{\boldsymbol{H}_{-\ell_0}}=\emptyset \text{ and }
|\mathcal{Q}| \geqslant (1-5\varepsilon K^4)\binom{|r_B(\mathcal{D}_{d-1})|}{2} \bigg], \nonumber\\
\label{eq:041b} \mathcal{E}_6' & :=\bigg[   |r_B(\mathcal{D}'_{d-1})|\geqslant 0.6\ell_0,
E_{\boldsymbol{G}}\setminus E_{\boldsymbol{G}_{-\ell_0}} \text{ contains a perfect matching of } r_B(\mathcal{D}'_{d-1}), \\
&  \ \ \ \ \ \ \ \ \ \ \ \ \ \ \ \ \ \ \ \ \ \ \ \
\binom{r_B(\mathcal{D}'_{d-1})}{2}\cap E_{\boldsymbol{G}_{-\ell_0}}=\emptyset \text{ and }
|\mathcal{Q}'| \geqslant (1-5\varepsilon K^4)\binom{|r_B(\mathcal{D}'_{d-1})|}{2} \bigg]. \nonumber
\end{align}
By Lemma \ref{rm-l1}, Claim \ref{cl:7.3} and \eqref{eq:031b}, we have, provided that $n$ is sufficiently large in terms of $d$ with $n\geqslant N^2$,
\begin{align*}
\mathbb{P}\bigg[ \big|\mathcal{Q}'\cap\, & (E_{\boldsymbol{G}}\setminus E_{\boldsymbol{G}_{-\ell_0}})\big|
\leqslant \frac{1}{\sqrt[4]{C_d}}\, |r_B(\mathcal{D}'_{d-1})| \,\bigg|\, \mathcal{E}_6' \bigg] = \\
& \mathbb{P}\bigg[ \big|\mathcal{Q}\cap (E_{\boldsymbol{H}}\setminus E_{\boldsymbol{H}_{-\ell_0}})\big|
\leqslant \frac{1}{\sqrt[4]{C_d}}\, |r_B(\mathcal{D}_{d-1})| \,\bigg|\, \mathcal{E}_6 \bigg]
\geqslant \mathbb{P}\big[\mathcal{E}_5\big] \geqslant  \exp(-4\ell_0) \, \mathbb{P}\big[\mathcal{E}_1\big].
\end{align*}
Thus, there exists a realization $G_{-\ell_0}$ of $\boldsymbol{G}_{-\ell_0}$ such that, denoting by $Q_0'$ and $F_0'$ the realizations of $\mathcal{Q}'$ and $r_B(\mathcal{D}'_{d-1})$ on the event $[\boldsymbol{G}_{-\ell_0} = G_{-\ell_0}]$, if $n$ is sufficiently large in terms of $d$ with $n\geqslant N^2$, we have $\mathbb{P}\big[ \mathcal{E}_6' \cap [\boldsymbol{G}_{-\ell_0} = G_{-\ell_0}]\big]>0$ and
\begin{equation} \label{eq:043b}
\mathbb{P}\bigg[ \big|Q'_0\cap (E_{\boldsymbol{G}}\setminus E_{G_{-\ell_0}})\big|
\leqslant \frac{1}{\sqrt[4]{C_d}}\, |F_0'| \,\bigg|\,
\mathcal{E}_6' \cap [\boldsymbol{G}_{-\ell_0} = G_{-\ell_0}] \bigg]
\geqslant \exp(-4\ell_0) \, \mathbb{P}\big[\mathcal{E}_1\big].
\end{equation}

Notice that, almost everywhere on the event $\mathcal{E}_6' \cap [\boldsymbol{G}_{-\ell_0} = G_{-\ell_0}]$, the following are satisfied.
\begin{enumerate}
\item[(1$'$)] \label{pr1'} We have $|F'_0|\geqslant 0.6\ell_0$; moreover, $\mathrm{deg}_{G_{-\ell_0}}(v)=d-1$, for every $v\in F'_0$.
\item[(2$'$)] \label{pr2'} The set of edges $E_{\boldsymbol{G}}\setminus E_{G_{-\ell_0}}$ contains a perfect matching of the set $F'_0$; moreover, the set $F'_0$ contains no edges of $G_{-\ell_0}$, that is, $\binom{F'_0}{2}\cap E_{G_{-\ell_0}}=\emptyset$.
\item[(3$'$)] \label{pr3'} We have $Q_0' \subseteq \binom{F'_0}{2}$ and $|Q_0'| \geqslant (1-5\varepsilon K^4)\binom{|F'_0|}{2}$.
\end{enumerate}
By properties (\hyperref[pr1']{$1'$}) and (\hyperref[pr1']{$2'$}), almost everywhere on the event $\mathcal{E}_6' \cap [\boldsymbol{G}_{-\ell_0} = G_{-\ell_0}]$, every edge in $E_{\boldsymbol{G}}\setminus E_{G_{-\ell_0}}$ either is contained in $F'_0$, or is disjoint from $F'_0$. Consequently, for any set of edges $E_0$ disjoint from $F'_0$ with $|E_0|=\ell_0 - \frac{|F'_0|}{2}$, either (i) $\mathcal{E}_6' \cap [\boldsymbol{G}_{-\ell_0} = G_{-\ell_0}]\cap [E_0\subseteq E_{\boldsymbol{G}}\setminus E_{G_{-\ell_0}}]=\emptyset$, or (ii) the random variable $E_{\boldsymbol{G}}\setminus (E_{G_{-\ell_0}}\cup E_0)$ conditioned on $\mathcal{E}_6' \cap [\boldsymbol{G}_{-\ell_0} = G_{-\ell_0}]\cap [E_0\subseteq E_{\boldsymbol{G}}\setminus E_{G_{-\ell_0}}]$ is uniformly distributed on the set of perfect matchings of $F_0'$.

By \eqref{eq:043b}, if $n$ is sufficiently large in terms of $d$ with $n\geqslant N^2$, there exists a set of edges $E_0$ disjoint from $F'_0$ with $|E_0|=\ell_0 - \frac{|F'_0|}{2}$ that satisfies $\mathbb{P}\big[\mathcal{E}_6' \cap [\boldsymbol{G}_{-\ell_0} = G_{-\ell_0}]\cap [E_0\subseteq E_{\boldsymbol{G}}\setminus E_{G_{-\ell_0}}]\big]>0$ (thus, the random variable $E_{\boldsymbol{G}}\setminus (E_{G_{-\ell_0}}\cup E_0)$ conditioned on $\mathcal{E}_6' \cap [\boldsymbol{G}_{-\ell_0} = G_{-\ell_0}]\cap [E_0\subseteq E_{\boldsymbol{G}}\setminus E_{G_{-\ell_0}}]$ is uniformly distributed on the set of perfect matchings of $F_0'$), and
\begin{equation} \label{eq:044b}
\mathbb{P}\big[\mathcal{E}_1\big]\!\leqslant \exp(4\ell_0) \,
\mathbb{P}\bigg[ \big|Q'_0\cap (E_{\boldsymbol{G}}\setminus E_{G_{-\ell_0}})\big|
\!\leqslant \! \frac{1}{\sqrt[4]{C_d}}\, |F_0'| \,\bigg|\, \mathcal{E}_6' \cap
[\boldsymbol{G}_{-\ell_0} = G_{-\ell_0}]\cap [E_0\subseteq E_{\boldsymbol{G}}\setminus E_{G_{-\ell_0}}]\bigg];
\end{equation}
moreover, selecting
\begin{enumerate}
\item[(C8)] \label{rp-C8} $\varepsilon$ small enough so that $5\varepsilon K^4\leqslant \exp\big(-10^3 KC\big)$, and
$C_d$ large enough so that $2/\sqrt[4]{C_d}\leqslant 5\varepsilon K^4$,
\end{enumerate}
by Lemma \ref{lemma-matchings} applied for ``$[\ell]=F'_0$'', ``$Y=Q'_0$'', ``$\varepsilon = 5\varepsilon K^4$'' and
``$c = 2/\sqrt[4]{C_d}$'', we have
\begin{align} \label{eq:044bnew}
\mathbb{P}\bigg[ \big|Q'_0\cap (E_{\boldsymbol{G}}\setminus E_{G_{-\ell_0}})\big| \leqslant
\frac{1}{\sqrt[4]{C_d}}|F_0'|  \,\bigg| \, & \mathcal{E}_6' \cap [\boldsymbol{G}_{-\ell_0} = G_{-\ell_0}]
\cap [E_0\subseteq E_{\boldsymbol{G}}\setminus E_{G_{-\ell_0}}] \bigg]   \\
& \ \ \ \ \ \ \ \ \  \leqslant \exp\big(-10KC|F'_0|\big) \leqslant \exp\big(-5KC\ell_0\big).\nonumber
\end{align}
By \eqref{eq:043b}, \eqref{eq:044bnew} and the choices of $m$ and $\ell_0$ in \eqref{rp-e2} and \eqref{eq:014b} respectively, we conclude that, if $n$ is sufficiently large in terms of $d$ with $n\geqslant N^2$,
\[ \mathbb{P}\big[\mathcal{E}_1\big] \leqslant \exp\big(-KC\ell_0\big) \leqslant \exp\Big(-\frac{Cn}{m}\Big) \leqslant \exp\Big(-\frac{n}{\log_{d-1}\log_{d-1}N}\Big) \leqslant \exp\Big(-\frac{n}{N^2}\Big), \]
as desired. Finally, notice that the values of $K,C,\varepsilon$ and $C_d$ can be determined by conditions (\hyperref[rp-C1]{C1})--(\hyperref[rp-C8]{C8}). The proof of Theorem \ref{random-poincare-2} is completed.

\subsection{Proof of Theorem \ref{regular-random-poincare}} \label{proof-regular-random-poincare}

As already mentioned, Theorem \ref{regular-random-poincare} is proved arguing precisely as in the proof of Theorem \ref{random-poincare-2}, except with slightly different conditioning for derandomizing $\ave_{\mathcal{M}}(\boldsymbol{f},1)$ in Step \hyperref[rp-sec4]{4}. As a consequence, we shall only indicate the necessary changes.

Again we start by observing that, by Remark \ref{rem8.6}, we may assume that $n\geqslant (\log_{d-1}N)^A$ with $A$ sufficiently large in terms of $d$. Set
\begin{equation} \label{reg-rp-e2}
m := \big\lfloor C \, \log_{d-1}\log_{d-1} N \big\rfloor,
\end{equation}
and observe that if $n\geqslant (\log_{d-1}N)^A$ with $A\geqslant 25\, C$, then $m\leqslant \frac{1}{25} \log_{d-1}n$. Thus, setting
\begin{equation} \label{reg-rp-e3}
\mathcal{X}:= \big[|\mathcal{T}(\boldsymbol{U},m)|\geqslant n - \sqrt{n}\big]
\ \ \ \text{ and } \ \ \
\mathcal{Y} := \big[ h(\boldsymbol{U}) \geqslant 0.005d \big],
\end{equation}
where $\mathcal{T}(\boldsymbol{U},m)$ is as in \eqref{eq-tree}, by Proposition \ref{prop-tree} and Corollary \ref{Friedman}, there is $\tau_1=\tau_1(d)>0$ so that, if $n\geqslant (\log_{d-1}N)^A$ with $A$ sufficiently large in terms of $d$,
\begin{equation} \label{reg-rp-e4}
\mathbb{P}\big[\mathcal{X}\cap \mathcal{Y}\big] \geqslant 1-O_d\Big(\frac{1}{n^{\tau_1}}\Big).
\end{equation}
Consequently, setting
\begin{equation}
\label{reg-rp-e5} \mathcal{E}_1 := \big[\gamma(\boldsymbol{U},\varrho)> C_d\, \log_{d-1}\log_{d-1} N\big]
\cap (\mathcal{X} \cap \mathcal{Y}),
\end{equation}
it is enough to show that if $n\geqslant (\log_{d-1}N)^A$ with $A$ sufficiently large in terms of $d$, then
\begin{equation} \label{eq-goal-2} \mathbb{P}\big[\mathcal{E}_1\big]\leqslant \exp\Big( - \frac{n}{(\log_{d-1}N)^A}\Big).
\end{equation}
The proof of this fact follows the same steps of the proof of Theorem \ref{random-poincare-2}. The main difference is that in Step \hyperref[rp-sec4]{4}, our assumption that $\cM$ is well-conditioned allows for a much more efficient union bound over possible values of $\ave_{\mathcal{M}}(\boldsymbol{f},1)$. Specifically, set $J:=\big\lfloor\log_2\big(2n^2 e^N\big)\big\rfloor$, and notice that for every non-constant function $f\colon [n]\to M$ there exists $j\in [J]$ such that
\[ 2^{-j}\, \mathrm{diam}(\mathcal{M}) < \mathrm{ave}_{\mathcal{M}}(f,1) \leqslant 2^{-j+1}\, \mathrm{diam}(\mathcal{M}). \]
Indeed, since $f$ is assumed to be non-constant, we have
\[ \frac{\diam(\cM)}{2^J} <
\frac{\diam(\cM)}{n^2 e^N} \leqslant
\frac{m(\cM)}{n^2} \leqslant \ave_{\cM}(f,1) \leqslant \diam(\cM), \]
where the first (strict) inequality follows from the choice of $J$, and the second inequality follows from the assumption that $\cM$ is well-conditioned. Moreover, by the choice of $m$ in \eqref{reg-rp-e2}, if $n\geqslant (\log_{d-1}N)^A$ with $A$ sufficiently large in terms of $d, K$ and $C$,
\begin{align*}
J \leqslant 2N+2\log n & \leqslant
\exp\Big( 2\log N+ 2\log\log n\Big)   \\
&\leqslant \exp\Big( \frac{dn}{2KC \log_{d-1}\log_{d-1}N}\Big)
\leqslant \exp\Big( \frac{dn}{2Km}\Big) \leqslant \exp(\ell_0),
\end{align*}
where $\ell_0$ is again defined as in \eqref{eq:014b} and we have used the fact that $N>d$. Thus, letting $\mathcal{E}_2$ denote the same event as in the proof of Theorem \ref{random-poincare-2}, there exists $j_0\in [J]$ such that if $n\geqslant (\log_{d-1}N)^A$ with $A$ sufficiently large in terms of $d$, then, setting
\begin{equation} \label{reg-eq:024b}
\mathcal{E}_3 := \mathcal{E}_2\cap \big[2^{-j_0}\, \mathrm{diam}(\mathcal{M}) <\ave_\mathcal{M}(\boldsymbol{f},1)
\leqslant 2^{-j_0+1}\, \mathrm{diam}(\mathcal{M})\big],
\end{equation}
we have
\begin{equation} \label{reg-eq:025b}
\mathbb{P}\big[\mathcal{E}_3\big] \geqslant \mathbb{P}\big[\mathcal{E}_2\big] \, J^{-1} \geqslant
\mathbb{P}\big[\mathcal{E}_2\big] \, \exp(-\ell_0).
\end{equation}
The rest of the proof of Theorem \ref{regular-random-poincare} follows exactly the same steps of the proof of Theorem \ref{random-poincare-2} replacing $r_0$ with its proxy $2^{-j_0}\, \mathrm{diam}(\mathcal{M})$.

\section{Reduction to well-conditioned metric spaces} \label{reduction}

The purpose of this section is to present a result that will enable us to reduce---in a certain regime---the case of arbitrary metric spaces in Theorem \ref{random-poincare} to the case of well-conditioned metric spaces. Recall that $a(\mathcal{M})$ denotes the aspect ratio of a metric space $\mathcal{M}=(M,\varrho)$, defined in \eqref{eq-aspect-ratio}.

\begin{proposition} [Reduction to well-conditioned spaces]\label{prop-aspect-ratio}
Let $n, N$ be any pair of integers with $n,N\geqslant 2$. Then for every metric space $\mathcal{M}=(M,\varrho)$ with $|M|=N$, there exists a metric space $\mathcal{M}'=(M',\varrho')$ with
\begin{equation} \label{red-e1}
N\leqslant |M'|\leqslant N^3 \ \ \ \text{ and } \ \ \ a(\mathcal{M}')\leqslant n^4,
\end{equation}
such that for every connected graph\footnote{Here, for any metric space $\mathcal{N}=(N,\rho)$, $\gamma(G,\rho)$ denotes the optimal Poincar\'{e} constant for functions from the vertex-set of $G$ into the metric space $\mathcal{N}$, that is, the smallest constant $\gamma\in (0,\infty]$ such that, for any $f\colon [n] \to N$, we have $\frac{1}{n^2} \sum_{v,u\in [n]} \rho\big(f(v),f(u)\big) \leqslant
\frac{\gamma}{|E_G|} \sum_{\{v,u\}\in E_G} \rho\big(f(v),f(u)\big)$.} $G$ on $[n]$,
\begin{equation} \label{eq:gap comparison}
\gamma(G,\varrho)\leqslant 2\,\gamma(G,\varrho').
\end{equation}
\end{proposition}

\begin{proof}
Fix $n,N$ and $\mathcal{M}=(M,\varrho)$, and let $\mathcal{D}$ be the set of distances within $\mathcal{M}$, that is,
\[ \mathcal{D}:=\big\{\varrho(x,y) \colon x,y\in M \text{ and } x\neq y\big\}. \]
For every $\tau\in\mathcal{D}$, define an auxiliary metric space $\mathcal{M}_\tau=(M, \varrho_\tau)$ on the same underlying set with $\mathcal{M}$, with the metric $\varrho_\tau$ defined by setting
\[ \varrho_\tau(x,y):=\min\Big\{n^2, \frac{\varrho(x,y)}{\tau}+\frac{1}{n^2}\Big\} \]
for every $x,y\in M$ with $x\neq y$, and $\varrho_\tau(x,x)=0$ otherwise.

\begin{claim} \label{ods:1}
For every $\tau\in\mathcal{D}$, $\mathcal{M}_\tau$ is a metric space.
\end{claim}

\begin{proof}[Proof of Claim \ref{ods:1}]
Fix $\tau\in\mathcal{D}$; it is enough to show that $\varrho_\tau$ satisfies the triangle inequality. To this end, let $x,y,z\in M$. Notice that if either $\varrho_\tau(x,y)=n^2$, or $\varrho_\tau(y,z)=n^2$, then we immediately obtain that $\varrho_\tau(x,z)\leqslant n^2\leqslant \varrho_\tau(x,y) +\varrho_\tau(y,z)$. Thus, we may assume that
\[ \varrho_\tau(x,y) = \frac{\varrho(x,y)}{\tau}+ \frac{1}{n^2}\leqslant n^2 \ \ \ \text{ and } \ \ \
\varrho_\tau(y,z) = \frac{\varrho(y,z)}{\tau}+\frac{1}{n^2}\leqslant n^2. \]
Then, by the definition of $\varrho_\tau$ and the triangle inequality for the metric $\varrho$, we have
\[ \varrho_\tau(x,z)\leqslant \frac{\varrho(x,z)}{\tau}+\frac{1}{n^2} \leqslant
\frac{\varrho(x,y)}{\tau}+\frac{1}{n^2} + \frac{\varrho(y,z)}{\tau}+\frac{1}{n^2} =
\varrho_\tau(x,y) +\varrho_\tau(y,z). \qedhere \]
\end{proof}

Given this family $(\mathcal{M}_\tau)_{\tau\in\mathcal{D}}$ of auxiliary metric spaces, we define $\mathcal{M}'=(M',\varrho')$ as the disjoint union of $(\mathcal{M}_\tau)_{\tau\in\mathcal{D}}$, where the distance between any two points $v,w$ belonging to different
``clusters'' is set to be equal to $n^2$. Explicitly,
\[ M' := \prod_{\tau\in\mathcal{D}} \big(\{\tau\}\times M\big), \]
and $\varrho'\big((\tau_1,x), (\tau_2,y) \big) = \varrho_{\tau_1}(x,y)$ if $\tau_1=\tau_2$, and $\varrho'\big((\tau_1,x), (\tau_2,y) \big) = n^2$ otherwise.

\begin{claim}\label{ods:2}
We have that $\mathcal{M}'$ is a metric space with $N\leqslant |M'|\leqslant N^3$ and $a(\mathcal{M}')\leqslant n^4$.
\end{claim}

\begin{proof}[Proof of Claim \ref{ods:2}]
It is clear that $N\leqslant |M'|\leqslant N^3$. To verify that $\mathcal{M}'$ is a metric space, it suffices to check that $\varrho'$ satisfies the triangle inequality. So let $(\tau_1,x),(\tau_2,y),(\tau_3,z)\in M'$. If $\tau_1 = \tau_3 = \tau_2$, then, by the definition of $\varrho'$ and Claim \ref{ods:1},
\[ \varrho'\big( (\tau_1,x), (\tau_3,z) \big)=\varrho_{\tau_1}(x,z) \leqslant
\varrho_{\tau_1}(x,y) + \varrho_{\tau_1}(y,z) =
\varrho'\big( (\tau_1,x), (\tau_2,y) \big)  + \varrho'\big( (\tau_2,y), (\tau_3,z) \big). \]
If $\tau_1 = \tau_3\neq \tau_2$, then, by the definition of $\varrho'$,
\[ \varrho'\big( (\tau_1,x), (\tau_3,z) \big)= \varrho_{\tau_1}(x,z) \leqslant n^2 \leqslant 2n^2 =
\varrho'\big( (\tau_1,x), (\tau_2,y) \big) + \varrho'\big( (\tau_2,y), (\tau_3,z) \big). \]
Lastly, if $\tau_1 \neq \tau_3$, then either $\tau_1\neq\tau_2$ or $\tau_2\neq\tau_3$, and consequently, by the definition of $\varrho'$,
\[ \varrho'\big( (\tau_1,x), (\tau_3,z) \big)=n^2\leqslant
\varrho'\big( (\tau_1,x), (\tau_2,y) \big) + \varrho'\big( (\tau_2,y), (\tau_3,z) \big); \]
this completes the proof that $\mathcal{M}'$ is a metric space. Finally, notice that $\mathrm{diam}(\mathcal{M}')\leqslant n^2$ and $m(\mathcal{M}')\geqslant n^{-2}$, that yield that $a(\mathcal{M}')\leqslant n^4$.
\end{proof}

It remains to show \eqref{eq:gap comparison}. To this end, let $G$ be a connected graph on $[n]$, and notice that it is enough to show that for every non-constant $f\colon [n] \to M$, there is some $f'\colon [n] \to M'$ such that
\begin{equation}\label{eq:gap comparison f}
\frac{\sum_{\{v,u\} \in E_G} \varrho'\big(f'(v),f'(u)\big)}{\sum_{v,u \in [n]} \varrho'\big(f'(v),f'(u)\big)}
\leqslant 2\, \frac{\sum_{\{v,u\}\in E_G} \varrho\big(f(v),f(u)\big)}{\sum_{v,u \in [n]} \varrho\big(f(v),f(u)\big)}.
\end{equation}
To this end, fix some non-constant $f\colon [n] \to M$, and define
\[ \tau = \tau(f): = \max\Big\{ \varrho\big(f(v),f(u)\big)\colon v,u \in [n]\Big\}. \]
Since $f$ is non-constant, $\tau$ is strictly positive. Moreover, since $G$ is connected, there is a simple path $P$ in $G$ connecting the
vertices $v$ and $u$ which saturate $\tau$. Then, by the triangle inequality,
\begin{equation} \label{eq:003transf}
\sum_{\{v,u\} \in E_G} \varrho\big(f(v),f(u)\big) \geqslant \sum_{\{v,u\} \in P} \varrho\big(f(v),f(u)\big)\geqslant \tau.
\end{equation}
Define the function $f'\colon [n] \to M'$ by setting $f'(v):= \big(\tau, f(v)\big)$ for every $v \in[n]$. By \eqref{eq:003transf} and the fact that $|E_G|\leqslant n^2$, we have
\begin{align} \label{eq:004transf}
\sum_{\{v,u\} \in E_G} \varrho\big(f(v),f(u)\big) & \geqslant \frac{1}{2}\sum_{\{v,u\} \in E_G}
\Big(\varrho\big(f(v),f(u)\big) + \frac{\tau}{n^2}\Big) \geqslant \frac{\tau}{2}\,
\sum_{\{v,u\} \in E_G} \varrho_\tau\big( f(v),f(u)\big)  \\
& = \frac{\tau}{2}\, \sum_{\{v,u\} \in E_G} \varrho'\big( f'(v),f'(u)\big). \nonumber
\end{align}
On the other hand,  by the choice of $\tau$, we have $\varrho\big(f(v),f(u)\big)\leqslant\tau$ for every $v,u\in[n]$. Hence, setting $R:=\big\{(v,u)\in [n]\times[n]\colon f(v)\neq f(u)\big\}$, by the definition of $\varrho_\tau$ and the fact that $n\geqslant 2$, we see that $\varrho_\tau\big(f(v), f(u)\big)= \tau^{-1}\varrho\big(f(v), f(u)\big)+n^{-2}$ for every $(v,u)\in R$. Thus,
\begin{align} \label{eq:005transf}
\sum_{v,u \in [n]} \varrho\big(f(v),f(u)\big)
&  \leqslant \sum_{(v, u) \in R} \Big( \varrho\big(f(v),f(u)\big) + \frac{\tau}{n^2}\Big)
= \tau \sum_{v,u\in[n]} \varrho_\tau\big( f(v),f(u)\big) \\
& = \tau \sum_{v,u \in [n]} \varrho' \big( f'(u), f'(v)\big). \nonumber
\end{align}
Inequality \eqref{eq:gap comparison f} follows by \eqref{eq:004transf} and \eqref{eq:005transf}. The proof is completed.
\end{proof}


\section{Completion of the proof of Theorem  \ref*{random-poincare}} \label{sec-proof-random-poincare}

We start by isolating the following consequence of Theorems \ref{thm-bourgain} and \ref{regular-random-poincare}.

\begin{corollary}[Theorem \ref{random-poincare} for well-conditioned metric spaces] \label{corollary-regular}
For every integer $d\geqslant 3$, there exist constants $C_d'\geqslant 1$ and $\tau'>0$, that depend only on the degree $d$, such that for any $\text{well-conditioned}$ metric space $\mathcal{M}=(M,\varrho)$ with $N:=|M|\geqslant 3$ and any integer $n>d$ with $dn$ even,
\begin{equation} \label{regular-e2}
\mathbb{P}_{G(n,d)}\Big[G\colon \gamma(G,\varrho)\leqslant C'_d \min\big\{\log n, \, \log_{d-1}\log_{d-1} N\big\} \Big]
\geqslant 1- O_d\Big(\frac{1}{n^{\tau'}}\Big).
\end{equation}
\end{corollary}

\begin{proof}
As we have already noted in Remark \ref{remark-new}, we may assume that $N>d$; thus, in what follows, we have Theorem \ref{regular-random-poincare} at our disposal. Next observe that, by Theorem \ref{thm-bourgain} and Matou\v{s}ek's extrapolation argument \cite{Ma97}---see, also, \cite[Lemma 5.5]{BLMN05}---we have, for every $G\in G(n,d)$,
\[ \gamma(G,\varrho) \leqslant c_{\mathrm{B}} \cdot \log n  \cdot \frac{d}{2(d-\lambda_2(G))}. \]
Therefore, by Corollary \ref{Friedman}, there exist $C_1=C_1(d)\geqslant 1$ and $\tau_1=\tau_1(d)>0$ so that
\begin{equation} \label{regular-e4}
\mathbb{P}_{G(n,d)}\big[G\colon \gamma(G,\varrho) \leqslant C_1 \log n\big]\geqslant 1-O_d\Big(\frac{1}{n^{\tau_1}}\Big).
\end{equation}
Finally, let $C_d,A\geqslant 1$ and $\tau>0$ be as in Theorem \ref{regular-random-poincare}, and set
\[ C'_d:=2\,C_1\, C_d\, A\, \log (d-1) \ \ \ \text{ and } \ \ \ \tau':=\min\{\tau,\tau_1,1\}. \]

Let $n>d$ be arbitrary. If $n \leqslant(\log_{d-1}N)^{2A}$, then $\min\big\{\log n, \, \log_{d-1}\log_{d-1} N\big\}\geqslant \frac{\log n}{2A\log(d-1)}$; therefore, in this case, \eqref{regular-e2} follows from \eqref{regular-e4}. On the other hand, if $n \geqslant(\log_{d-1}N)^{2A}$, then $\min\big\{\log n, \, \log_{d-1}\log_{d-1} N\big\} =\log_{d-1}\log_{d-1} N$ and $\frac{n}{(\log_{d-1}N)^A} \geqslant \sqrt{n}$, and so, by Theorem~\ref{regular-random-poincare} and the choices of $C'_d$ and $\tau'$,
\begin{equation} \label{regular-e5}
\mathbb{P}_{G(n,d)}\Big[G\colon \gamma(G,\varrho)\leqslant C'_d\, \log_{d-1}\log_{d-1} N\Big] \geqslant
1- O_d\Big(\frac{1}{n^{\tau}}\Big) - O_d\Big( \exp\big( - \sqrt{n}\big)\Big) \geqslant 1- O_d\Big(\frac{1}{n^{\tau'}}\Big);
\end{equation}
thus, in this case, \eqref{regular-e2} follows from \eqref{regular-e5}.
\end{proof}

We are finally in a position to complete the proof of Theorem \ref{random-poincare}.

\begin{proof}[Completion of proof of Theorem \ref{random-poincare}]
As in the proof of Corollary \ref{corollary-regular}, we start by observing that, by Remark \ref{remark-new}, it is enough to show that for every integer $d\geqslant 3$, there exist constants $C_d''\geqslant 1$ and $\tau''>0$, that depend only on the degree $d$, such that for any metric space $\mathcal{M}=(M,\varrho)$ with $N:=|M|\geqslant \max\{d^3,10\}$, we have
\begin{equation} \label{proof-e1}
\mathbb{P}_{G(n,d)}\Big[G\colon \gamma(G,\varrho)\leqslant C''_d \min\big\{\log n, \, \log_{d-1}\log_{d-1} N\big\} \Big]
\geqslant 1- O_d\Big(\frac{1}{n^{\tau''}}\Big).
\end{equation}
Let $C_d\geqslant 1$ and $\tau>0$ be as in Theorem \ref{random-poincare-2}, let $C_d'\geqslant 1$ and $\tau'>0$ be as in Corollary \ref{corollary-regular}, let $\tau_1>0$ be as in Corollary \ref{Friedman}, and set
\[ C_d'':= 4\, C_d\, C_d' \ \ \ \text{ and } \ \ \ \tau'':=\min\{\tau_1,\tau,\tau',1\}. \]
We claim that, with these choices, \eqref{proof-e1} is satisfied. Indeed, let $n> d$ and consider cases.

If $n^4\geqslant \exp(N)$, then $\min\big\{\log n, \, \log_{d-1}\log_{d-1} N\big\} =\log_{d-1}\log_{d-1} N$ and $\frac{n}{N^2} \geqslant \sqrt{n}$, and so, by Theorem \ref{random-poincare-2},
\begin{equation} \label{proof-e2}
\mathbb{P}_{G(n,d)}\Big[G\colon \gamma(G,\varrho)\leqslant C_d\, \log_{d-1}\log_{d-1} N\Big] \geqslant
1- O_d\Big(\frac{1}{n^{\tau}}\Big) - O_d\Big( \exp\big( - \sqrt{n}\big)\Big) \geqslant 1- O_d\Big(\frac{1}{n^{\tau'}}\Big);
\end{equation}
consequently, in this case, \eqref{proof-e1} follows from \eqref{proof-e2}.

Next assume that $n^4\leqslant \exp(N)$. Let $\mathcal{M}'=(M',\varrho')$ be the metric space obtained by Proposition~\ref{prop-aspect-ratio} applied to $\mathcal{M}$. Since $N\leqslant |M'| \leqslant N^3$ and $a(\mathcal{M}')\leqslant n^4\leqslant \exp(N)\leqslant \exp\big(|M'|\big)$, we see that the metric space $\mathcal{M}'$ is well-conditioned in the sense of Definition \ref{def-regular-metric}. Therefore, by Corollary \ref{corollary-regular},
\begin{align} \label{proof-e3}
\mathbb{P}_{G(n,d)}\Big[G &\, \colon \gamma(G,\varrho')\leqslant 2C'_d \min\big\{\log n, \, \log_{d-1}\log_{d-1} N
\big\} \Big] \\
& \geqslant
\mathbb{P}_{G(n,d)}\Big[G\colon \gamma(G,\varrho')\leqslant C'_d \min\big\{\log n, \, \log_{d-1}\log_{d-1} \big(|M'|\big)\big\} \Big]
\geqslant 1- O_d\Big(\frac{1}{n^{\tau'}}\Big). \nonumber
\end{align}
Observe that, by Corollary \ref{Friedman}, we have $\mathbb{P}_{G(n,d)}\big[G \text{ is connected}\big] \geqslant 1- O_d\big(\frac{1}{n^{\tau_1}}\big)$. Thus, in this case, \eqref{proof-e1} follows from \eqref{proof-e3} and \eqref{eq:gap comparison}. The above cases are exhaustive, and so the proof is completed.
\end{proof}


\section{Optimality of Theorem~\ref*{random-poincare}} \label{sec-optimality}

Our goal in this section is to show that the estimate \eqref{random-poincare-intro-e1} obtained by Theorem~\ref{random-poincare} is optimal for all values of the parameters $n,N$. Specifically, we have the following proposition.

\begin{proposition}[Optimality of Theorem~\ref{random-poincare}] \label{prop:optofrandom-poincare}
For every integer $d\geqslant 3$, there exist constants $c_d,\tau_d>0$, depending only on $d$, such that the following holds true. For every integer $N\geqslant 3$, there is a metric space $\mathcal{M}=(M,\varrho)$ with $|M|=N$, such that for every integer $n>d$ with $dn$ even,
\begin{equation} \label{sec11-e1}
\mathbb{P}_{G(n,d)}\Big[G\colon \gamma(G,\varrho)\geqslant c_d \min\big\{\log n, \, \log\log N\big\} \Big] \geqslant 1- O_d\Big(\frac{1}{n^{\tau_d}}\Big).
\end{equation}
\end{proposition}

For the proof of Proposition \ref{prop:optofrandom-poincare}, we will need the following definitions.

\begin{definition} \label{truncation}
For every nonnegative integer $k$, let ${\rm trunc}_k(x) := \mathrm{sign}(x) \cdot \min\big\{k,|x|\big\}$ denote the truncation at level $k$ of $x\in\R$.
\end{definition}

\begin{definition} \label{metric-space}
Let $N\geqslant 3$ be an integer, and set
\begin{equation}\label{eq:opt001}
k:=\max\{1,\lfloor \log \log N \rfloor\} \ \ \ \text{ and } \ \ \
s:=\bigg\lfloor \frac{\log N}{\log (2k+1)} \bigg\rfloor.
\end{equation}
Define $\widetilde{\cM}$ as the metric space of size $(2k+1)^{s}$ obtained by restricting $\ell_\infty^{s}$ to the set $\{-k,\dots, k\}^s$; notice that $\widetilde{\cM}$ has at most $N$ points. We extend (arbitrarily) $\widetilde{\cM}$ to a metric space, denoted by~$\cM_N$, with exactly $N$ points.
\end{definition}

\begin{definition}\label{defn_opt_function}
Let $d,N\geqslant 3$ and $n>d$ be integers, let $k, s, \mathcal{M}_N$ be as in Definition \ref{metric-space},~and~set
\begin{equation} \label{eq:opt004}
s_0:= \min \{n, s\} \ \ \ \text{ and } \ \ \
r_0:= \bigg\lfloor \log_{d-1}\Big(\frac{n}{s_0}\Big) \bigg\rfloor.
\end{equation}
Moreover, for every $d$-regular graph $G$ on $[n]$, define a function $f_{G,N}\colon [n] \to \cM_N$ as follows. For every~$v\in [n]$, set $f_{G,N}(v):=(x_i)_{i=1}^s$, where
\[  x_i := {\rm trunc}_k \big(\mathrm{dist}_G(v,i)-r_0\big) \]
if $i\in [s_0]$, and $x_i=0$ if $i\in [s]\setminus[s_0]$.
\end{definition}

We are ready to proceed to the proof of Proposition \ref{prop:optofrandom-poincare}.

\begin{proof}[Proof of Proposition \ref{prop:optofrandom-poincare}]
We may assume\footnote{Indeed, for every $d$-regular graph $G$ and every metric space with at least two points, the function mapping a single vertex to one point and all remaining vertices to another yields $\gamma(G,\varrho)\geqslant (|V_G|-1)/|V_G| \geqslant 1/2$.} that $N$ is sufficiently large. Let $k, s, \cM_N$ be as in Definition \ref{metric-space}. We claim that the metric space $\cM_N=(M,\varrho)$ satisfies the conclusion of the proposition.

To this end, let $n> d\geqslant 3$ be integers; notice that we may also assume that $n$ is sufficiently large in terms of $d$. Given a parameter $\alpha\in(0,1)$, we say that a $d$-regular graph $G$ on $[n]$ satisfies the \emph{long-range expansion property $\mathrm{Expan}(\alpha)$} if for every nonempty $S\subseteq [n]$ and every nonnegative integer $\ell$,
\[ |B_G(S,\ell)| \geqslant \min\Big\{ \frac{3n}{4}, \alpha (d-1)^\ell|S| \Big\}. \]
By \cite[Proposition 1.9, part (a)]{ADTT24}, there exist positive constants $\alpha_d$ and $\tau_d$, depending only on $d$, such that
\[ \mathbb{P}_{G(n,d)}\big[ G\colon G \text{ satisfies property } \mathrm{Expan}(\alpha_d)\big] \geqslant  1- O_d\Big(\frac{1}{n^{\tau_d}}\Big). \]
Set $c_d:=\frac{\alpha_d}{4d^2}$. It is enough to show that if $G\in G(n,d)$ satisfies property $\mathrm{Expan}(\alpha_d)$, then
\begin{equation} \label{eq:opt003}
\gamma(G,\varrho) \geqslant c_d \min \{\log n,\log\log N\}.
\end{equation}
Indeed, let $s_0$ and $r_0$ be as in \eqref{eq:opt004}. Notice that, if $n$ is sufficiently large in terms of $d$ and $N$ is sufficiently large,
\begin{equation} \label{eq:opt002}
\log_{d-1}(s_0) - 2\geqslant \frac{1}{2\log d} \min\big\{\log n, \log\log N\big\}.
\end{equation}
Let $f:=f_{G,N}\colon [n]\to \cM_N$ be as in Definition \ref{defn_opt_function}, and set
\[  V = B_G\big([s_0], r_0\big). \]
Observe that, by property $\mathrm{Expan}(\alpha_d)$,
\begin{equation}\label{eq:opt005}
|V|\geqslant \frac{\alpha_d}{d-1} n.
\end{equation}
Moreover, by the $d$-regularity of $G$, for every $v\in [n]$, if $n$ is sufficiently large in terms of $d$,
\begin{equation}\label{eq:opt006}
\big| B_G\big(v, \lfloor \log_{d-1}n\rfloor - 2\big)^\complement\big|
\geqslant \Big(1-\frac{d}{(d-1)^2}\Big)n.
\end{equation}
For every $v\in V$, select a vertex $s_v\in [s_0]$ such that $\mathrm{dist}_G\big(v, s_v\big)\leqslant r_0$, and set
\[ W_v := B_G\big( s_v, \lfloor \log_{d-1}n\rfloor - 2 \big)^\complement. \]
Then, for every $v\in V$ and every $u\in W_v$, setting $f(v) = (x_i^v)_{i=1}^s$ and $f(u) = (x_i^u)_{i=1}^s$, we have $x^v_{s_v}\leqslant 0 $ and, therefore, by the choice of $r_0$ in \eqref{eq:opt004}, \eqref{eq:opt002} and taking $N$ sufficiently large,
\begin{align} \label{eq:opt007}
\varrho\big(f(v), f(u)\big) & \geqslant  x^u_{s_v} \geqslant \min \big\{k, \lfloor\log_{d-1}n\rfloor-r_0 - 1\big\}  \geqslant\min \big\{k, \log_{d-1}(s_0)-2\big\} \\
& \geqslant \frac{1}{2\log d}\min\big\{\log n, \log\log N\big\}, \nonumber
\end{align}
where the last inequality follows by taking $n$ sufficiently large in terms of $d$.
Thus, by \eqref{eq:opt005}--\eqref{eq:opt007},
\begin{align}\label{eq:opt008}
\frac{1}{n^2} \sum_{v,u\in[n]} & \varrho\big(f(v), f(u)\big)
\geqslant  \frac{1}{n^2} \sum_{v\in V} \sum_{u\in W_v}
\varrho\big(f(v), f(u)\big) \\
& \geqslant \frac{\alpha_d\Big(1-\frac{d}{(d-1)^2}\Big)}{2(d-1)\log d} \,
\min\big\{\log n, \log\log N\big\} \geqslant c_d\, \min\big\{\log n, \log\log N\big\}.\nonumber
\end{align}
Finally, by the triangle inequality, for every $\{v,u\}\in E_G$, we have $\varrho\big(f(v), f(u)\big)\leqslant 1$, and so
\begin{equation}\label{eq:opt009}
\frac{1}{|E_G|}\sum_{\{v,u\}\in E_G} \varrho\big(f(v), f(u)\big) \leqslant1.
\end{equation}
The desired inequality \eqref{eq:opt003} follows by \eqref{eq:opt008} and \eqref{eq:opt009}.
\end{proof}


\section{Proof of Proposition \ref*{prop-seeds-and-conc-final}} \label{sec-prop-5.2-new}

\subsection{Preliminaries} \label{subsec-preliminaries-prop-5.2-new}

Recall that we have separated the generation of a random regular graph by first selecting its isomorphism class, and then a labeling of the vertices provided by a random permutation. We begin by introducing some notation for the ordering that a labeling induces on a subset of the vertices.

For every nonempty finite set $R$, by $\mathrm{LO}(R)$ we denote the set of all (total) linear orders on $R$; we will use the letters $\alpha$ and $\beta$ to denote linear orders, and we will write $r_1 <_{\alpha} r_2$ to denote the fact that $r_1$ is smaller than~$r_2$ in the linear order $\alpha$. If $R\in \binom{\mathbb{N}}{k}$ for some positive integer $k$ and $\alpha\in \mathrm{LO}\big([k]\big)$, then we define $\alpha(R)\in \mathrm{LO}(R)$ as follows. We write the set $R$ in increasing order as $\{r_1<\dots<r_k\}$, and we set, for every $i,j\in [k]$,
\begin{equation} \label{preli-e1}
r_i \, <_{\alpha(R)} \, r_j \ \text{ if and only if } \ i \, <_{\alpha} \, j;
\end{equation}
moreover, for every nonempty subset $A$ of $R$, let
\begin{equation} \label{preli-e2}
\min_{\alpha(R)}(A)
\end{equation}
denote the least, with respect to the order $\alpha(R)$, element of $A$. We isolate, for future use, the following elementary observations.
\begin{enumerate}
\item[(O1)] \label{obs1} Let $A\subseteq R\subseteq \mathbb{N}$ be nonempty finite sets, and set $k:=|R|$ and $j:=|A|$. Let $\boldsymbol{\alpha}$ and $\boldsymbol{\beta}$ be random linear orders uniformly distributed on $\mathrm{LO}\big([k]\big)$ and $\mathrm{LO}\big([j]\big)$, respectively. Then both random variables $\min_{\boldsymbol{\alpha}(R)}(A)$ and $\min_{\boldsymbol{\beta}(A)}(A)$ are uniformly distributed on $A$.
\item[(O2)] \label{obs2} Let $R\subseteq \mathbb{N}$ be nonempty, set $k:=|R|$, and let $B_1,\dots,B_j$ be pairwise disjoint nonempty subsets of $R$. Let $\boldsymbol{\alpha}$ be a random linear order uniformly distributed on $\mathrm{LO}\big([k]\big)$. Then the random variables $\min_{\boldsymbol{\alpha}(R)}(B_1),\dots, \min_{\boldsymbol{\alpha}(R)}(B_j)$ are independent.
\end{enumerate}

We proceed to introduce an analogue of the function $g_{G,m,k}$ defined in Subsection \ref{subsec-seeds}. Again, we will use the value ``$\square$'' to denote an auxiliary element that is our notation for ``undefined''. Let $n$ be a positive integer, let $G$ be a graph on $[n]$ (not necessarily regular), and let $m,k$ be positive integers with $k,m\leqslant n$. Also let $R\in\binom{[n]}{k}$, and let $\alpha\in\mathrm{LO}\big([k]\big)$. We define a function $h_{G,m,R,\alpha}\colon [n]\to R\cup\{\square\}$ as follows. Let $v\in [n]$ be arbitrary.
\begin{enumerate}
\item If $\partial B_G(v,m)\cap R\neq \emptyset$ (that is, there is a vertex in $R$ which is at distance from $v$ exactly $m$), then let $h_{G,m,R,\alpha}(v)$ be the smallest, with respect to the linear order $\alpha(R)$ on $R$, element $s\in R$ that satisfies $\dist_G(s,v)=m$.
\item Otherwise, if there is no vertex $s\in R$ with $\dist_G(s,v)=m$, then set $h_{G,m,R,\alpha}(v):=\square$.
\end{enumerate}

Next, towards developing concentration estimates for ``local properties'' on graphs (that is, properties which only depend on a graph neighborhood), we introduce the following combinatorial result.

\begin{lemma}[Decomposition] \label{prop5.6}
Let $d,m,n$ be integers with $3 \leqslant d\leqslant m \leqslant\log_{d-1}n$, and set
\begin{equation} \label{eq:001prop5.2**}
M=M(d,m):= \sum_{j=1}^{2m}d (d-1)^{j-1}.
\end{equation}
Let $G$ be a graph on $[n]$ of maximum degree at most $d$, and let $W$ be a subset of $[n]$. Then there exists a partition of\, $W$ into sets $W_1,\dots,W_{M+1}$ with the following properties.
\begin{enumerate}
\item [(i)] For every $i\in [M+1]$ and every distinct $v,u\in W_i$, we have $\mathrm{dist}_G(v,u)\geqslant 2m+1$.
\item [(ii)] For every $i\in [M+1]$, we have $|W_i|\geqslant \big\lfloor\frac{|W|}{M+1}\big\rfloor$.
\end{enumerate}
\end{lemma}

\begin{proof}
Construct a graph $G'$ on $[n]$ with edge-set $E_{G'}:= \big\{ \{v,u\}\in \binom{[n]}{2}\colon \mathrm{dist}_G(v,u) \leqslant 2m\big\}$, and let $H:= G'[W]$ be the induced on $W$ subgraph of $G'$. Notice that $H$ has degree at most $M$. Also recall that an equitable $k$-coloring of $H$ is a proper vertex coloring of $H$ with the property that each color is used either $\big\lfloor |W|/k \big\rfloor$ or $\big\lceil |W|/k \big\rceil$ times. By the classical Hajnal–Szemer\'edi theorem~\cite{HS70}, there is an equitable $M+1$ coloring of $W$, and the result follows.
\end{proof}

\subsection{Main result} \label{subsec-statement-prop-5.2-new}

The following proposition is the main result in this section.

\begin{proposition} \label{prop-seeds-and-conc}
For every positive integer $d\geqslant 3$ there exists $K_0=K_0(d)\geqslant 1$ with the following property.
Let $K\geqslant K_0$ and let $m,n$ be integers with
\begin{equation} \label{eq:001prop5.2}
K \leqslant m\leqslant  K ^{-1}\log_{d-1} n.
\end{equation}
Set $\ell_0:=\big\lfloor\frac{d n}{K m}\big\rfloor$ and\, $k_0:=\big\lfloor\frac{K n}{(d-1)^m}\big\rfloor$. Assume the setting in Section \ref{sec-models}, and let $\mathcal{J}$ be a random $\boldsymbol{U}\text{-measurable}$ set of pairs of vertices such that almost everywhere,
\begin{equation} \label{eq:002prop5.2}
|\mathcal{J}|\leqslant  K ^{-1} (d-1)^m\,n.
\end{equation}
Let $\boldsymbol{R}$ and $\boldsymbol{\alpha}$ be two independent random variables, independent of $(\boldsymbol{U}, \boldsymbol{U}_{-\ell_0})$, that are uniformly distributed on $\binom{[n]}{k_0}$ and\, $\mathrm{LO}\big([k_0]\big)$, respectively.
Next, define the random $\boldsymbol{U}_{-\ell_0}$-measurable and $(\boldsymbol{U},\boldsymbol{U}_{-\ell_0})$-measurable, respectively,  sets of vertices by
\begin{equation}
\label{eq:010prop5.2} \mathcal{D}_{d-1} := \big\{v \in [n]\colon \mathrm{deg}_{\boldsymbol{U}_{-\ell_0}}(v)=d-1 \big\} \ \ \
\text{ and } \ \ \ \mathcal{D}_{d-1}^{\mathrm{Tr}} := \mathcal{D}_{d-1}\cap\mathcal{T}(\boldsymbol{U},m),
\end{equation}
where $\mathcal{T}(\boldsymbol{U},m)$ is as in \eqref{eq-tree}. Also, set
\begin{gather}
\label{eq:003prop5.2}  V_{d-1} :=\big\{v\in \mathcal{D}_{d-1}^{\mathrm{Tr}}\colon \partial B_{\boldsymbol{U}_{-\ell_0}}(v,1)\cap  \mathcal{D}_{d-1}^{\mathrm{Tr}}= \emptyset
\text{ and }
h_{\boldsymbol{U}_{-\ell_0},m,\boldsymbol{R}, \boldsymbol{\alpha}}(v) \neq \square\big\},  \\
\label{eq:004prop5.2}  V'_{d-1} := \Big\{ v\in V_{d-1} \colon
h_{\boldsymbol{U}_{-\ell_0},m,\boldsymbol{R}, \boldsymbol{\alpha}}(v) =
h_{\boldsymbol{U},m,\boldsymbol{R}, \boldsymbol{\alpha}}(v) \text{ and }
\big\{v,h_{\boldsymbol{U},m,\boldsymbol{R}, \boldsymbol{\alpha}}(v)\big\}\not\in\mathcal{J}\Big\},\\
\label{eq:005prop5.2ba}
 V''_{d-1} := \Big\{ v\in V_{d-1} \colon |\{w\in V_{d-1}\colon
h_{\boldsymbol{U},m,\boldsymbol{R}, \boldsymbol{\alpha}}(w) =
h_{\boldsymbol{U},m,\boldsymbol{R}, \boldsymbol{\alpha}}(v) \}| \leqslant \frac{(d-1)^m}{m} \Big\},
\end{gather}
and define the events
\begin{align}
\label{eq:009prop5.2ba}
\mathcal{F}_1 & := \Big[ |V_{d-1}| \geqslant 2\ell_0
\Big(1-\frac{5}{\sqrt[4]{K}}\Big) \Big], \\
\label{eq:010prop5.2ba}
\mathcal{F}_2 & := \Big[ |V'_{d-1}| \geqslant \frac{d-1}{d}
\Big(1-\frac{1}{\sqrt[7]{K}}\Big) |V_{d-1}| \Big], \\
\label{eq:011prop5.2ba}
\mathcal{F}_3 & := \Big[ \big|V''_{d-1}\big| \geqslant
\Big(1-\frac{2}{\sqrt[3]{K}}\Big) |V_{d-1}| \Big].
\end{align}
Finally set $\mathcal{X} := \big[|\mathcal{T}(\boldsymbol{U},m)|\geqslant n - \sqrt{n}\big]$. Then, for any realization $U$ of\, $\boldsymbol{U}$ on the event $\mathcal{X}$, we~have
\begin{equation} \label{eq:005prop5.2}
\mathbb{P}\big[ \mathcal{F}_1\cap\mathcal{F}_2\cap\mathcal{F}_3 \, \big|\, \boldsymbol{U}=U\big]
\geqslant 1 - \frac{2}{\sqrt[4]{K}}-3e^{-\sqrt{n}}.
\end{equation}
\end{proposition}

The proof of Proposition \ref{prop-seeds-and-conc} is deferred to Subsection \ref{subsec-proof-5.2-new}. At this point, let us derive Proposition~\ref{prop-seeds-and-conc-final} from Proposition \ref{prop-seeds-and-conc}.

\begin{proof}[Proof of Proposition \ref{prop-seeds-and-conc-final} assuming Proposition \ref{prop-seeds-and-conc}]
First, we define a random $\boldsymbol{U}$-measurable set $\mathcal{J}$ of pairs of vertices as follows: for every realization $U$ of $\boldsymbol{U}$, let $\mathcal{J}$ be equal, on the event $[\boldsymbol{U}=U]$, to the value of $\mathcal{L}$ on the event $[\boldsymbol{U}=U,\boldsymbol{\pi}=\mathrm{Id}]$, where $\mathrm{Id}$ denotes the identity map on $[n]$.

Next, define the random variable $\boldsymbol{R}:=\boldsymbol{\pi}^{-1}\big([k_0]\big)$; notice that $\boldsymbol{R}$ is uniformly distributed on $\binom{[n]}{k_0}$. Enumerating $\boldsymbol{R}=\{r_1<\cdots<r_{k_0}\}$, we also define a random linear order $\boldsymbol{\alpha}$ of $[k_0]$ by setting $i<_{\boldsymbol{\alpha}}j$ if and only if
$\boldsymbol{\pi}(r_i)<\boldsymbol{\pi}(r_j)$. Note that $\boldsymbol{\alpha}$ is uniformly distributed on $\mathrm{LO}\big([k_0]\big)$; moreover, the random variables $\boldsymbol{R}$ and $\boldsymbol{\alpha}$ are independent, and independent of $(\boldsymbol{U}, \boldsymbol{U}_{-\ell_0})$.

We apply Proposition \ref{prop-seeds-and-conc} to $\mathcal{J}$ and $\boldsymbol{R},\boldsymbol{\alpha}$. Let $\mathcal{D}_{d-1},\mathcal{D}^{\mathrm{Tr}}_{d-1},V_{d-1}, V'_{d-1}, V''_{d-1}$ be the random set of vertices defined in \eqref{eq:010prop5.2}--\eqref{eq:005prop5.2}, and let $\mathcal{F}_1,\mathcal{F}_2, \mathcal{F}_3$ be the events defined in \eqref{eq:009prop5.2ba}--\eqref{eq:011prop5.2ba}. By~Proposition \ref{prop-seeds-and-conc}, if $K_1\geqslant K_0$, we have
\begin{equation} \label{eq:070a}
\mathbb{P}\big[ \mathcal{F}_1\cap\mathcal{F}_2\cap\mathcal{F}_3 \, \big|\, \boldsymbol{U}=U\big]
\geqslant 1 - \frac{2}{\sqrt[4]{K}}-3e^{-\sqrt{n}}
\geqslant 1 - \frac{3}{\sqrt[4]{K}}
\end{equation}
for every realization $U$ of\, $\boldsymbol{U}$ on the event $\mathcal{X}$. Next, set
\begin{gather}
\label{eq:12.15} \widehat{\mathcal{D}}_{d-1} := \big\{v \in [n]\colon \mathrm{deg}_{\boldsymbol{H}_{-\ell_0}}(v)=d-1 \big\} \ \ \
\text{ and } \ \ \ \widehat{\mathcal{D}}_{d-1}^{\mathrm{Tr}} := \widehat{\mathcal{D}}_{d-1}\cap\mathcal{T}(\boldsymbol{H},m), \\
\label{eq:071a} \widehat{V}_{d-1} :=\big\{v\in \widehat{\mathcal{D}}_{d-1}^{\mathrm{Tr}}\colon \partial B_{\boldsymbol{H}_{-\ell_0}}(v,1)\cap  \widehat{\mathcal{D}}_{d-1}^{\mathrm{Tr}}= \emptyset
\text{ and } g_{\boldsymbol{H}_{-\ell_0},m,k_0}(v) \neq \square\big\}, \\
\label{eq:072a}  \widehat{V}'_{d-1} := \Big\{ v\in \widehat{V}_{d-1} \colon
g_{\boldsymbol{H}_{-\ell_0},m,k_0}(v) = g_{\boldsymbol{H},m,k_0}(v) \text{ and }
\big\{v,g_{\boldsymbol{H},m,k_0}(v)\big\}\not\in\mathcal{L}\Big\}, \\
\label{eq:073a}  \widehat{V}''_{d-1} := \Big\{v\in \widehat{V}_{d-1} \colon
\big|\big\{w\in \widehat{V}_{d-1} \colon g_{\boldsymbol{H},m,k_0}(w)=g_{\boldsymbol{H},m,k_0}(v)\big\}\big|
\leqslant\frac{(d-1)^m}{m}\Big\}.
\end{gather}
Notice that, almost everywhere, for every $v\in [n]$,
\[ \boldsymbol{\pi}\big(\partial B_{\boldsymbol{U}}(\boldsymbol{\pi}^{-1}(v),m)\big) =
\partial B_{\boldsymbol{H}}(v,m) \ \ \ \text{ and } \ \ \
\boldsymbol{\pi}\big(\partial B_{\boldsymbol{U}_{-\ell_0}}(\boldsymbol{\pi}^{-1}(v),m)\big)
=\partial B_{\boldsymbol{H}_{-\ell_0}}(v,m). \]
Hence, almost everywhere, for every $v\in [n]$,
\begin{gather}
\label{eq:078a} \boldsymbol{\pi}\big(\partial B_{\boldsymbol{U}}(\boldsymbol{\pi}^{-1}(v),m)\cap\boldsymbol{R}\big)
= \partial B_{\boldsymbol{H}}(v,m)\cap[k_0], \\
\label{eq:079a} \boldsymbol{\pi}\big(\partial B_{\boldsymbol{U}_{-\ell_0}}(\boldsymbol{\pi}^{-1}(v),m)\cap\boldsymbol{R}\big)
= \partial B_{\boldsymbol{H}_{-\ell_0}}(v,m)\cap[k_0].
\end{gather}
By \eqref{eq:078a}, almost everywhere, for every $v\in [n]$,
\begin{align}\label{eq:080a}
g_{\boldsymbol{H},m,k_0}(v) & = \min\big(\partial B_{\boldsymbol{H}}(v,m)\cap [k_0]\big) =
\min\Big(\boldsymbol{\pi}\big(\partial B_{\boldsymbol{U}}(\boldsymbol{\pi}^{-1}(v),m)\cap\boldsymbol{R}\big)\Big) \\
& =\boldsymbol{\pi}\Big( \min_{\boldsymbol{\alpha}(\boldsymbol{R})} \big(\partial B_{\boldsymbol{U}}(\boldsymbol{\pi}^{-1}(v),m)\cap\boldsymbol{R}\big) \Big) =
\boldsymbol{\pi}\Big( h_{\boldsymbol{U},m,\boldsymbol{R}, \boldsymbol{\alpha}} \big(\boldsymbol{\pi}^{-1}(v)\big)\Big), \nonumber
\end{align}
where $\min(A)$ denotes the minimum of $A$ in the natural order and is equal to $\square$ if $A$ is empty. Similarly, using \eqref{eq:079a} instead of \eqref{eq:078a}, almost everywhere, for every $v\in [n]$,
\begin{equation} \label{eq:081a}
g_{\boldsymbol{H}_{-\ell_0},m,k_0}(v) =
\boldsymbol{\pi}\Big( h_{\boldsymbol{U}_{-\ell_0},m,\boldsymbol{R}, \boldsymbol{\alpha}} \big(\boldsymbol{\pi}^{-1}(v)\big)\Big).
\end{equation}
Notice that, almost everywhere, for every $v\in [n]$, we have $\mathrm{deg}_{\boldsymbol{H}_{-\ell_0}}(v)
=\mathrm{deg}_{\boldsymbol{U}_{-\ell_0}}\big(\boldsymbol{\pi}^{-1}(v)\big)$; thus, $\widehat{\mathcal{D}}_{d-1} = \boldsymbol{\pi}( \mathcal{D}_{d-1} )$ and $ \widehat{\mathcal{D}}_{d-1}^{\mathrm{Tr}} = \boldsymbol{\pi}( \mathcal{D}_{d-1}^{\mathrm{Tr}})$. Hence, after observing that for every $v\in [n]$,
\[ \partial B_{\boldsymbol{H}_{-\ell_0}}(v,1) =
\boldsymbol{\pi}\Big( \partial B_{\boldsymbol{U}_{-\ell_0}}\big( \boldsymbol{\pi}^{-1}(v) ,1\big)\Big), \]
by \eqref{eq:081a}, we have, almost everywhere,
\begin{equation} \label{eq:082a}
\widehat{V}_{d-1} = \boldsymbol{\pi}(V_{d-1}).
\end{equation}
Since $\mathcal{L}$ is invariant (see Definition \ref{defn-invariant}), by \eqref{eq:080a}--\eqref{eq:082a}, almost everywhere,
\begin{equation} \label{eq:083a}
\widehat{V}'_{d-1} = \boldsymbol{\pi}(V'_{d-1}).
\end{equation}
Finally notice that, by \eqref{eq:080a}--\eqref{eq:082a}, almost everywhere,
\begin{equation} \label{eq:084a}
\widehat{V}''_{d-1} = \boldsymbol{\pi}(V''_{d-1}).
\end{equation}

Set $\widehat{V}:=\widehat{V}_{d-1}\cap\widehat{V}'_{d-1}\cap\widehat{V}''_{d-1}$, and observe that $\widehat{V}$ satisfies part (\hyperref[prop-i]{i}) of Proposition \ref{prop-seeds-and-conc-final}. Moreover, by \eqref{eq:082a}--\eqref{eq:084a}, almost everywhere on the event $\mathcal{F}_1\cap\mathcal{F}_2\cap\mathcal{F}_3$, we have
\begin{equation} \label{eq:085a}
|\widehat{V}|\geqslant \frac{d-1}{d}
\Big(1-\frac{7}{\sqrt[7]{K}} -
\frac{2d}{(d-1)\sqrt[3]{K}}\Big)\Big(1-\frac{5}{\sqrt[4]{K}}\Big)2\ell_0
\geqslant 1.3\, \ell_0,
\end{equation}
where the last inequality follows by taking $K_1$ sufficiently large. Next set
\[ E_{\mathrm{con}} := \big\{ e\in E_{\boldsymbol{H}}\setminus E_{\boldsymbol{H}_{-\ell_0}}\colon e\subseteq \widehat{V}\big\}
\ \ \ \text{ and } \ \ \
E_{\mathrm{att}} := \big\{ e\in E_{\boldsymbol{H}}\setminus E_{\boldsymbol{H}_{-\ell_0}}\colon |e\cap\widehat{V}|=1\big\}, \]
and define
\[ \mathcal{O} := \bigcup_{e\in E_{\mathrm{con}}} e. \]
Since, almost everywhere, for every $v\in\widehat{V}$, we have $\mathrm{deg}_{\boldsymbol{H}_{-\ell_0}}(v)=d-1$, by \eqref{eq:085a}, almost everywhere on the event $\mathcal{F}_1\cap\mathcal{F}_2\cap\mathcal{F}_3$,
\[ 2|E_{\mathrm{con}}|+|E_{\mathrm{att}}| = |\widehat{V}| \stackrel{\eqref{eq:085a}}{\geqslant} 1.3\, \ell_0. \]
Since $|E_{\mathrm{con}}|+|E_{\mathrm{att}}|\leqslant\ell_0$, we thus have, almost everywhere on the event $\mathcal{F}_1\cap\mathcal{F}_2\cap\mathcal{F}_3$,
\begin{equation}\label{eq:086a}
  |E_{\mathrm{con}}| \geqslant 1.3\ell_0 - |E_{\mathrm{con}}| - |E_{\mathrm{att}}|\geqslant 0.3\ell_0.
\end{equation}
Clearly, $\mathcal{O}$ is a subset of $\widehat{V}$ and, consequently, $\mathcal{O}$ satisfies part (\hyperref[prop-i]{i}) of the proposition. Notice that, almost everywhere, for every $v\in\mathcal{O}$, we have $\mathrm{deg}_{\boldsymbol{H}_{-\ell_0}}(v)=d-1$; therefore, the edges in $E_{\mathrm{con}}$ are pairwise disjoint. Moreover, $\mathcal{O}\subseteq \widehat{V}_{d-1}$ and, by the definition of $\widehat{V}_{d-1}$, we have $\binom{\widehat{V}_{d-1}}{2}\cap E_{\boldsymbol{H}_{-\ell_0}} = \emptyset$. Thus, by the definitions of $E_{\mathrm{con}}$ and $\mathcal{O}$, we see that  $\mathcal{O}$ also satisfies part (\hyperref[prop-ii]{ii}) of the proposition. Finally, by \eqref{eq:086a}, almost everywhere on the event $\mathcal{F}_1\cap\mathcal{F}_2\cap\mathcal{F}_3$, we have $|\mathcal{O}|\geqslant 0.6\, \ell_0$; therefore, $\mathcal{O}$ satisfies part (\hyperref[prop-iii]{iii}) of the proposition. By \eqref{eq:070a}, the proof is completed.
\end{proof}

\subsection{Proof of Proposition \ref{prop-seeds-and-conc}} \label{subsec-proof-5.2-new}

For the rest of the proof, we fix a realization $U$ of\, $\boldsymbol{U}$ that satisfies $|\mathcal{T}(U,m)|\geqslant n-\sqrt{n}$. We shall proceed by defining several auxiliary events.

\subsection*{Step 1: definition of auxiliary event $\mathcal{A}_1$} \label{step1}

Notice that, due to the $d$-regularity of $U$, for every $v,u\in [n]$, we have $\mathbb{P}[\mathrm{deg}_{\boldsymbol{U}_{-\ell_0}}(v) = d-1 \,|\,\boldsymbol{U}=U] = \mathbb{P}[\mathrm{deg}_{\boldsymbol{U}_{-\ell_0}}(u)= d-1\, |\,\boldsymbol{U}=U]$; that is, the probability $\mathbb{P}[\mathrm{deg}_{\boldsymbol{U}_{-\ell_0}}(v) = d-1 \,|\,\boldsymbol{U}=U]$ is independent of $v$. Set
\begin{equation} \label{eq:007prop5.2}
\delta := \mathbb{P}\big[\mathrm{deg}_{\boldsymbol{U}_{-\ell_0}}(v) = d-1\,\big|\,\boldsymbol{U}=U\big],
\end{equation}
where $v \in [n]$ is an arbitrary vertex. We claim that
\begin{equation}\label{eq:008prop5.2}
\frac{2\ell_0}{n} \Big( 1- \frac{1}{K} \Big) \leqslant \delta \leqslant \frac{2\ell_0}{n}.
\end{equation}
Indeed, denoting by $(n)_k=n(n-1)\cdots (n-k+1)$ the falling factorial, we have
\begin{equation}\label{eq:009prop5.2}
\delta = d\, \frac{\binom{\frac{dn}{2}-d}{\ell_0-1}}{\binom{\frac{dn}{2}}{\ell_0}} =
\frac{2\ell_0}{n} \cdot \frac{(\frac{dn}{2} - d)_{\ell_0-1}}{(\frac{dn}{2} - 1)_{\ell_0-1}}.
\end{equation}
The right-hand side of \eqref{eq:008prop5.2} follows immediately from \eqref{eq:009prop5.2}. Selecting $K_0\geqslant 4d$, by the choice of $\ell_0$ and \eqref{eq:001prop5.2}, we have $\frac{\ell_0}{d}\leqslant\frac{n}{4}$. Thus, by \eqref{eq:001prop5.2} and \eqref{eq:009prop5.2}, Bernoulli's inequality and the choice of $\ell_0$,
\[ \delta \geqslant \frac{2\ell_0}{n} \Big( 1- \frac{d}{\frac{dn}{2}-\ell_0} \Big)^{\ell_0-1}
\geqslant \frac{2\ell_0}{n} \Big( 1- \frac{4(\ell_0-1)}{n} \Big) \geqslant
\frac{2\ell_0}{n} \Big( 1- \frac{4d}{K^2} \Big) \geqslant
\frac{2\ell_0}{n} \Big( 1- \frac{1}{K} \Big).\]

\subsubsection*{Substep 1.1: definition of auxiliary event $\mathcal{A}_{1,1}$} \label{substep1.1}

Set
\begin{gather}
\label{eq:011prop5.2*} \mathcal{A}_{1,1}:= \bigg[|\mathcal{D}_{d-1}^{\mathrm{Tr}}|\geqslant 2\ell_0 \Big(1-\frac{1}{\sqrt{K}}\Big) \bigg].
\end{gather}
We claim that
\begin{equation} \label{eq:012prop5.2*}
\mathbb{P}\big[\mathcal{A}_{1,1}\,\big|\,\boldsymbol{U}=U\big]\geqslant 1- \frac{2}{\sqrt{K}}.
\end{equation}
Indeed, first observe that $|\mathcal{D}_{d-1}|\leqslant 2\ell_0$ almost everywhere. Since \eqref{eq:001prop5.2} implies that $K^2\leqslant n$, by \eqref{eq:008prop5.2} and the fact that $|\mathcal{T}(U,m)|\geqslant n-\sqrt{n}$, we have
\begin{equation} \label{eq:013prop5.2*}
\mathbb{E}\bigg[ \frac{2\ell_0}{n} - \frac{|\mathcal{D}_{d-1}\cap\mathcal{T}(\boldsymbol{U},m)|}{n} \, \bigg|\, \boldsymbol{U}=U\bigg]
= \frac{2\ell_0}{n}-\frac{|\mathcal{T}(U,m)|}{n}\cdot\delta \leqslant\frac{2\ell_0}{n}\cdot\frac{2}{K};
\end{equation}
thus, \eqref{eq:011prop5.2*} follows by \eqref{eq:013prop5.2*} and Markov's inequality.

\subsubsection*{Substep 1.2: definition of auxiliary event $\mathcal{A}_{1,2}$} \label{substep1.2}

Set
\begin{gather}
\label{eq:12.34new} \mathcal{D}_{d-1}^{\mathrm{Tr,Sp}} := \Big\{ v\in\mathcal{D}_{d-1}^{\mathrm{Tr}} \colon
\partial B_{\boldsymbol{U}_{-\ell_0}}(v,1) \cap \mathcal{D}_{d-1}^{\mathrm{Tr}} = \emptyset \Big\}, \\
\label{eq:12.34} \mathcal{A}_{1,2}:= \bigg[\big| \mathcal{D}_{d-1}^{\mathrm{Tr,Sp}} \big|
\geqslant 2\ell_0 \Big(1-\frac{2}{\sqrt{K}}\Big) \bigg].
\end{gather}
We claim that
\begin{equation} \label{eq:12.35}
\mathbb{P}\big[\mathcal{A}_{1,1} \cap \mathcal{A}_{1,2}\,\big|\,\boldsymbol{U}=U\big]\geqslant 1- \frac{3}{\sqrt{K}}.
\end{equation}
For every $e\in\binom{[n]}{2}$, introduce the event
\begin{equation}\label{eq:12.36}
\mathcal{B}_e := \Big[ e\in E_{\boldsymbol{U}_{-\ell_0}} \text{ and } e\subseteq \mathcal{D}_{d-1}^{\mathrm{Tr}} \Big].
\end{equation}
By the choice of $\ell_0$ and $m$ in \eqref{eq:001prop5.2} and Bernoulli's inequality, for every $e \in E_U$ with $e\subseteq \mathcal{T}(U,m)$, we have, using the fact that $n\geqslant m K \geqslant K^2$ and selecting $K_0$ sufficiently large in terms of $d$,
\begin{align}\label{eq:12.37}
  \mathbb{P}\big[ \mathcal{B}_e \, \big|\,  \boldsymbol{U} = U \big]  &
  =\frac{(d-1)^2 \binom{\frac{dn}{2}-2d+1}{\ell_0-2}}{\binom{\frac{dn}{2}}{\ell_0}}
  \leqslant \frac{(d-1)^2 \ell_0^2 \big(\frac{dn}{2}\big)^{2d-3}}{\big(\frac{dn}{2}\big)^{2d-1}\prod_{j=0}^{2d-2}\big(1- \frac{2j}{dn}\big)}
  \leqslant \frac{(d-1)^2 (2\ell_0)^2 }{(dn)^2\big(1- \frac{4d}{n}\big)}\\
   & \leqslant \Big(\frac{2(d-1)^2}{1-\frac{4d}{n}} \cdot \frac{\ell_0}{dn} \Big) \frac{2}{dn} \ell_0
   \leqslant \frac{1}{K} \frac{2}{dn} \ell_0.\nonumber
\end{align}
Thus, setting
\[ X := \sum_{e\in \binom{[n]}{2}} \mathbbm{1}_{\mathcal{B}_e}, \]
by linearity of expectation, the $d$-regularity of $U$ and \eqref{eq:12.37}, we see that $\mathbb{E}\big[ X \, \big|\,  \boldsymbol{U} = U\big] \leqslant \frac{1}{K} \ell_0$. Therefore, by Markov's inequality, we obtain that
\begin{equation}\label{eq:12.38}
\mathbb{P}\Big[ X\leqslant \frac{1}{\sqrt{K}} \ell_0 \, \Big|\, \boldsymbol{U} = U \Big] \geqslant 1 - \frac{1}{\sqrt{K}}.
\end{equation}
Observing that $\big[X \leqslant \frac{1}{\sqrt{K}} \ell_0 \big]\cap\mathcal{A}_{1,1}\subseteq \mathcal{A}_{1,1} \cap \mathcal{A}_{1,2}$, we see that the claimed estimate \eqref{eq:12.35} follows from \eqref{eq:012prop5.2*} and \eqref{eq:12.38}.

\subsubsection*{Substep 1.3: definition of auxiliary event $\mathcal{A}_{1,3}$} \label{substep1.3}

Next set
\begin{equation} \label{eq:011prop5.2}
\mathcal{A}_{1,3} := \bigg[ \frac{1}{|\mathcal{T}(\boldsymbol{U},m)|} \sum_{v\in\mathcal{T}(\boldsymbol{U},m)} \frac{|\partial B_{\boldsymbol{U}_{-\ell_0}}(v,m)|}{(d-1)^m}  \mathbbm{1}_{\mathcal{D}_{d-1}}(v) \geqslant \frac{2\ell_0}{|\mathcal{T}(\boldsymbol{U},m)|} \Big(1 - \sqrt{\frac{5}{K}}\Big) \bigg].
\end{equation}
We claim that
\begin{equation}\label{eq:012prop5.2}
\mathbb{P}\big[ \mathcal{A}_{1,3}\,\big|\, \boldsymbol{U} = U\big] \geqslant 1 - \frac{3}{\sqrt{5K}}.
\end{equation}
To this end, first observe that
\begin{align} \label{eq:013prop5.2}
\mathbb{E}&\bigg[ \frac{1}{|\mathcal{T}(\boldsymbol{U},m)|} \sum_{v\in\mathcal{T}(\boldsymbol{U},m)} \frac{|\partial B_{\boldsymbol{U}_{-\ell_0}}(v,m)|}{(d-1)^m}  \mathbbm{1}_{\mathcal{D}_{d-1}}(v)\,\bigg|\, \boldsymbol{U}=U \bigg] \\
& = \frac{1}{|\mathcal{T}(U,m)|} \sum_{v\in\mathcal{T}(U,m)} \frac{1}{(d-1)^m}  \sum_{w\in\partial B_U(v,m)}
\mathbb{P}\big[w\in\partial B_{\boldsymbol{U}_{-\ell_0}}(v,m) \text{ and }v\in\mathcal{D}_{d-1} \,\big|\, \boldsymbol{U}=U\big]. \nonumber
\end{align}
Selecting $K_0\geqslant d$, we have $\ell_0-1\leqslant \frac{dn}{4}$ and, by \eqref{eq:001prop5.2}, $m\geqslant d$. Thus, by the choice of $\ell_0$ and Bernoulli's inequality, for every $v\in\mathcal{T}(U,m)$ and every $w\in\partial B_U(v,m)$, we have
\begin{align*}
\mathbb{P}\big[ w\in\partial B_{\boldsymbol{U}_{-\ell_0}}(v,m) \text{ and } & v\in\mathcal{D}_{d-1}\,\big|\, \boldsymbol{U}=U\big]
= (d-1)\frac{\binom{\frac{dn}{2}-m-d+1}{\ell_0-1}}{\binom{\frac{dn}{2}}{\ell_0}} \\
& \geqslant \frac{d-1}{d} \cdot \frac{2\ell_0}{n} \cdot \Big(1-8\frac{m}{dn}\Big)^{\ell_0-1}
\geqslant \frac{d-1}{d} \cdot \frac{2\ell_0}{n} \cdot \Big(1-\frac{8}{K}\Big),
\end{align*}
and since $|\partial B_U(v,m)|= d(d-1)^{m-1}$ for every $v\in \mathcal{T}(U,m)$, by \eqref{eq:013prop5.2}, we obtain that
\begin{equation} \label{eq-5.2-new1}
\mathbb{E}\bigg[ \frac{1}{|\mathcal{T}(\boldsymbol{U},m)|} \sum_{v\in\mathcal{T}(\boldsymbol{U},m)} \frac{|\partial B_{\boldsymbol{U}_{-\ell_0}}(v,m)|}{(d-1)^m}  \mathbbm{1}_{\mathcal{D}_{d-1}}(v) \,\bigg|\, \boldsymbol{U}=U\bigg] \geqslant \frac{2\ell_0}{n}\Big(1-\frac{8}{K}\Big).
\end{equation}
As already noted, \eqref{eq:001prop5.2} implies that $K^2\leqslant n$. By \eqref{eq-5.2-new1} and the fact that $|\mathcal{T}(U,m)|\geqslant n-\sqrt{n}$,
\begin{equation}\label{eq:014prop5.2}
\mathbb{E}\bigg[ \frac{2\ell_0}{|\mathcal{T}(\boldsymbol{U},m)|} -\frac{1}{|\mathcal{T}(\boldsymbol{U},m)|} \sum_{v\in\mathcal{T}(\boldsymbol{U},m)} \frac{|\partial B_{\boldsymbol{U}_{-\ell_0}}(v,m)|}{(d-1)^m}  \mathbbm{1}_{\mathcal{D}_{d-1}}(v)\,\bigg|\, \boldsymbol{U}=U \bigg]
\leqslant \frac{2\ell_0}{|\mathcal{T}(U,m)|} \cdot \frac{9}{K}.
\end{equation}
Finally observe that $\frac{1}{|\mathcal{T}(\boldsymbol{U},m)|} \sum_{v\in\mathcal{T}(\boldsymbol{U},m)} \frac{|\partial B_{\boldsymbol{U}_{-\ell_0}}(v,m)|}{(d-1)^m}  \mathbbm{1}_{\mathcal{D}_{d-1}}(v)\leqslant \frac{2\ell_0}{|\mathcal{T}(U,m)|}$ almost everywhere on the event $[\boldsymbol{U} = U]$. Consequently, \eqref{eq:012prop5.2} follows by \eqref{eq:014prop5.2} and Markov's inequality.

\subsubsection*{Substep 1.4: definition of auxiliary event $\mathcal{A}_{1,4}$} \label{substep1.4}

For every $v\in [n]$, define a random $\boldsymbol{U}$-measurable set of vertices by
\begin{equation} \label{eq-new-J}
\mathcal{J}_v:=\big\{u\in [n] \colon u\in \partial B_{\boldsymbol{U}}(v,m) \text{ and } \{v,u\}\in \mathcal{J}\big\}.
\end{equation}
By \eqref{eq:002prop5.2}, we see that $\sum_{v\in [n]} |\mathcal{J}_v|\leqslant 2 K ^{-1} (d-1)^m\,n$ almost everywhere. Thus, by Markov's inequality, setting
\begin{equation} \label{eq:015prop5.2}
\mathcal{B}_\mathcal{J} := \Big\{v\in[n] \colon |\mathcal{J}_v|\geqslant\frac{(d-1)^m}{\sqrt{K/2}}\Big\},
\end{equation}
we have
\begin{equation} \label{eq:016prop5.2}
|\mathcal{B}_\mathcal{J}|\leqslant \frac{n}{\sqrt{K/2}}.
\end{equation}
Next set
\begin{equation}\label{eq:017prop5.2}
\mathcal{A}_{1,4}:= \bigg[ \frac{1}{n}\sum_{v\in [n]} \mathbbm{1}_{\mathcal{D}_{d-1}}(v)
\mathbbm{1}_{\mathcal{B}_{\mathcal{J}}}(v)\leqslant \frac{2\ell_0}{n}\sqrt[4]{\frac{2}{K}}\bigg].
\end{equation}
We claim that
\begin{equation} \label{eq:018prop5.2}
\mathbb{P}\big[ \mathcal{A}_{1,4}\,\big|\, \boldsymbol{U} = U\big] \geqslant 1 - \sqrt[4]{\frac{2}{K}}.
\end{equation}
Indeed, by \eqref{eq:008prop5.2} and \eqref{eq:016prop5.2}, we have
\begin{equation}\label{eq:019prop5.2}
\mathbb{E}\bigg[ \frac{1}{n} \sum_{v\in[n]} \mathbbm{1}_{\mathcal{D}_{d-1}}(v)
\mathbbm{1}_{\mathcal{B}_{\mathcal{J}}}(v) \,\bigg|\, \boldsymbol{U}=U\bigg]
= \frac{|\mathcal{B}_\mathcal{J}|}{n} \delta \leqslant \frac{2\ell_0}{n}\sqrt{\frac{2}{K}}.
\end{equation}
Therefore, \eqref{eq:018prop5.2} follows by \eqref{eq:019prop5.2} and Markov's inequality.

\subsubsection*{Substep 1.5: definition of auxiliary event $\mathcal{A}_{1,5}$} \label{substep1.5}

Set
\[ \mathcal{R} := \Big\{ w\in[n]\colon |\partial B_{\boldsymbol{U}}(w, m) \cap \mathcal{D}_{d-1}| \geqslant\frac{(d-1)^m}{m}\Big\}, \]
and for every $v\in [n]$, set $\mathcal{R}_v := \mathcal{R}\cap \partial B_{\boldsymbol{U}}(v,m)$. Notice that $\mathcal{R}$ and $\mathcal{R}_v$ are both random $(\boldsymbol{U},\boldsymbol{U}_{-\ell_0})\text{-measurable}$ sets of vertices. Also set
\begin{equation}\label{eq:020prop5.2}
\mathcal{A}_{1,5}:=\bigg[\frac{1}{|\mathcal{T}(\boldsymbol{U},m)|}\sum_{v\in \mathcal{T}(\boldsymbol{U},m)} \mathbbm{1}_{\mathcal{D}_{d-1}}(v) \cdot
|\mathcal{R}_v|\leqslant d(d-1)^{m-1} \frac{2\ell_0}{n} \cdot \frac{1}{K^4}\bigg].
\end{equation}
We claim that
\begin{equation}\label{eq:021prop5.2}
\mathbb{P}\big[ \mathcal{A}_{1,5}\, \big| \, \boldsymbol{U} = U\big] \geqslant 1 - \frac{1}{K^4}.
\end{equation}
Indeed, first observe that
\begin{align} \label{eq:022prop5.2}
\mathbb{E}\bigg[ \frac{1}{|\mathcal{T}(\boldsymbol{U},m)|} & \sum_{v\in \mathcal{T}(\boldsymbol{U},m)} \mathbbm{1}_{\mathcal{D}_{d-1}}(v)\cdot
|\mathcal{R}_v| \,\bigg|\, \boldsymbol{U}=U\bigg] \\
& = \frac{1}{|\mathcal{T}(U,m)|}\sum_{v\in \mathcal{T}(U,m)} \sum_{w\in \partial B_U(v,m)}
\mathbb{P}\big[v\in\mathcal{D}_{d-1} \text{ and } w\in\mathcal{R} \,\big|\, \boldsymbol{U}=U\big].\nonumber
\end{align}
Set $t:= \big\lceil\frac{(d-1)^m}{m}\big\rceil$, and select $K_0$ large enough so that: (i) $t-1 \geqslant \frac{(d-1)^m}{2m}\geqslant 16$, and (ii) $K_0\geqslant(6ed)^2$. Fix $v\in \mathcal{T}(U,m)$ and $w\in \partial B_U(v,m)$, and notice that $|\partial B_U(w,m)|=d(d-1)^{m-1}$. Hence, by \eqref{eq:001prop5.2} and the choices of $\mathcal{R}$ and $\ell_0$,
\begin{align*}
\mathbb{P}\big[v\in\mathcal{D}_{d-1} \text{ and } w\in\mathcal{R} \,\big|\, \boldsymbol{U}=U\big]
& \leqslant \binom{\frac{dn}{2}}{\ell_0}^{-1} d^t
\binom{d(d-1)^{m-1}-1}{t-1} \binom{\frac{dn}{2}-dt}{\ell_0-t}\\
& \leqslant \frac{2\ell_0}{n} \Big(\frac{6ed}{K}\Big)^{t-1}
\leqslant \frac{2\ell_0}{n} \frac{1}{K^8};
\end{align*}
therefore, by \eqref{eq:022prop5.2},
\begin{equation} \label{eq:023prop5.2}
\mathbb{E}\bigg[ \frac{1}{|\mathcal{T}(\boldsymbol{U},m)|}\sum_{v\in \mathcal{T}(\boldsymbol{U},m)}
\mathbbm{1}_{\mathcal{D}_{d-1}}(v)\cdot |\mathcal{R}_v|
\,\bigg|\, \boldsymbol{U}=U\bigg] \leqslant d(d-1)^{m-1} \frac{2\ell_0}{n} \frac{1}{K^8}.
\end{equation}
The desired estimate \eqref{eq:021prop5.2} follows by \eqref{eq:023prop5.2} and Markov's inequality.

\subsubsection*{Substep 1.6: definition of $\mathcal{A}_{1}$ and verification of its properties} \label{substep1.6}

Selecting $K_0$ sufficiently large, by \eqref{eq:12.35}, \eqref{eq:012prop5.2}, \eqref{eq:018prop5.2} and \eqref{eq:021prop5.2}, we see that
\begin{equation}\label{eq:024prop5.2}
\mathbb{P}\big[ \mathcal{A}_{1,1}\cap\mathcal{A}_{1,2}\cap\mathcal{A}_{1,3}\cap \mathcal{A}_{1,4}\cap \mathcal{A}_{1,5}\,\big|\,
\boldsymbol{U} = U\big] \geqslant 1 - \frac{2}{\sqrt[4]{K}}.
\end{equation}
Define the random $(\boldsymbol{U}, \boldsymbol{U}_{-\ell_0})$-measurable set of vertices
\begin{align} \label{eq:025prop5.2}
\mathcal{V}_1 :=\Big\{v\in \mathcal{D}_{d-1}^{\mathrm{Tr,Sp}} \colon &
|\partial B_{\boldsymbol{U}_{-\ell_0}}(v,m)|\geqslant \Big(1-\sqrt[4]{\frac{5}{K}}\Big)(d-1)^m, \\
& \ \ \ \ \ |\mathcal{J}_v|\leqslant\frac{(d-1)^m}{\sqrt{K/2}} \text{ and } |\mathcal{R}_v|\leqslant\frac{(d-1)^m}{K} \Big\}, \nonumber
\end{align}
and set
\begin{equation} \label{eq:026prop5.2**}
\mathcal{A}_1:= \bigg[ |\mathcal{V}_1|\geqslant 2\ell_0\Big(1-\frac{3}{\sqrt[4]{K}}\Big)\bigg].
\end{equation}

\begin{claim} \label{cl:01prop5.2}
Almost everywhere on the event $[\boldsymbol{U}=U]$, we have $\mathcal{A}_{1,1} \cap \mathcal{A}_{1,2}\cap\mathcal{A}_{1,3}\cap\mathcal{A}_{1,4}\cap\mathcal{A}_{1,5}\subseteq \mathcal{A}_1$.
\end{claim}

\begin{proof}[Proof of Claim \ref{cl:01prop5.2}]
Set $\mathcal{A}:=\mathcal{A}_{1,1} \cap \mathcal{A}_{1,2}\cap\mathcal{A}_{1,3}\cap\mathcal{A}_{1,4}\cap\mathcal{A}_{1,5}$. By the choice of $\mathcal{A}_{1,3}$ in \eqref{eq:011prop5.2}, the fact that $|\mathcal{D}_{d-1}|\leqslant2\ell_0$ and the fact that $|\partial B_{\boldsymbol{U}_{-\ell_0}}(v,m)|\leqslant (d-1)^m$ for every $v\in \mathcal{D}_{d-1}$, we have, almost everywhere on the event $\mathcal{A}\cap [\boldsymbol{U}=U]$,
\[ \sum_{v\in \mathcal{D}_{d-1}\cap\mathcal{T}(\boldsymbol{U},m)}1-\frac{|\partial B_{\boldsymbol{U}_{-\ell_0}}(v,m)|}{(d-1)^m}
\leqslant2\ell_0\sqrt{\frac{5}{K}}. \]
Therefore, by Markov's inequality, almost everywhere on the event $\mathcal{A}\cap [\boldsymbol{U}=U]$,
\[ \bigg|\bigg\{v\in\mathcal{D}_{d-1}\cap\mathcal{T}(\boldsymbol{U},m)\colon \frac{|\partial B_{\boldsymbol{U}_{-\ell_0}}(v,m)|}{(d-1)^m}
\leqslant 1-\sqrt[4]{\frac{5}{K}}\bigg\}\bigg|\leqslant 2\ell_0 \sqrt[4]{\frac{5}{K}}. \]
Thus, by the choice of $\mathcal{A}_{1,2}$ in \eqref{eq:12.34}, almost everywhere on the event $\mathcal{A}\cap[\boldsymbol{U}=U]$,
\begin{equation} \label{eq:029prop5.2}
\bigg|\bigg\{v\in\mathcal{D}_{d-1}^{\mathrm{Tr,Sp}}\colon \frac{|\partial B_{\boldsymbol{U}_{-\ell_0}}(v,m)|}{(d-1)^m}
\geqslant 1-\sqrt[4]{\frac{5}{K}}\bigg\}\bigg|\geqslant 2\ell_0 \Big(1-\frac{2}{\sqrt{K}}-\sqrt[4]{\frac{5}{K}}\Big).
\end{equation}
Next observe that, by \eqref{eq:029prop5.2}, the choice of $\mathcal{A}_{1,4}$ in \eqref{eq:017prop5.2} and the choice of $\mathcal{B}_J$ in \eqref{eq:015prop5.2}, almost everywhere on the event $\mathcal{A}\cap[\boldsymbol{U}=U]$, the random set
\[ \bigg\{ v\in \mathcal{D}_{d-1}^{\mathrm{Tr,Sp}} \colon
|\partial B_{\boldsymbol{U}_{-\ell_0}}(v,m)|\geqslant \Big(1-\sqrt[4]{\frac{5}{K}}\Big)(d-1)^m \text{ and } |\mathcal{J}_v|\leqslant\frac{(d-1)^m}{\sqrt{K/2}} \bigg\} \]
has size at least $2\ell_0 \big(1-\frac{2}{\sqrt{K}}-\sqrt[4]{\frac{5}{K}}-\sqrt[4]{\frac{2}{K}}\big)$.
By the choice of $\mathcal{A}_{1,5}$ in \eqref{eq:020prop5.2}, we have, almost everywhere on the event $\mathcal{A}\cap[\boldsymbol{U}=U]$,
\[ \sum_{v\in\mathcal{D}_{d-1}\cap\mathcal{T}(\boldsymbol{U},m)} \frac{|\mathcal{R}_v|}{(d-1)^m}\leqslant\frac{3\ell_0}{K^4}, \]
that implies, by Markov's inequality, that
\[ \bigg|\bigg\{v\in\mathcal{D}_{d-1}\cap\mathcal{T}(\boldsymbol{U},m)\colon
\frac{|\mathcal{R}_v|}{(d-1)^m}\geqslant\frac{1}{K} \bigg\}\bigg| \leqslant 2\ell_0\, \frac{3}{2K^3}. \]
Hence, almost everywhere on the event $\mathcal{A}\cap[\boldsymbol{U}=U]$, the random set
\[ \bigg\{ v\in \mathcal{D}_{d-1}^{\mathrm{Tr,Sp}} \colon
|\partial B_{\boldsymbol{U}_{-\ell_0}}(v,m)|\geqslant \Big(1-\sqrt[4]{\frac{5}{K}}\Big)(d-1)^m , |\mathcal{J}_v|\leqslant\frac{(d-1)^m}{\sqrt{K/2}} \text{ and } |\mathcal{R}_v|\leqslant\frac{(d-1)^m}{K} \bigg\} \]
has size at least $2\ell_0 \big(1-\frac{2}{\sqrt{K}}-\sqrt[4]{\frac{5}{K}}-\sqrt[4]{\frac{2}{K}}-\frac{3}{2K^3}\big)$. The proof of Claim \ref{cl:01prop5.2} is completed by selecting $K_0$ sufficiently large.
\end{proof}

By \eqref{eq:024prop5.2} and Claim \ref{cl:01prop5.2}, we conclude that
\begin{equation}\label{eq:032prop5.2}
\mathbb{P}\big[ \mathcal{A}_1 \, \big|\, \boldsymbol{U} = U\big] \geqslant 1 - \frac{2}{\sqrt[4]{K}}.
\end{equation}
This completes the proof of Step \hyperref[step1]{1}.

\subsection*{Step 2: definition of auxiliary event $\mathcal{A}_2$} \label{step2}

For every subset $R$ of $[n]$, we define a random $(\boldsymbol{U},\boldsymbol{U}_{-\ell_0})\text{-measurable}$ set of vertices by
\begin{align} \label{eq:033prop5.2}
\mathcal{V}_2(R)  := \Big\{ v\in \mathcal{V}_1 \colon &
\big|R\cap \big( \partial B_{\boldsymbol{U}_{-\ell_0}}(v,m)\setminus(\mathcal{J}_v\cup\mathcal{R}_v) \big)\big|\geqslant K\Big(1-\frac{1}{\sqrt[4]{K}}\Big)\Big(1-\frac{3}{\sqrt[4]{K}}\Big), \\
& \, \big|R\cap \big(\partial B_{\boldsymbol{U}}(v,m)\setminus \partial B_{\boldsymbol{U}_{-\ell_0}}(v,m)\big)\big|
\leqslant\Big(1+\frac{1}{\sqrt[6]{K}}\Big)\frac{K}{d-1}, \nonumber \\
& \, \text{and }|R\cap (\mathcal{J}_v\cup\mathcal{R}_v)|\leqslant4\sqrt{K} \Big\}, \nonumber
\end{align}
where $\mathcal{J}_v$ is as in Substep \hyperref[substep1.4]{1.4} and $\mathcal{R}_v$ is as in Substep \hyperref[substep1.5]{1.5}. We set
\begin{equation} \label{eq:034prop5.2}
\mathcal{A}_2:= \bigg[|\mathcal{V}_2(\boldsymbol{R})|\geqslant 2\ell_0 \Big(1-\frac{2}{K}\Big)\Big(1-\frac{3}{\sqrt[4]{K}}\Big)\bigg].
\end{equation}

\begin{claim} \label{cl:02prop5.2}
We have
\begin{equation} \label{eq:035prop5.2}
\mathbb{P}\big[ \mathbb{\mathcal{A}}_2 \,\big|\, [\boldsymbol{U}=U]\cap\mathcal{A}_1\big]\geqslant 1-e^{-\sqrt{n}}.
\end{equation}
\end{claim}

\begin{proof}[Proof of Claim \ref{cl:02prop5.2}]
Almost everywhere on the event $[\boldsymbol{U}=U]\cap\mathcal{A}_1$, by Lemma \ref{prop5.6} applied for ``$W = \mathcal{V}_1$'', where $\mathcal{V}_1$ is as in \eqref{eq:025prop5.2}, there is a random partition of $\mathcal{V}_1$ into $W_1,\dots,W_{M+1}$~such~that
\begin{enumerate}
\item [(i)] for every $i\in [M+1]$ and every distinct $v,u\in W_i$, we have $\mathrm{dist}_{\boldsymbol{U}}(v,u)\geqslant2m+1$, and
\item [(ii)] for every $i\in [M+1]$, we have $|W_i|\geqslant \big\lfloor\frac{|\mathcal{V}_1|}{M+1}\big\rfloor$,
\end{enumerate}
where $M$ is as in \eqref{eq:001prop5.2**}. Note that the random sets $W_1,\dots,W_{M+1}$ are $(\boldsymbol{U},\boldsymbol{U}_{-\ell_0})$-measurable. Moreover, by the choice of $M$ in \eqref{eq:001prop5.2**}, we have  $M+1\leqslant(d-1)^{2m+1}$; thus, selecting $K_0$ sufficiently large and invoking \eqref{eq:001prop5.2}, the choice of $\mathcal{A}_1$ in \eqref{eq:026prop5.2**} and the fact that $\ell_0=\big\lfloor \frac{dn}{Km}\big\rfloor$, almost everywhere on the event $[\boldsymbol{U}=U]\cap\mathcal{A}_1$, we have, for every $i\in [M+1]$,
\begin{equation} \label{eq:037prop5.2**}
|W_i|\geqslant \frac{dn\Big(1-\frac{1}{\sqrt[4]{K}}\Big)}{n^{1/K}\log_{d-1}n}
\geqslant n^{0.9}\frac{2K^2}{1-\frac{1}{K}}.
\end{equation}
Set $p:=\frac{k_0}{n}$, and notice that, by the choice of $k_0$,
\begin{equation}\label{eq:037prop5.2*}
\frac{K}{(d-1)^m}-\frac{1}{n}\leqslant p\leqslant\frac{K}{(d-1)^m}.
\end{equation}
Let $\hat{\boldsymbol{R}}$ be a random subset of $[n]$ drawn from the $p$-biased distribution that is independent of $(\boldsymbol{U},\boldsymbol{U}_{-\ell_0})$. By a non-asymptotic version of Stirling's approximation---see, e.g., \cite{Ro55}---and elementary computations, we see that
\begin{equation} \label{eq-prop-5.2-new.e3}
\mathbb{P}\big[ |\hat{\boldsymbol{R}}|=k_0\big] \geqslant \frac{24}{25\sqrt{2\pi}} \, \sqrt{\frac{n}{k_0(n-k_0)}} \geqslant \frac{1}{3\sqrt{n}}.
\end{equation}
Moreover, by the choice of $\mathcal{A}_1$ in \eqref{eq:026prop5.2**} and the choice of $\mathcal{A}_2$ in \eqref{eq:034prop5.2},
\begin{align}\label{eq:036prop5.2}
\mathbb{P}\big[ \mathbb{\mathcal{A}}_2^\complement \,\big|\, [\boldsymbol{U}=U] & \cap\mathcal{A}_1\big]
\leqslant \mathbb{P}\Big[ |\mathcal{V}_2(\boldsymbol{R})| \leqslant\Big(1-\frac{2}{K}\Big)|\mathcal{V}_1| \,\Big|\,
[\boldsymbol{U}=U]\cap\mathcal{A}_1\Big] \\
& \stackrel{\eqref{eq-prop-5.2-new.e3}}{\leqslant} 3\sqrt{n}\cdot
\mathbb{P}\Big[ |\mathcal{V}_2(\hat{\boldsymbol{R}})| \leqslant\Big(1-\frac{2}{K}\Big)|\mathcal{V}_1| \,\Big|\, [\boldsymbol{U}=U]\cap\mathcal{A}_1\Big] \nonumber \\
& \ \, \leqslant 3\sqrt{n}\sum_{i=1}^{M+1} \mathbb{P}\Big[ |\mathcal{V}_2(\hat{\boldsymbol{R}})\cap W_i| \leqslant\Big(1-\frac{2}{K}\Big)|\mathcal{V}_1\cap W_i| \,\Big|\, [\boldsymbol{U}=U]\cap\mathcal{A}_1\Big]. \nonumber
\end{align}

For every $v\in [n]$, set
\begin{equation} \label{eq-ev}
\mathcal{E}_v:=[\boldsymbol{U}=U]\cap\mathcal{A}_1 \cap [v\in \mathcal{V}_1].
\end{equation}
Selecting $K_0$ sufficiently large, we may assume that $\sqrt[4]{5/K} + \sqrt{2/K}+(1/K)\leqslant2/\sqrt[4]{K}$. Then,  by the choice of the set $\mathcal{V}_1$ in \eqref{eq:025prop5.2}, for every $v\in [n]$, almost everywhere on the event $\mathcal{E}_v$, we have $|\partial B_{\boldsymbol{U}_{-\ell_0}}(v,m)\setminus(\mathcal{J}_v\cup\mathcal{R}_v)|\geqslant(d-1)^m(1-2/\sqrt[4]{K})$. Hence, by \eqref{eq:001prop5.2}, \eqref{eq:037prop5.2*} and selecting $K_0$ sufficiently large, for every $v\in [n]$, almost everywhere on the event $\mathcal{E}_v$,
\[  p\cdot|\partial B_{\boldsymbol{U}_{-\ell_0}}(v,m)\setminus(\mathcal{J}_v\cup\mathcal{R}_v)|
\geqslant \Big(1-\frac{2}{\sqrt[4]{K}}\Big)K-\Big(1-\frac{2}{\sqrt[4]{K}}\Big)\frac{(d-1)^m}{n}
\geqslant \Big(1-\frac{3}{\sqrt[4]{K}}\Big)K. \]
Recall that $\mathcal{A}_1$ is $(\boldsymbol{U}, \boldsymbol{U}_{-\ell_0})$-measurable and that $\hat{\boldsymbol{R}}$ is independent of $(\boldsymbol{U}, \boldsymbol{U}_{-\ell_0})$. Therefore, by part \eqref{chernoff-e2} of Lemma \ref{chernoff} and taking $K_0$ sufficiently large, for every $v\in [n]$, conditioning on any realization of $(\boldsymbol{U}, \boldsymbol{U}_{-\ell_0})$ on the event $\mathcal{E}_v$,
\begin{equation} \label{eq:038prop5.2}
\mathbb{P}\bigg[ \big|\hat{\boldsymbol{R}}\cap \big(\partial B_{\boldsymbol{U}_{-\ell_0}}(v,m)\setminus
\big(\mathcal{J}_v\,\cup\, \mathcal{R}_v)\big)\big|\leqslant K\Big(1 -\frac{1}{\sqrt[4]{K}}\Big)\Big(1-3\frac{1}{\sqrt[4]{K}}\Big)\bigg]
\leqslant \exp\big(-\sqrt[4]{K}\big).
\end{equation}

Now observe that, for every $v\in [n]$, almost everywhere on the event $\mathcal{E}_v$,
\[(d-1)^{m-1}\leqslant
|\partial B_{\boldsymbol{U}}(v,m)\setminus \partial B_{\boldsymbol{U}_{-\ell_0}}(v,m)|
\leqslant (d-1)^{m-1}\Big(1+\frac{1}{\sqrt[5]{K}}\Big)\]
and so, by \eqref{eq:037prop5.2*},
\[ \frac{1}{2}\cdot\frac{K}{d-1}\leqslant p \cdot|\partial B_{\boldsymbol{U}(v,m)}\setminus \partial B_{\boldsymbol{U}_{-\ell_0}(v,m)}|\leqslant\frac{K}{d-1}
\Big(1+\frac{1}{\sqrt[5]{K}}\Big). \]
Again recall that $\mathcal{A}_1$ is $(\boldsymbol{U}, \boldsymbol{U}_{-\ell_0})$-measurable and that $\hat{\boldsymbol{R}}$ is independent of $(\boldsymbol{U}, \boldsymbol{U}_{-\ell_0})$. Applying part \eqref{chernoff-e1} of Lemma \ref{chernoff} this time and taking again $K_0$ sufficiently large, for every $v\in [n]$, conditioning on any realization of $(\boldsymbol{U}, \boldsymbol{U}_{-\ell_0})$ on the event $\mathcal{E}_v$,
\begin{equation} \label{eq:040prop5.2}
\mathbb{P}\bigg[ \big| \hat{\boldsymbol{R}}\cap \big(\partial B_{\boldsymbol{U}}(v,m) \setminus \partial B_{\boldsymbol{U}_{-\ell_0}}(v,m)\big) \big| \geqslant \Big(1+\frac{1}{\sqrt[6]{K}}\Big)\frac{K}{d-1}\bigg]
\leqslant\exp\big(-\sqrt[4]{K}\big).
\end{equation}

Next, select $K_0$ sufficiently large so that, for every $v\in [n]$, almost everywhere on the event $\mathcal{E}_v$, we have $|\mathcal{J}_v\cup\mathcal{R}_v|\leqslant (2/\sqrt{K})(d-1)^m$. Then, by \eqref{eq:037prop5.2*}, for every $v\in [n]$, almost everywhere on the event~$\mathcal{E}_v$, we have $p\cdot|\mathcal{J}_v\cup\mathcal{R}_v|\leqslant 2\sqrt{K}$. Invoking again the fact that $\mathcal{A}_1$ is $(\boldsymbol{U}, \boldsymbol{U}_{-\ell_0})$-measurable and that $\hat{\boldsymbol{R}}$ is independent of $(\boldsymbol{U}, \boldsymbol{U}_{-\ell_0})$, by part \eqref{chernoff-e1} of Lemma \ref{chernoff} and taking $K_0$ large enough, for every $v\in [n]$, conditioning on any realization of $(\boldsymbol{U}, \boldsymbol{U}_{-\ell_0})$ on the event $\mathcal{E}_v \cap \big[|\mathcal{J}_v\cup\mathcal{R}_v|\geqslant4\sqrt{K}\big]$,
\begin{equation}\label{eq:041prop5.2}
\mathbb{P} \Big[ \big|\hat{\boldsymbol{R}}\, \cap \, (\mathcal{J}_v\cup\mathcal{R}_v)\big|\geqslant  4\sqrt{K}\Big]
\leqslant \exp\Big(-\frac{4K}{6\sqrt{K}}\Big) \leqslant\exp\big(-\sqrt[4]{K}\big);
\end{equation}
on the other hand, $\mathbb{P}\big[ |\hat{\boldsymbol{R}}\cap (\mathcal{J}_v\cup\mathcal{R}_v)|\geqslant 4\sqrt{K}\,\big|\, \mathcal{E}_v \cap  [|\mathcal{J}_v\cup\mathcal{R}_v|< 4\sqrt{K}] \big]=0$. Thus, for every $v\in [n]$, conditioning on any realization of $(\boldsymbol{U}, \boldsymbol{U}_{-\ell_0})$ on the event $\mathcal{E}_v$,
\begin{equation} \label{eq:042prop5.2}
\mathbb{P}\Big[ \big|\hat{\boldsymbol{R}}\cap (\mathcal{J}_v\cup\mathcal{R}_v)\big |\geqslant 4\sqrt{K} \Big] \leqslant\exp\big(-\sqrt[4]{K}\big).
\end{equation}

By \eqref{eq:038prop5.2}, \eqref{eq:040prop5.2}, \eqref{eq:042prop5.2} and a union bound, we conclude that, for every $v\in [n]$, conditioning on any realization of $(\boldsymbol{U}, \boldsymbol{U}_{-\ell_0})$ on the event $\mathcal{E}_v$,
\begin{equation} \label{eq:043prop5.2}
\mathbb{P}\big[ v\in\mathcal{V}_2(\hat{\boldsymbol{R}}) \big] \geqslant 1-3\exp\big(-\sqrt[4]{K}\big).
\end{equation}

For every $i\in[M+1]$, set $\mu_i:=\mathbb{E}\big[ |\mathcal{V}_2(\hat{\boldsymbol{R}})\cap W_i| \,\big|\, \boldsymbol{U}, \boldsymbol{U}_{-\ell_0}\big]$; notice that, almost everywhere on the event $[\boldsymbol{U}=U]\cap\mathcal{A}_1$,
\begin{equation} \label{eq:044prop5.2}
\mu_i\geqslant \Big(1-3\exp\big(-\sqrt[4]{K}\big)\Big)\cdot |W_i| \geqslant (1-K^{-1})|W_i|,
\end{equation}
where the last inequality holds by selecting $K_0$ sufficiently large. Moreover, by part (i) of the choice of $(W_i)_{i=1}^{M+1}$, for every distinct $v,u\in W_i$, we have $\partial B_{\boldsymbol{U}}(v,m) \cap \partial B_{\boldsymbol{U}}(u,m) = \emptyset$; by deleting edges we can only increase the path-distance\footnote{In fact, since $W_i\subseteq \mathcal{V}_1\subseteq \mathcal{T}(\boldsymbol{U},m)$, we have $\partial B_{\boldsymbol{U}_{-\ell_0}}(v,m)\subseteq \partial B_{\boldsymbol{U}}(v,m)$ for all $v\in W_i$.} and, consequently, $\partial B_{\boldsymbol{U}_{-\ell_0}}(v,m) \cap \partial B_{\boldsymbol{U}_{-\ell_0}}(u,m) = \emptyset$. Therefore, conditioning on any realization of $(\boldsymbol{U}, \boldsymbol{U}_{-\ell_0})$ on the event $[\boldsymbol{U}=U]\cap\mathcal{A}_1$, the random variables $\big\{ \mathbbm{1}_{\mathcal{V}_2(\hat{\boldsymbol{R}})}(v)\colon v\in W_i\big\}$ are independent. By part \eqref{chernoff-e2} of Lemma \ref{chernoff}, \eqref{eq:037prop5.2**} and \eqref{eq:044prop5.2}, we have, for every $i\in [M+1]$,
\begin{equation} \label{eq:045prop5.2}
\mathbb{P}\Big[ |\mathcal{V}_2(\hat{\boldsymbol{R}}) \cap W_i| \leqslant\Big(1-\frac{2}{K}\Big)|\mathcal{V}_1\cap W_i| \,\Big|\, [\boldsymbol{U}=U]\cap\mathcal{A}_1\Big] \leqslant\exp\big(-n^{0.9}\big).
\end{equation}
Finally notice that $M\leqslant n^{3/K}$, by \eqref{eq:001prop5.2**} and \eqref{eq:001prop5.2}. Therefore, the claim follows by \eqref{eq:036prop5.2}, \eqref{eq:045prop5.2} and taking $K_0$ sufficiently large.
\end{proof}

By \eqref{eq:035prop5.2} and \eqref{eq:032prop5.2}, we conclude that
\begin{equation} \label{eq:047prop5.2}
\mathbb{P}\big[ \mathcal{A}_1 \cap \mathcal{A}_2\, \big|\, \boldsymbol{U} = U\big]
\geqslant \Big(1 - \frac{2}{\sqrt[4]{K}}\Big)\big(1-e^{-\sqrt{n}}\big).
\end{equation}
This completes the proof of Step \hyperref[step2]{2}.

\subsection*{Step 3: definition of auxiliary events $\mathcal{A}_3$ and $\mathcal{A}_4$} \label{step3}

Define the $(\boldsymbol{U},\boldsymbol{U}_{-\ell_0},\boldsymbol{R},\boldsymbol{\alpha})$-measurable random sets of vertices
\begin{gather}
\label{eq:049prop5.2} \mathcal{V}_3 := \Big\{ v\in\mathcal{V}_2(\boldsymbol{R}) \colon
\min_{\boldsymbol{\alpha}(\boldsymbol{R})}\big(\boldsymbol{R}\cap\partial B_{\boldsymbol{U}}(v,m)\big)
\in \partial B_{\boldsymbol{U}_{-\ell_0}}(v,m)\setminus (\mathcal{J}_v\cup\mathcal{R}_v)\Big\},  \\
\label{eq:050prop5.2}
\mathcal{V}_4 := \Big\{ v\in\mathcal{V}_2(\boldsymbol{R}) \colon
\min_{\boldsymbol{\alpha}(\boldsymbol{R})}\big(\boldsymbol{R}\cap\partial B_{\boldsymbol{U}}(v,m)\big) \not\in \mathcal{R}_v\Big\},
\end{gather}
and set
\begin{equation} \label{eq:051prop5.2}
\mathcal{A}_3 := \bigg[ |\mathcal{V}_3| \geqslant
\frac{d-1}{d}\Big(1-\frac{2}{\sqrt[7]{K}}\Big)|\mathcal{V}_2(\boldsymbol{R})|\bigg] \ \ \ \text{ and } \ \ \
\mathcal{A}_4 := \bigg[ |\mathcal{V}_4| \geqslant \Big(1-\frac{2}{\sqrt[3]{K}}\Big)|\mathcal{V}_2(\boldsymbol{R})|\bigg].
\end{equation}

\begin{claim} \label{cl:03prop5.2}
We have
\begin{gather}
\label{eq:053prop5.2} \mathbb{P}\big[ \mathcal{A}_3 \,\big|\,
[\boldsymbol{U}=U]\cap \mathcal{A}_1\cap \mathcal{A}_2\big] \geqslant1-\exp(-\sqrt{n}), \\
\label{eq:054prop5.2} \mathbb{P}\big[ \mathcal{A}_4 \,\big|\,
[\boldsymbol{U}=U]\cap \mathcal{A}_1\cap \mathcal{A}_2\big] \geqslant1-\exp(-\sqrt{n}).
\end{gather}
\end{claim}

\begin{proof}[Proof of Claim \ref{cl:03prop5.2}]
Recall that the random linear order $\boldsymbol{\alpha}$ is independent of $(\boldsymbol{U}, \boldsymbol{U}_{-\ell_0}, \boldsymbol{R})$. Thus, by observation (\hyperref[obs1]{O1}) and selecting $K_0$ sufficiently large, we have, for every $v\in [n]$, conditioning on any realization of $(\boldsymbol{U}, \boldsymbol{U}_{-\ell_0}, \boldsymbol{R})$ on the event $[\boldsymbol{U}=U] \cap\mathcal{A}_1\cap\mathcal{A}_2$, we have
\begin{gather}
\label{eq:055prop5.2} \mathbb{P}\big[v\in\mathcal{V}_3 \big]  =
\frac{\big|\boldsymbol{R}\cap \big(\partial B_{\boldsymbol{U}_{-\ell_0}}(v,m)\setminus
(\mathcal{J}_v\cup\mathcal{R}_v)\big)\big|}{|\boldsymbol{R}\cap\partial B_{\boldsymbol{U}}(v,m)|}\, \mathbbm{1}_{\mathcal{V}_2(\boldsymbol{R})}(v)
\geqslant \frac{d-1}{d}\Big(1-\frac{1}{\sqrt[7]{K}}\Big)\, \mathbbm{1}_{\mathcal{V}_2(\boldsymbol{R})}(v), \\
\label{eq:056prop5.2} \mathbb{P}\big[v\in\mathcal{V}_4 \big] =
\frac{\big|\big(\boldsymbol{R}\cap\partial B_{\boldsymbol{U}}(v,m)\big)\setminus\mathcal{R}_v\big|}{|\boldsymbol{R}\cap\partial B_{\boldsymbol{U}}(v,m)|}\, \mathbbm{1}_{\mathcal{V}_2(\boldsymbol{R})}(v)
\geqslant \Big(1-\frac{1}{\sqrt[3]{K}}\Big)\, \mathbbm{1}_{\mathcal{V}_2(\boldsymbol{R})}(v).
\end{gather}
Indeed, set
\begin{align*}
& a:= \big|\boldsymbol{R}\cap \big(\partial B_{\boldsymbol{U}}(v,m)\setminus
\partial B_{\boldsymbol{U}_{-\ell_0}}(v,m)\big)\big|,
&& a_0:=\Big(1+\frac{1}{\sqrt[6]{K}}\Big)\frac{K}{d-1}, \\
& b:=\big|\boldsymbol{R}\cap \big( \partial B_{\boldsymbol{U}_{-\ell_0}}(v,m)\setminus(\mathcal{J}_v\cup\mathcal{R}_v) \big)\big|,
&& b_0:=  K\Big(1-\frac{1}{\sqrt[4]{K}}\Big)\Big(1-\frac{3}{\sqrt[4]{K}}\Big), \\
& c:= |\boldsymbol{R}\cap (\mathcal{J}_v\cup\mathcal{R}_v)|,
&& c_0:=4\sqrt{K}.
\end{align*}
Then, by the choice of $\mathcal{V}_2(R)$ in \eqref{eq:033prop5.2}, almost everywhere on the event $\big[v\in \mathcal{V}_2(\boldsymbol{R})\big]$, we have $a\leqslant a_0$, $b\geqslant b_0$ and $c\leqslant c_0$. Therefore, if $K$ is sufficiently large,
\begin{align*}
\frac{\big|\boldsymbol{R}\cap \big(\partial B_{\boldsymbol{U}_{-\ell_0}}(v,m)\setminus
(\mathcal{J}_v\cup\mathcal{R}_v)\big)\big|}{|\boldsymbol{R}\cap\partial B_{\boldsymbol{U}}(v,m)|} & \, \mathbbm{1}_{\mathcal{V}_2(\boldsymbol{R})}(v)
\geqslant \frac{b}{a+b+c} \, \mathbbm{1}_{\mathcal{V}_2(\boldsymbol{R})}(v)
\geqslant \frac{b_0}{a+b_0+c} \, \mathbbm{1}_{\mathcal{V}_2(\boldsymbol{R})}(v) \\
& \geqslant \frac{b_0}{a_0+b_0+c_0} \, \mathbbm{1}_{\mathcal{V}_2(\boldsymbol{R})}(v)
\geqslant \frac{d-1}{d}\Big(1-\frac{1}{\sqrt[7]{K}}\Big)\, \mathbbm{1}_{\mathcal{V}_2(\boldsymbol{R})}(v)
\end{align*}
and, similarly,
\begin{align*}
\frac{\big|\big(\boldsymbol{R}\cap\partial B_{\boldsymbol{U}}(v,m)\big)\setminus\mathcal{R}_v\big|}{|\boldsymbol{R}\cap\partial B_{\boldsymbol{U}}(v,m)|} \, \mathbbm{1}_{\mathcal{V}_2(\boldsymbol{R})}(v)
& \geqslant\frac{a+b-c}{a+b+c} \, \mathbbm{1}_{\mathcal{V}_2(\boldsymbol{R})}(v)
= \Big(1-\frac{2c}{a+b+c}\Big) \, \mathbbm{1}_{\mathcal{V}_2(\boldsymbol{R})}(v) \\
& \geqslant \Big(1-\frac{2c_0}{b_0}\Big) \, \mathbbm{1}_{\mathcal{V}_2(\boldsymbol{R})}(v)
\geqslant \Big(1-\frac{1}{\sqrt[3]{K}}\Big)\, \mathbbm{1}_{\mathcal{V}_2(\boldsymbol{R})}(v).
\end{align*}

Almost everywhere on the event $[\boldsymbol{U}=U] \cap\mathcal{A}_1\cap\mathcal{A}_2$, by Lemma \ref{prop5.6} applied for ``$W = \mathcal{V}_2(\boldsymbol{R})$'', we obtain a random partition of $\mathcal{V}_2(\boldsymbol{R})$ into $W_1,\dots,W_{M+1}$ such that
\begin{enumerate}
\item [(i)] for every $i\in [M+1]$ and every distinct $v,u\in W_i$, we have $\mathrm{dist}_{\boldsymbol{U}}(v,u)\geqslant 2m+1$,
\item [(ii)] for every $i\in [M+1]$, we have $|W_i|\geqslant \big\lfloor\frac{|\mathcal{V}_2(\boldsymbol{R})|}{M+1}\big\rfloor$,
\end{enumerate}
where $M$ is as in \eqref{eq:001prop5.2**}. Observe that the random sets $W_1,\dots,W_{M+1}$ are $(\boldsymbol{U},\boldsymbol{U}_{-\ell_0},\boldsymbol{R})$-measurable. By the choice of $M$ in \eqref{eq:001prop5.2**} and \eqref{eq:001prop5.2}, we see that $M+1\leqslant(d-1)^{2m+1}\leqslant n^{3/K}$; thus, selecting $K_0$ sufficiently large, almost everywhere on the event $[\boldsymbol{U}=U]\cap\mathcal{A}_1\cap\mathcal{A}_2$, for every $i\in [M+1]$,
\begin{equation}\label{eq:059prop5.2}
|W_i|\geqslant \frac{dn}{Km (d-1)^{2m+1}} \geqslant
n^{0.9}\frac{2\sqrt[3]{K^2}}{1-\frac{1}{\sqrt[3]{K}}} \geqslant
n^{0.9}\frac{2d\sqrt[7]{K^2}}{(d-1)\Big(1-\frac{1}{\sqrt[7]{K}}\Big)}.
\end{equation}

Now notice that
\begin{align} \label{eq:060prop5.2}
\mathbb{P}\big[ \mathcal{A}_3^\complement \,\big|\, & [\boldsymbol{U}=U] \cap\mathcal{A}_1\cap\mathcal{A}_2\big]
\leqslant \mathbb{P}\Big[ |\mathcal{V}_3| \leqslant \frac{d-1}{d}\Big(1-\frac{2}{\sqrt[7]{K}}\Big)| \mathcal{V}_2(\boldsymbol{R})| \,\Big|\, [\boldsymbol{U}=U] \cap\mathcal{A}_1\cap\mathcal{A}_2\Big] \\
& \leqslant \sum_{i=1}^{M+1}
\mathbb{P}\Big[ |\mathcal{V}_3\cap W_i|
\leqslant \frac{d-1}{d}\Big(1-\frac{2}{\sqrt[7]{K}}\Big)
|W_i| \,\Big|\, [\boldsymbol{U}=U] \cap\mathcal{A}_1\cap\mathcal{A}_2\Big]. \nonumber
\end{align}
For every $i\in[M+1]$, set $\mu_{3,i}:=\mathbb{E}\big[|\mathcal{V}_3\cap W_i|\,\big|\, \boldsymbol{U},\boldsymbol{U}_{-\ell_0}, \boldsymbol{R}\big]$. By \eqref{eq:055prop5.2}, almost everywhere on the event $[\boldsymbol{U}=U] \cap\mathcal{A}_1\cap\mathcal{A}_2$, we have
\begin{equation}\label{eq:061prop5.2}
\mu_{3,i}\geqslant\frac{d-1}{d}\Big(1-\frac{1}{\sqrt[7]{K}}\Big)|W_i|.
\end{equation}
By part (i) of the choice of $(W_i)_{i=1}^{M+1}$, for every $i\in [M+1]$ and every distinct $v,u\in W_i$, we have $\partial B_{\boldsymbol{U}}(v,m) \cap \partial B_{\boldsymbol{U}}(u,m) = \emptyset$. Therefore, by observation (\hyperref[obs2]{O2}),
conditioning on any realization of $(\boldsymbol{U}, \boldsymbol{U}_{-\ell_0}, \boldsymbol{R})$ on the event $[\boldsymbol{U}=U] \cap\mathcal{A}_1\cap\mathcal{A}_2$, the random variables $\big\{ \mathbbm{1}_{\mathcal{V}_3}(v)\colon v\in W_i\big\}$ are independent, for every $i\in [M+1]$. By part \eqref{chernoff-e2} of Lemma \ref{chernoff} and \eqref{eq:059prop5.2}--\eqref{eq:061prop5.2}, we have
\begin{equation} \label{eq:062prop5.2}
\mathbb{P}\big[ \mathcal{A}_3^\complement \,\big|\, [\boldsymbol{U}=U] \cap\mathcal{A}_1\cap\mathcal{A}_2\big]
\leqslant(M+1)\exp\big(-n^{0.9}\big)\leqslant \exp\big(-\sqrt{n}\big),
\end{equation}
where we have used the fact that $M+1\leqslant n^{3/K}$---that follows from \eqref{eq:001prop5.2**} and \eqref{eq:001prop5.2}---and taking $K_0$ sufficiently large.

Similarly, observe that
\begin{align} \label{eq:063prop5.2}
\mathbb{P}\big[\mathcal{A}_4^\complement \,\big|\, [\boldsymbol{U}=U] & \cap\mathcal{A}_1\cap\mathcal{A}_2\big]
\leqslant \mathbb{P}\Big[|\mathcal{V}_4| \leqslant
\Big(1-\frac{2}{\sqrt[3]{K}}\Big)
|\mathcal{V}_2(\boldsymbol{R})| \,\Big|\, [\boldsymbol{U}=U] \cap\mathcal{A}_1\cap\mathcal{A}_2\Big] \\
& \leqslant \sum_{i=1}^{M+1}
\mathbb{P}\Big[ |\mathcal{V}_4\cap W_i|
\leqslant \Big(1-\frac{2}{\sqrt[3]{K}}\Big)
|W_i| \,\Big|\, [\boldsymbol{U}=U] \cap\mathcal{A}_1\cap\mathcal{A}_2\Big]. \nonumber
\end{align}
As above, for every $i\in [M+1]$, set $\mu_{4,i}:=\mathbb{E}\big[|\mathcal{V}_4\cap W_i|\,\big|\, \boldsymbol{U},\boldsymbol{U}_{-\ell_0},
\boldsymbol{R}\big]$ and notice that, by \eqref{eq:056prop5.2}, almost everywhere on the event $[\boldsymbol{U}=U] \cap\mathcal{A}_1\cap\mathcal{A}_2$,
\begin{equation}\label{eq:064prop5.2}
\mu_{4,i}\geqslant\Big(1-\frac{1}{\sqrt[3]{K}}\Big)|W_i|.
\end{equation}
By observation (\hyperref[obs2]{O2}) and the fact that $\partial B_{\boldsymbol{U}}(v,m) \cap \partial B_{\boldsymbol{U}}(u,m) = \emptyset$ for every $i\in[M+1]$ and every distinct $v,u\in W_i$, we see that, conditioning on any realization of $(\boldsymbol{U}, \boldsymbol{U}_{-\ell_0}, \boldsymbol{R})$ on the event $[\boldsymbol{U}=U] \cap\mathcal{A}_1\cap\mathcal{A}_2$, the random variables $\big\{\mathbbm{1}_{\mathcal{V}_4}(v)\colon v\in W_i\big\}$ are independent. Thus, by part \eqref{chernoff-e2} of Lemma \ref{chernoff}, \eqref{eq:063prop5.2}, \eqref{eq:064prop5.2} and \eqref{eq:059prop5.2},
\begin{equation} \label{eq:065prop5.2}
\mathbb{P}\big[ \mathcal{A}_4^\complement \,\big|\, [\boldsymbol{U}=U] \cap\mathcal{A}_1\cap\mathcal{A}_2\big]
\leqslant(M+1)\exp\big(-n^{0.9}\big)\leqslant\exp\big(-\sqrt{n}\big),
\end{equation}
where we have used, again, the fact that $M+1\leqslant n^{3/K}$ and taking $K_0$ sufficiently large.

The proof of Claim \ref{cl:03prop5.2} is completed by \eqref{eq:062prop5.2} and \eqref{eq:065prop5.2}.
\end{proof}

By \eqref{eq:047prop5.2}, \eqref{eq:053prop5.2} and \eqref{eq:054prop5.2}, we conclude that
\begin{equation} \label{eq:066prop5.2}
\mathbb{P}\big[ \mathcal{A}_1 \cap \mathcal{A}_2\cap \mathcal{A}_3\cap \mathcal{A}_4 \,\big|\, \boldsymbol{U} = U\big]
 \geqslant \Big(1 - \frac{2}{\sqrt[4]{K}}\Big)\big(1-3e^{-\sqrt{n}}\big).
\end{equation}
This completes the proof of Step \hyperref[step3]{3}.

\subsection*{Step 4: completion of the proof} \label{step4}

Notice that almost everywhere on the event $[\boldsymbol{U}=U]$,
\[ \mathcal{V}_2(\boldsymbol{R})\subseteq V_{d-1}, \]
\[ \mathcal{V}_3 \subseteq V'_{d-1}, \]
\[ \mathcal{V}_4 \subseteq V''_{d-1}, \]
where $\mathcal{V}_2(\boldsymbol{R})$ is as in \eqref{eq:033prop5.2}, $\mathcal{V}_3$ is as in \eqref{eq:049prop5.2}, and $\mathcal{V}_4$ is as in \eqref{eq:050prop5.2}; consequently,
\[ \mathcal{A}_1 \cap \mathcal{A}_2\cap \mathcal{A}_3\cap \mathcal{A}_4 \cap [\boldsymbol{U}=U] \subseteq
\mathcal{F}_1 \cap \mathcal{F}_2\cap \mathcal{F}_3\cap [\boldsymbol{U}=U], \]
where $\mathcal{A}_1$ is as in \eqref{eq:026prop5.2**}, $\mathcal{A}_2$ is as in \eqref{eq:034prop5.2}, and $\mathcal{A}_3$ and $\mathcal{A}_4$ are as in \eqref{eq:051prop5.2}. Therefore, \eqref{eq:005prop5.2} follows from \eqref{eq:066prop5.2}, and the entire proof of Proposition \ref{prop-seeds-and-conc} is completed.


\part{Applications to metric embeddings} \label{part3}

\section{Proof of Corollary \ref*{intro-cor1}} \label{sec12}

The goal of this section is to obtain optimal estimates on the bi-Lipschitz distortion of a uniformly random regular graph into an arbitrary finite metric space. We start with the proof of Corollary~\ref{intro-cor1}.

\begin{proof}[Proof of Corollary \ref{intro-cor1}]
Let $G\in G(n,d)$, set $D:=c_{\mathcal{M}}(G)$, and fix a map $f\colon G \to M$ that saturates the distortion, that is, it satisfies $s\, \dist_G(v,u) \leqslant \varrho\big(f(v),f(u)\big) \leqslant sD\, \dist_G(v,u)$ for all $v,u\in [n]$, where $s>0$ is a scaling factor. Then, as a deterministic consequence of $d$-regularity,
\[  s\log n \lesssim_d \frac{1}{n^2}\sum_{v,u\in [n]}\varrho\big(f(v),f(u)\big) \ \ \ \text{ and } \ \ \
 \frac{1}{|E_G|}\sum_{\{v,u\}\in E_G} \varrho\big(f(v),f(u)\big) \leqslant sD. \]
However, Theorem \ref{random-poincare} asserts that with probability $1-O_d\big(\frac{1}{n^\tau}\big)$ it holds for any function, and in particular for the selected $f$,
\[  \frac{1}{n^2}\sum_{v,u\in [n]} \varrho\big(f(v),f(u)\big) \leqslant
C_{d} \min\big\{\log n, \log\log N\big\}\, \frac{1}{|E_G|}\sum_{\{v,u\}\in E_G} \varrho\big(f(v),f(u)\big). \]
Combining the above displays yields the result.
\end{proof}

The following proposition complements Corollary \ref{intro-cor1} and shows that the lower bound of the bi-Lipschitz distortion obtained in \eqref{intro-e4} is actually optimal, up to a multiplicative factor, for a wide range of the parameters $n,N$.

\begin{proposition} \label{prop-optimality}
Let $k\geqslant n> d\geqslant 3$ be integers, and set
\begin{equation} \label{opt-e1}
D:= \frac{3\,\log n}{\min\{ \log n, \log\log k\}}.
\end{equation}
Then there exists a metric space $\mathcal{M}=(M,\varrho)$ with
\begin{equation} \label{opt-e2}
|M| = \exp\Big( \Theta\big( \log^2 n\cdot \log\log n \cdot \min\{ n, \log k\}\big)\Big)
\end{equation}
such that for any connected $G\in G(n,d)$ with $\diam(G)\leqslant 500 \log n$, we have $c_{\mathcal{M}}(G) \leqslant D$. Consequently, since $\min\big\{\log n, \log\log\big(|M|\big)\big\} =\Theta\big( \min\{\log n, \log\log k\}\big)$, by Corollaries~\ref{intro-cor1} and~\ref{Friedman}, there exists $\tau'=\tau'(d)>0$ so that
\begin{equation} \label{opt-e4}
\mathbb{P}_{G(n,d)}\bigg[ G\colon c_{\mathcal{M}}(G) = \Theta_d\Big( \frac{\log n}{\min\{ \log n, \log\log (|M|)\}}\Big) \bigg]
\geqslant 1- O_d\Big(\frac{1}{n^{\tau'}}\Big).
\end{equation}
\end{proposition}

The proof of Proposition \ref{prop-optimality} is a based on the following result which is a slight modification of the argument in Matou\v{s}ek's work \cite[Theorem 4]{Ma92}.

\begin{theorem}[Matou\v{s}ek] \label{thm-Matousek}
Let $n, \Delta\geqslant 2$ be integers, and let $D \geqslant 1$ be a distortion parameter. Then there exists a metric space $\mathcal{M}= (M,\varrho)$ with
\begin{equation} \label{mat-e1}
|M| = \exp\Big( \Theta\big( n^{3/D} \cdot \log\Delta \cdot \log^2 n \big)\Big)
\end{equation}
such that every connected graph on $n$ vertices and diameter at most $\Delta$, embeds into $\mathcal{M}$ with distortion at most $D$.
\end{theorem}

For the reader's convenience, we give the proof of Theorem~\ref{thm-Matousek} in Appendix \ref{appendix-A}. We are ready to proceed to the proof of Proposition \ref{prop-optimality}.

\begin{proof}[Proof of Proposition \ref{prop-optimality}]
By Theorem \ref{thm-Matousek} applied for $\Delta:= \lfloor 500 \log n\rfloor$ and the distortion parameter $D$ selected in \eqref{opt-e1}, there exists a metric space $\mathcal{M}=(M,\varrho)$ of size
\[ \exp\Big( \Theta\big( n^{3/D} \cdot \log\Delta \cdot \log^2 n \big)\Big) =
\exp\Big( \Theta\big( \log^2 n\cdot \log\log n \cdot \min\{ n, \log k\}\big)\Big), \]
such that any connected $d$-regular graph $G$ on $[n]$ with $\diam(G)\leqslant 500\log n$, embeds into $\mathcal{M}$ with distortion at most $D$, as desired.
\end{proof}

\section{Proof of Corollary \ref*{intro-cor2}} \label{sec13}

Applying Theorem \ref{random-poincare} and Corollary \ref{Friedman} for $d=4$, we may select a constant $C\geqslant 1$ and an integer $n_0\geqslant 5$ so that for any\footnote{We have chosen to work with $4$-regular graphs because $4n$ is always even, and so, we have \eqref{opt-new-e1} at our disposal for any integer $n\geqslant n_0$.} integer $n\geqslant n_0$ and any metric space $\mathcal{M}=(M,\varrho)$ with $|M|\geqslant 3$,
\begin{equation} \label{opt-new-e1}
\mathbb{P}_{G(n,4)}\Big[ G\colon G \text{ is connected and } \gamma(G,\varrho)\leqslant C \min\big\{\log n, \log\log \big(|M|\big)\big\}\Big] \geqslant \frac12.
\end{equation}

Let $n\geqslant n_0$ be an integer, let $D\geqslant 1$, and fix a metric space $\mathcal{M}=(M,\varrho)$ so that every connected graph on $n$ vertices embeds into $\mathcal{M}$ with distortion at most $D$. Then, by~\eqref{opt-new-e1}, there exists a connected graph $G\in G(n,4)$ such that $\gamma(G,\varrho)\leqslant C  \min\big\{\log n, \log\log \big(|M|\big)\big\}$ and $c_{\mathcal{M}}(G)\leqslant D$. Fix an embedding $e\colon [n]\to M$ and $s>0$ such that $s\,\dist_G(v,u)\leqslant \varrho\big(e(v),e(u)\big)\leqslant sD\, \dist_G(v,u)$ for all $v,u\in [n]$. Then, by the $4$-regularity of $G$,
\begin{align*}
s\log n \lesssim \frac{1}{n^2}\sum_{v,u\in [n]} \varrho\big(e(v),e(u)\big)
& \leqslant \gamma(G,\varrho) \cdot \frac{1}{|E_G|}\sum_{\{v,u\}\in E_G} \varrho\big(e(v),e(u)\big) \\
& \leqslant sD \, \gamma(G,\varrho) \leqslant CsD\, \log\log \big(|M|\big),
\end{align*}
that yields that $|M|\geqslant \exp\big( n^{c/D}\big)$, for some universal constant $c>0$. The proof is completed by appropriately selecting the implied constant in \eqref{intro-e6} to cover the remaining cases $n\in \{2,\dots, n_0\}$.


\appendix

\numberwithin{equation}{section}

\section{Proof of Theorem \ref*{thm-Matousek}} \label{appendix-A}

As we have already pointed out, the proof of Theorem \ref{thm-Matousek} is a slight modification of the argument in \cite[Theorem 4]{Ma92} which, in turn, is based on the following lemma.

\begin{lemma}[Johnson--Lindenstrauss--Schechtman \textcolor{black}{\cite{JLS87}}] \label{appendix-l1}
Let $\mathcal{N}=(N,\rho)$ be an arbitrary finite metric space with $n:=|N|\geqslant 2$ points, and set $m=m(\mathcal{N}):= \lfloor \log_2 n\rfloor +1$; moreover, for every $k\in \{0,\dots,m-1\}$, let $\boldsymbol{A}_k$ be a random subset of $N$ uniformly distributed on $\binom{N}{2^k}$. Finally, let $x,y\in N$ be distinct. Then, for any $\alpha\in (0,1/3)$, there is $k=k(\alpha,x,y)\in \{0,\dots,m-1\}$~so~that
\begin{equation} \label{appendix-e2}
\mathbb{P}\big[ |\rho(x,\boldsymbol{A}_k)-\rho(y,\boldsymbol{A}_k)|\geqslant \alpha\, \rho(x,y)\big] \geqslant c_1\, n^{-3\alpha},
\end{equation}
where $c_1>0$ is a universal constant.
\end{lemma}

We are ready to proceed to the proof of Theorem \ref{thm-Matousek}.

\begin{proof}[Proof of Theorem \ref{thm-Matousek}]
We may clearly assume that $n$ and $D$ are both sufficiently large. Let $c_1$ be as in Lemma~\ref{appendix-l1}, and set $m:=\lfloor \log_2 n\rfloor +1$, $K:= m\, \big\lceil \frac{2}{c_1}\, n^{3/D}\, \log n\big\rceil$ and $r:=K/m$. Define the metric space $\mathcal{M}=(M,\varrho)$, where $M=\{0,\dots,\Delta\}^K$ and $\varrho$ is the metric induced by~$\ell_\infty^K$. We claim that every connected graph on $n$ vertices and diameter at most $\Delta$, embeds into $\mathcal{M}$ with distortion at most~$D$. Indeed, fix a connected graph $G$ on $[n]$ with $\diam(G)\leqslant \Delta$. Let $(\boldsymbol{A}_k^i)_{k\in \{0,\dots,m-1\}, i\in [r]}$ be a random vector with independent components, where for every $k\in \{0,\dots,m-1\}$ and every $i\in [r]$, $\boldsymbol{A}^i_k$ is uniformly distributed on $\binom{[n]}{2^k}$. Define the random function $\boldsymbol{f}\colon [n]\to M$ by
\[ \boldsymbol{f}(v):=\big(\dist_G(v,\boldsymbol{A}^i_k)\big)_{k\in \{0,\dots,m-1\}, i\in [r]}. \]
Note that, almost everywhere,  $\boldsymbol{f}$~is well-defined and $1$-Lipschitz. Moreover, for every distinct $v,u\in [n]$, by Lemma \ref{appendix-l1} applied for ``$\alpha=1/D$", setting $k:=k(\alpha,v,u)$, we have
\begin{align*}
\mathbb{P}\big[  |\dist_G(v,\boldsymbol{A}^i_k) - \dist_G(u,\boldsymbol{A}^i_k)| <\alpha  & \, \dist_G(v,u)
\text{ for all } i\in [r] \big] \leqslant \\
& (1-c_1 n^{-3/D})^r \leqslant \exp\Big( -c_1 n^{-3/D} \frac{K}{m} \Big) \leqslant \frac{1}{n^2},
\end{align*}
that further implies, by a union bound over all $\{v,u\}\in \binom{[n]}{2}$, that $\mathbb{P}\big[ \|\boldsymbol{f}^{-1}\|_{\mathrm{Lip}}> D\big] <1$. Thus, with positive probability, a realization of $\boldsymbol{f}$ certifies that $G$ embeds into $\mathcal{M}$ with distortion at most~$D$, as desired.
\end{proof}
We close this appendix with the following straightforward consequence of Theorem \ref{thm-Matousek}.
\begin{corollary}[Matou\v{s}ek] \label{appendix-cor}
There exists a universal constant $C>0$ with the following property. Let $n\geqslant 2$ be an integer, and let $1\leqslant D\leqslant \frac{\log n}{\log\log n}$ be a distortion parameter. Then there exists a metric space $\mathcal{M}= (M,\varrho)$ with
\begin{equation} \label{appendix-e1}
|M|\lesssim \exp\big( n^{C/D} \big),
\end{equation}
such that every connected graph on $n$ vertices embeds into $\mathcal{M}$ with distortion at most $D$.
\end{corollary}

\subsection*{Acknowledgments}

The authors would like to thank the anonymous referee for numerous comments, remarks and suggestions that helped us improve the exposition.

The research was supported in the framework of H.F.R.I call ``Basic research Financing (Horizontal support of all Sciences)" under the National Recovery and Resilience Plan ``Greece 2.0" funded by the European Union--NextGenerationEU (H.F.R.I. Project Number: 15866). The third named author (K.T.) is partially supported by the NSF grant DMS 2331037.

\end{document}